\documentclass[10pt,reqno]{amsart}
\usepackage{hyperref}
\usepackage{latexsym,amsmath}
\usepackage{enumerate}
\usepackage{amsfonts}
\usepackage{amssymb}
\usepackage{latexsym}
\usepackage{fixmath}
\usepackage[usenames,dvipsnames]{color}
\usepackage{mathdots}
\usepackage{nicematrix}

\usepackage{graphicx}
\usepackage{comment}

\usepackage[T1]{fontenc}
\usepackage{fourier}
\usepackage{bbm}
\usepackage{color}
\usepackage{fullpage}

\newcommand\unif{\mathrm{unif}}

\newcommand\bigzero{\makebox(0,0){\text{\huge0}}}

\newcommand\mycheck{\hat}

\numberwithin{equation}{section}

\newcommand{\trans}{\top}

\newcommand{\alg}{\fA}

\renewcommand{\vec}[1]{\boldsymbol{#1}}

\renewcommand{\subset}{\subseteq}

\newcommand\disteq{\,\sim\,}

\newcommand\THETA{\vec\theta}

\newcommand\CPC{Combinatorics, Probability and Computing}

\newcommand{\GG}{\mathbb G}

\newcommand\ALPHA{\vec\alpha}
\newcommand\BETA{\vec\beta}

\newcommand\CHI{{\vec\chi}}
\newcommand\DELTA{{\vec\Delta}}
\newcommand\GAMMA{{\vec\gamma}}

\newcommand\nix{\,\cdot\,}
\newcommand\dd{{\mathrm d}}
\newcommand\G{\vec G}

\newcommand\bemph[1]{{\bf\em #1}}

\newcommand\cA{\mathcal A}
\newcommand\cB{\mathcal B}
\newcommand\cC{\mathcal C}

\newcommand\cE{\mathcal E}

\newcommand\cI{\mathcal I}

\newcommand\cL{\mathcal L}

\newcommand\cP{\mathcal P}

\newcommand\cR{\mathcal R}
\newcommand\cS{\mathcal S}
\newcommand\cT{\mathcal T}

\newcommand\cX{\mathcal X}

\newcommand\cZ{\mathcal Z}
\newcommand\fA{\mathfrak A}
\newcommand\fB{\mathfrak B}

\newcommand\fD{\mathfrak D}
\newcommand\fE{\mathfrak E}
\newcommand\fF{\mathfrak F}

\newcommand\fI{\mathfrak I}

\newcommand\fL{\mathfrak L}
\newcommand\fM{\mathfrak M}

\newcommand\fO{\mathfrak O}
\newcommand\fP{\mathfrak P}

\newcommand\fR{\mathfrak R}
\newcommand\fS{\mathfrak S}
\newcommand\fT{\mathfrak T}

\newcommand\fX{\mathfrak X}

\newcommand\fb{\mathfrak b}

\newcommand\fd{\mathfrak d}

\newcommand\ff{\mathfrak f}

\newcommand\fk{\mathfrak k}

\newcommand\fo{\mathfrak o}
\newcommand\fp{\mathfrak p}

\newcommand\vA{\vec A}
\newcommand\vB{\vec B}
\newcommand\vC{\vec C}

\newcommand\vG{\vec G}

\newcommand\vM{\vec M}
\newcommand\vN{\vec N}

\newcommand\vQ{\vec Q}
\newcommand\vR{\vec R}

\newcommand\vX{\vec X}
\newcommand\vY{\vec Y}
\newcommand\vZ{\vec Z}

\newcommand\vd{\vec d}

\newcommand\vi{\vec i}

\newcommand\vk{\vec k}

\newcommand\vm{\vec m}

\newcommand\vr{\vec r}

\newcommand\vx{\vec x}
\newcommand\vy{\vec y}

\newcommand\eul{\mathrm{e}}
\newcommand\eps{\varepsilon}

\renewcommand\AA{\mathbb{A}}
\newcommand\BB{\mathbb{B}}
\newcommand\QQ{\mathbb{Q}}
\newcommand\ZZ{\mathbb{Z}}
\newcommand\FF{\mathbb{F}}

\newcommand\NN{\mathbb{N}}

\newcommand\Erw{\mathbb{E}}
\newcommand\ex{\Erw}
\newcommand{\vecone}{\mathbb{1}}

\newcommand{\Po}{{\rm Po}}

\newcommand\dTV{d_{\mathrm{TV}}}

\newcommand\bc[1]{\left({#1}\right)}
\newcommand\cbc[1]{\left\{{#1}\right\}}
\newcommand\bcfr[2]{\bc{\frac{#1}{#2}}}

\newcommand\brk[1]{\left\lbrack{#1}\right\rbrack}

\newcommand\norm[1]{\left\|{#1}\right\|}
\newcommand\abs[1]{\left|{#1}\right|}

\newcommand\RR{\mathbb{R}}

\newcommand{\Whp}{W.h.p.}
\newcommand{\whp}{w.h.p.}

\newcommand{\Komlos}{Koml\'os}

\newcommand\pr{\mathbb{P}} 
\renewcommand\Pr{\pr} 
\newcommand\Lem{Lemma}
\newcommand\Prop{Proposition}
\newcommand\Thm{Theorem}

\newcommand\Cor{Corollary}
\newcommand\Sec{Section}

\newcommand\id{\mathrm{id}}

\newcommand{\var}{\mathbb{X}}

\newtheorem{definition}{Definition}[section]
\newtheorem{claim}[definition]{Claim}
\newtheorem{example}[definition]{Example}

\newtheorem{theorem}[definition]{Theorem}
\newtheorem{lemma}[definition]{Lemma}
\newtheorem{proposition}[definition]{Proposition}
\newtheorem{corollary}[definition]{Corollary}

\newtheorem{fact}[definition]{Fact}

\newtheoremstyle{case}{}{}{}{}{}{:}{ }{}
\theoremstyle{case}

\DeclareMathOperator{\vol}{vol}

\DeclareMathOperator{\nul}{nul}
\DeclareMathOperator{\rank}{rk}

\newcommand{\rk}{\rank}

\newcommand{\supp}{{\mathrm{supp}}}

\newcommand\A{\vA}

\def\B{{\mathcal B}}

\def\Z{{\mathcal Z}}

\def\pr{{\mathbb P}}

\newcommand{\remove}[1]{}

\newcommand{\be}{\begin{equation}}
\newcommand{\bel}[1]{\begin{equation}\lab{#1}\ }
\newcommand{\ee}{\end{equation}}
\newcommand{\bea}{\begin{eqnarray}}
\newcommand{\eea}{\end{eqnarray}}
\newcommand{\bean}{\begin{eqnarray*}}
\newcommand{\eean}{\end{eqnarray*}}

\begin{document}
\title{The full rank condition for sparse random matrices}
\author{Amin Coja-Oghlan, Pu Gao, Max Hahn-Klimroth, Joon Lee, Noela M\"{u}ller, Maurice Rolvien}
\thanks{Amin Coja-Oghlan is supported by DFG CO 646/3 and DFG CO 646/5. Max Hahn-Klimroth is supported by DFG CO 646/5. Noela Müller is supported by NWO Gravitation grant NETWORKS-024.002.003.}
\address{Amin Coja-Oghlan, {\tt amin.coja-oghlan@tu-dortmund.de}, TU Dortmund, Faculty of Computer Science, 12 Otto-Hahn-St, Dortmund 44227, Germany.}
\address{Pu Gao, {\tt p3gao@uwaterloo.ca}, Department of Combinatorics and Optimization, University of Waterloo, Canada.}
\address{Max Hahn-Klimroth, {\tt maximilian.hahnklimroth@tu-dortmund.de }, TU Dortmund, Faculty of Computer Science, 12 Otto-Hahn-St, Dortmund 44227, Germany.}
\address{Joon Lee, {\tt joon.lee@tu-dortmund.de}, TU Dortmund, Faculty of Computer Science, 12 Otto-Hahn-St, Dortmund 44227, Germany.}
\address{Noela M\"uller, {\tt n.s.muller@tue.nl}, Eindhoven University of Technology, Department of Mathematics and Computer Science, MetaForum MF 4.084, 5600 MB Eindhoven, the Netherlands.}
\address{Maurice Rolvien, {\tt maurice.rolvien@tu-dortmund.de}, TU Dortmund, Faculty of Computer Science, 12 Otto-Hahn-St, Dortmund 44227, Germany.}
\maketitle
\begin{abstract}
We derive a sufficient condition for a sparse random matrix with given numbers of non-zero entries in the rows and columns having full row rank.
The result covers both matrices over finite fields with independent non-zero entries and $\{0,1\}$-matrices over the rationals.
The sufficient condition is generally necessary as well.
%
\hfill MSc: 60B20, 15B52
\end{abstract}
 
\section{Introduction}\label{Sec_intro}

\subsection{Background and motivation}\label{sec_motivation}
Few subjects in combinatorics have had as profound an impact on other disciplines as combinatorial random matrix theory.
Prominent applications include powerful error correcting codes called low-density parity check codes~\cite{RichardsonUrbanke}, data compression~\cite{AMc,Maneva} and hashing~\cite{Dietzfelbinger}.
Needless to mention, random combinatorial matrices are of keen interest to statistical physicists, too~\cite{MM}.
It therefore comes as no surprise that the subject has played a central role in probabilistic combinatorics since the early days~\cite{Kolchin,Kolchin1,Kolchin2,Komlos}.
The current state of affairs is that the theory of dense random matrices is significantly more advanced than that of sparse ones with a bounded average number of non-zero entries per row or column~\cite{Vu,Vu2}.
This is in part because concentration techniques apply more easily in the dense case.
Another reason is that the study of sparse random matrices is closely tied to the investigation of satisfiability thresholds of random constraint satisfaction problems, an area where many fundamental questions still await a satisfactory solution~\cite{ANP}.

Perhaps the most basic question to be asked about any random matrix model is whether the resulting matrix will likely have full rank.
This paper contributes a succinct sufficient condition that covers a broad range of sparse random matrix models.
As we will see, the condition is essentially necessary as well.
The main result can be seen as a satisfiability threshold theorem as the full rank property is equivalent to a random linear system of equations possessing a solution \whp\
This formulation generalises a number of prior results such as the satisfiability threshold theorem for the random $k$-XORSAT problem, one of the most intensely studied random constraint satisfaction problems (e.g.,~\cite{AchlioptasMolloy,Dietzfelbinger,DuboisMandler,Ibrahimi,PittelSorkin}).
In addition, the main theorem covers other important random matrix models, including those that low-density parity check codes rely on~\cite{RichardsonUrbanke}.

The classical approach to tackling the full rank problem is the second moment method~\cite{AM,ANP}.
This technique was pioneered in the seminal work on the $k$-XORSAT threshold of Dubois and Mandler~\cite{DuboisMandler}.
Characteristic of this approach is the emergence of complicated analytic optimisation problems that encode entropy-probability trade-offs resulting from large deviations problems.
Tackling these optimisation problems turns out to be rather challenging even in relatively simple special cases such as random $k$-XORSAT, as witnessed by the intricate calculations that Pittel and Sorkin~\cite{PittelSorkin} and Goerdt and Falke~\cite{GoerdtFalke} had to go through.
For the general model that we investigate here this proof technique thus appears futile.

We therefore pursue a totally different proof strategy, largely inspired by ideas from spin glass theory~\cite{MM,MRTZ}.
In statistical physics jargon, the second moment method constitutes an ``annealed'' computation.
This means that we effectively average over all random matrices, including atypical specimens apt to boost the average.
By contrast, the present work relies on a ``quenched'' strategy based on a coupling argument that implicitly discards such pathological events.
In effect, we will show that a truncated moment calculation confined to certain benign ``equitable'' solutions suffices to determine the satisfiability threshold.
This part of the proof is an extension of prior work of (some of) the authors on the normalised rank and variations on the random $k$-XORSAT problem~\cite{Ayre,Maurice}.
In addition, to actually compute the truncated second moment we need to determine the precise expected number of equitable solutions.
To this end, we devise a new proof ingredient that combines local limit theorem techniques with algebraic ideas, particularly the combinatorial analysis of certain integer lattices.
This technique can be seen as a generalisation of an argument of Huang~\cite{Huang} for the study of adjacency matrices of $d$-regular random graphs.

Let us proceed to present the main results of the paper.
The first theorem deals with random matrices over finite fields.
As an application we obtain a result on sparse $\{0,1\}$-matrices over the rationals.

\subsection{Results}\label{sec_results_finite}
We work with the comprehensive random matrix model from~\cite{Maurice}.
Hence, let $\vd \geq 0$, $\vk \geq 3$ be independent integer-valued random variables such that $\Erw[\vd^{2+\eta}] + \Erw\brk{\vk^{2+\eta}} < \infty$ for an arbitrarily small $\eta > 0$.
Let $(\vd_i,\vk_i)_{i\geq 1}$ be independent copies of $(\vd, \vk)$ and set $d = \Erw[\vd], k = \Erw[\vk]$.
Moreover, let $\fd$ and $\fk$ be the greatest common divisors of the support of $\vd$ and $\vk$, respectively. 
Further, let $n>0$ be an integer divisible by $\fk$ and let $\vm$ be a Poisson variable with mean $dn/k$, independent of $(\vd_i,\vk_i)_i$.
Routine arguments reveal that the event
\begin{align}\label{deg_sums}
    \sum_{i=1}^n \vd_i = \sum_{j=1}^{\vm} \vk_j
\end{align}
occurs with probability at least $\Omega(n^{-1/2})$~\cite[\Prop~1.7]{Maurice}.
Given \eqref{deg_sums} let $\GG=\GG_n(\vd,\vk)$ be a simple random bipartite graph on a set $\cbc{a_1 \ldots, a_{\vm}}$ of {\em check nodes} and a set $\cbc{x_1,\ldots ,x_n}$ of {\em variable nodes} such that the degree of $a_i$ equals $\vk_i$ and the degree of $x_j$ equals $\vd_j$ for all $i, j$. 
Following coding theory jargon, we refer to $\GG$ as the {\em Tanner graph}.
The edges of $\GG$ are going to mark the positions of the non-zero entries of the random matrix.
The entries themselves will depend on whether we deal with a finite field or the rationals.

\subsubsection{Finite fields}\label{sec_results_finite}
Suppose that $q \geq 2$ is a prime power, let $\FF_q$ signify the field with $q$ elements and let $\vec\chi$ be a random variable that takes values in the set $\FF_q^\ast = \FF_q \setminus \{0\}$ of units of $\FF_q$.
Moreover, let $(\vec\chi_{i,j})_{i,j\geq1}$ be copies of $\vec\chi$, mutually independent and independent of the $\vd_i,\vk_i$, $\vm$ and $\GG$.
Finally, let $\AA=\AA_n(\vd,\vk,\vec\chi)$ be the $\vm \times n$-matrix with entries
\begin{align*}
    \AA_{i,j} = \vecone\cbc{a_ix_j \in E(\GG)} \cdot \CHI_{i,j}.
\end{align*}
Hence, the $i$-th row of $\AA$ contains $\vk_i$ non-zero entries and the $j$-th column contains $\vd_j$ non-zero entries.

The following theorem provides a sufficient condition for $\AA$ having full row rank.
The condition comes in terms of the probability generating functions $D(x)$ and $K(x)$ of $\vd$ and $\vk$.
Since $\Erw[\vd^2]+\Erw[\vk^2]<\infty$, we may define
\begin{align}\label{eqBFE}
	\Phi:[0,1]&\to\RR,&z&\mapsto D\left(1-K'(z)/k\right)-\frac{d}{k}\bc{1-K(z)-(1-z)K'(z)}.
\end{align}

\begin{theorem}\label{thm_main}
	If $q$ and $\fd$ are coprime and 
	\begin{align}\label{eqmain}
		\Phi(z)&<\Phi(0)&\mbox{for all $0<z\leq1$},
	\end{align}
	then $\AA$ has full row rank over $\FF_q$ \whp
\end{theorem}

\noindent
Observe that the function $\Phi$ does not depend on $q$.
Hence, neither does \eqref{eqmain}.

The sufficient condition \eqref{eqmain} is generally necessary, too.
Indeed, \cite[\Thm~1.1]{Maurice} determines the likely value of the {\em normalised} rank of $\AA$:
\begin{align}\label{eqMaurice}
	\frac{\rk(\AA)}{n} &\stackrel{\Pr}{\longrightarrow} 1-\max_{z\in[0,1]}\Phi(z)&&\mbox{as }n\to\infty.
\end{align}
Since $\vk\geq3$, the definition \eqref{eqBFE} ensures that $\Phi(0)=1-d/k$ and thus $n\Phi(0)\sim n-\vm$ \whp\ 
Hence, \eqref{eqMaurice} implies that $\rk(\AA)\leq\vm-\Omega(n)$ \whp\ unless $\Phi(z)$ attains its maximum at $z=0$.
In other words, $\AA$ has full row rank {\em only if} $\Phi(z)\leq\Phi(0)$ for all $0<z\leq1$.
Indeed, in \Sec~\ref{sec_examples} we will discover examples that requite a strict inequality as in \eqref{eqmain}.
The condition that $q$ and $\fd$ be coprime is generally necessary as well, as we will see in Example~\ref{ex_zerorowsums} below.

Let us emphasise that \eqref{eqMaurice} does not guarantee that $\AA$ has full row rank \whp\ even if \eqref{eqmain} is satisfied.
Rather due to the normalisation on the l.h.s.\ \eqref{eqMaurice} only implies the much weaker statement $\rk(\AA)=\vm-o(n)$ \whp\
Hence, in the case that \eqref{eqmain} is satisfied, \Thm~\ref{thm_main} improves over the asymptotic estimate~\eqref{eqMaurice} rather substantially.
Unsurprisingly, this stronger result also requires a more delicate proof strategy.

\subsubsection{Zero-one matrices over the rationals}\label{sec_results_rational}
Apart from matrices over finite fields, the rational rank of sparse random $\{0,1\}$-matrices has received a great deal of attention~\cite{Vu,Vu2}.
The random graph $\GG$ naturally induces a $\{0,1\}$-matrix, namely the $\vm\times n$-biadjacency matrix $\BB=\BB(\GG)$.
Explicitly, $\BB_{ij}=\vecone\{a_ix_j\in E(\GG)\}$.
As an application of \Thm~\ref{thm_main} we obtain the following result.

\begin{corollary}\label{thm_Q}
	If \eqref{eqmain} is satisfied then the random matrix $\BB$ has full row rank over $\QQ$ \whp\
\end{corollary}

Since \eqref{eqMaurice} holds for random matrices over the rationals as well, \Cor~\ref{thm_Q} is optimal to the extent that $\BB$ fails to have full row rank \whp\ if $\max_{x\in[0,1]}\Phi(x)>\Phi(0)$.
Moreover, in Example~\ref{ex_identical} we will see that $\BB$ does not generally have full rank \whp\ unless $x=0$ is the {unique} maximiser of $\Phi$.

\subsection{Examples}\label{sec_examples}
To illustrate the power of \Thm~\ref{thm_main} and \Cor~\ref{thm_Q} we consider a few instructive special cases of distributions $\vd,\vk,\vec\chi$.

\begin{figure}
	\centering
	\includegraphics[width=0.33\textwidth]{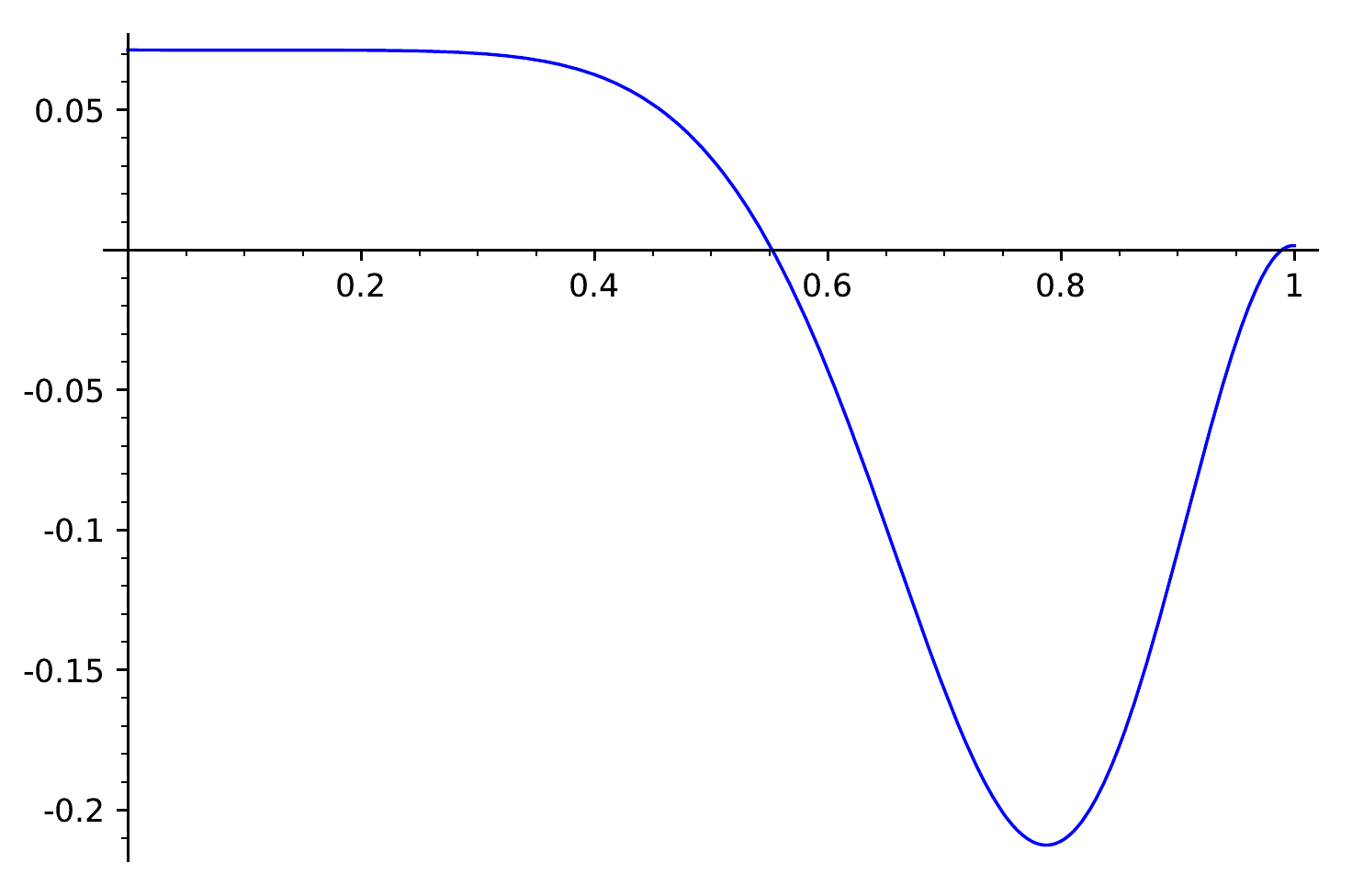}	
	\includegraphics[width=0.33\textwidth]{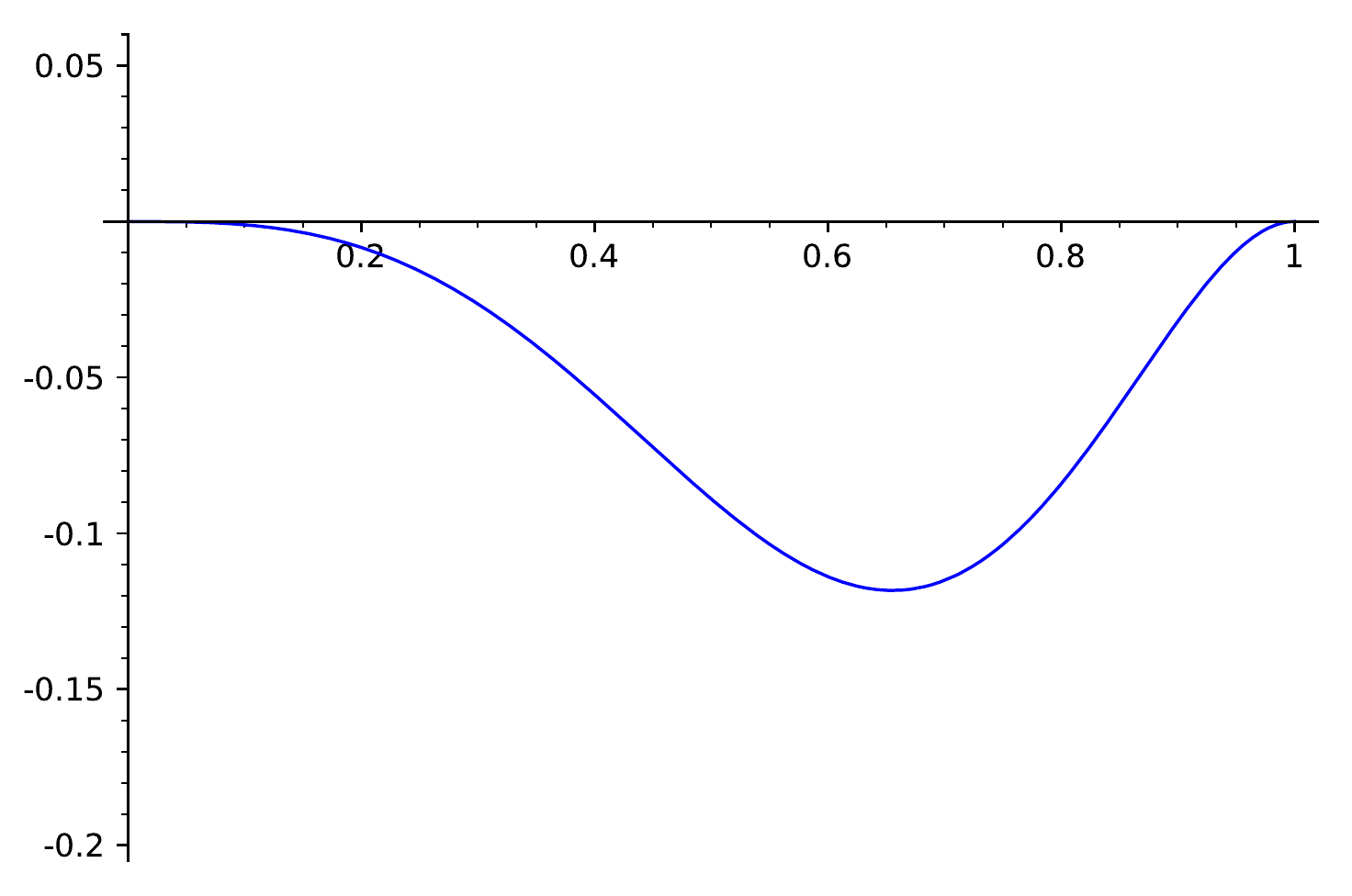}

	\caption{Left: Example \ref{ex_XOR} with $D(z)=\exp(6.5(z-1))$ and $K(z)=z^7$. Middle: Example \ref{ex_identical} with $D(z)=K(z)=(z^3+z^4)/2$.}
	\label{fig:ex1-4} 
\end{figure}

\begin{figure}
	\centering
	\includegraphics[width=0.33\textwidth]{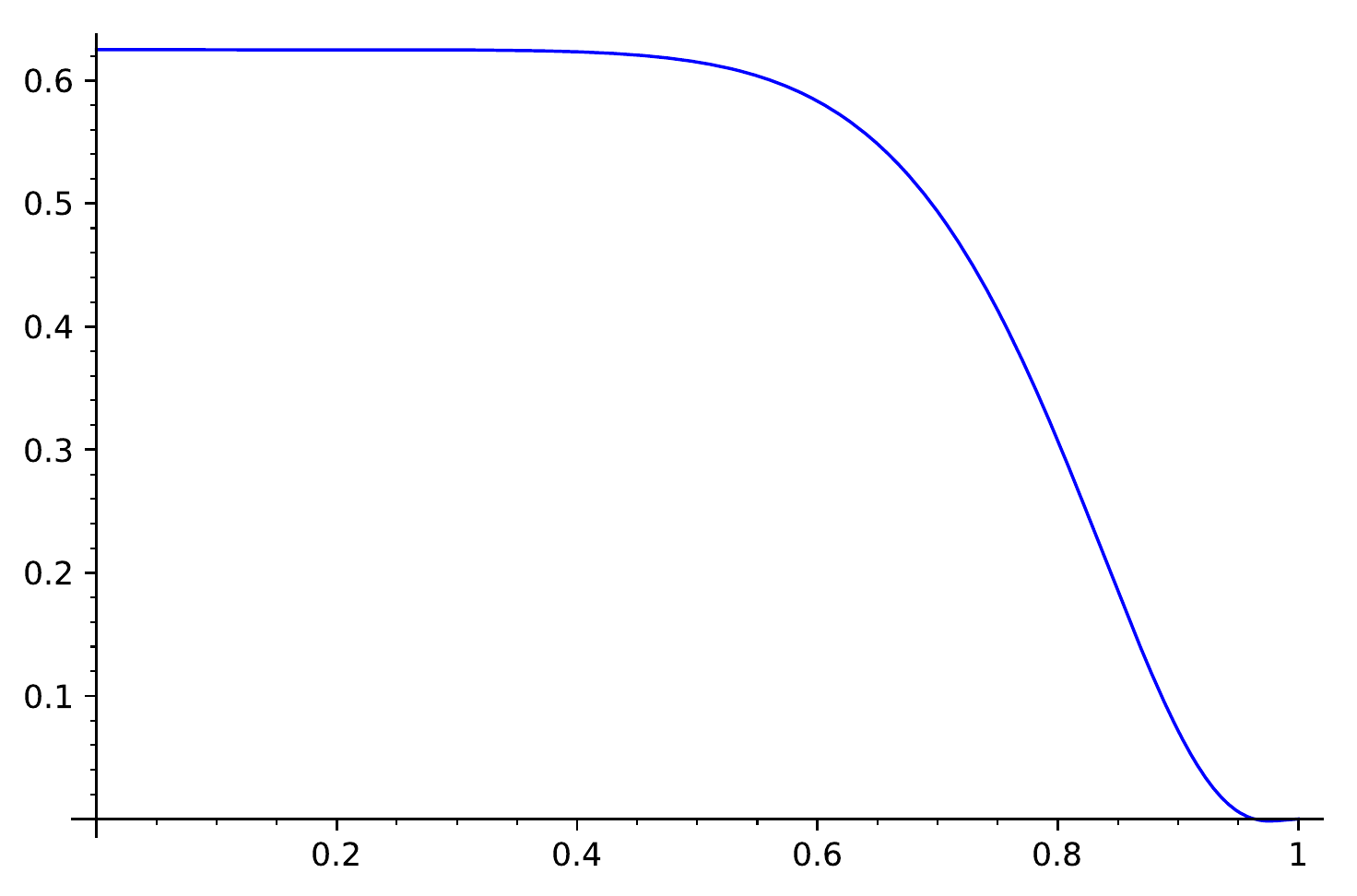}
	\includegraphics[width=0.33\textwidth]{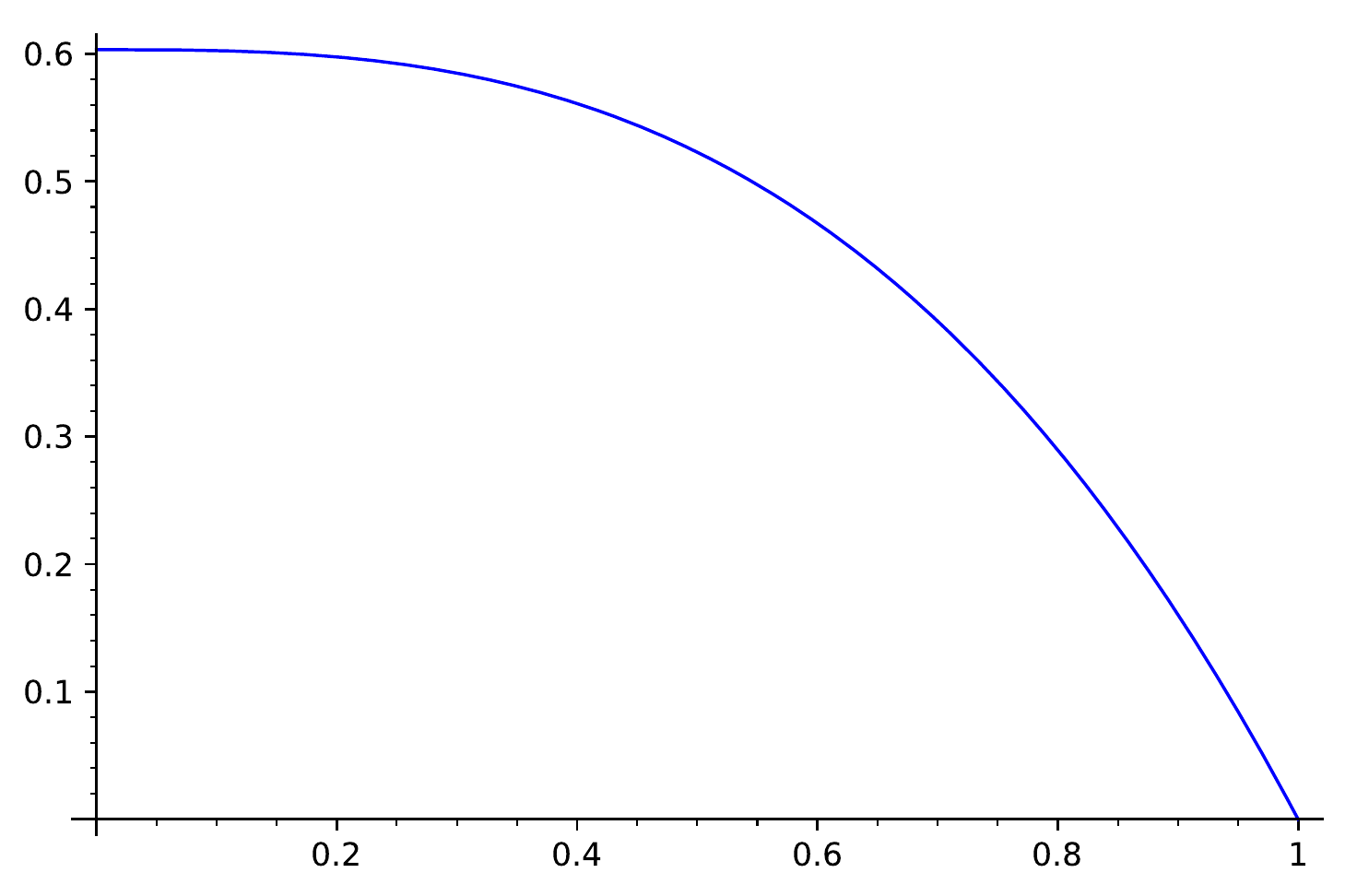}	
	\caption{Left: Example \ref{ex_fixed} with $D(z)=z^{3}, K(z)=z^{8}$. Right: Example \ref{ex_power} with $D(z) = \sum_{\ell = 1}^{\infty} \zeta(3.5)^{-1} z^{\ell} \ell^{-3.5}$ and $K(x) = x^3$.}
	\label{fig:ex1-6} 
\end{figure}

\begin{example}[random $k$-XORSAT]\upshape\label{ex_XOR}
	In random $k$-XORSAT we are handed a number of independent random constraints $c_i$ of the type
	\begin{align}\label{eqXOR}
			c_i=y_{i1}\ \mathrm{XOR}\ \cdots\ \mathrm{XOR}\ y_{ik},
		\end{align}
	where each $y_{ij}$ is either one of $n$ available Boolean variables $x_1,\ldots,x_n$ or a negation $\neg x_1,\ldots,\neg x_n$.
	The obvious question is to determine the satisfiability threshold, i.e., the maximum number of random constraints can be satisfied simultaneously \whp\

	Because Boolean XOR boils down to addition over $\FF_2$, this problem can be rephrased as the full rank problem for the random matrix $\AA$ with $q=2$, $\vk=k$ fixed to a deterministic value and $\vd\sim\Po(d)$ for a parameter $d>0$.
	To elaborate, because the constraints $c_i$ are drawn uniformly and independently, we can think of each as tossing $k$ balls randomly into $n$ bins that represent $x_1,\ldots,x_n$.
	If there are $\vm\sim\Po(dn/k)$ constraints $c_i$, the joint distribution of the variable degrees coincides with the distribution of $(\vd_1,\ldots,\vd_n)$ subject to the condition \eqref{deg_sums}.
	Furthermore, the random negation patterns of the constraints \eqref{eqXOR} amount to choosing a random right-hand side vector $\vy$ for which we are to solve $\AA x=\vy$.

	Since the generating functions of $\vd,\vk$ work out to be $D(z)=\exp(d(z-1))$ and $K(z)=z^k$, we obtain
	\begin{align*}
		\Phi_{d,k}(z)=\exp(-dz^{k-1})-\frac dk\bc{1-kz^{k-1}+(k-1)z^k}.
	\end{align*}
	Thus, \Thm~\ref{thm_main} implies that for a given $k\geq3$ the threshold of $d$ up to which random $k$-XORSAT is satisfiable \whp\ equals the largest $d$ such that
	\begin{align}\label{eqXORex}
		\Phi_{d,k}(z)<\Phi_{d,k}(0)=1-d/k\qquad\mbox{for all }0<z\leq1.
	\end{align}
	A few lines of calculus verify that \eqref{eqXORex} matches the formulas for the $k$-XORSAT threshold derived by combinatorial methods tailored to this specific case~\cite{Dietzfelbinger,DuboisMandler,PittelSorkin,MRTZ}.
	\Thm~\ref{thm_main} also encompasses the generalisations to other finite fields $\FF_q$ from~\cite{Ayre,GoerdtFalke}.
\end{example}

\begin{example}[identical distributions]\upshape\label{ex_identical}
	An interesting scenario arises when $\vd,\vk$ are identically distributed.
	For example, suppose that $\pr[\vd=3]=\pr[\vd=4]=\pr[\vk=3]=\pr[\vk=4]=1/2$.
	Thus, $D(z)=K(z)=(z^3+z^4)/2$ and
\begin{align*}
		\Phi(z)&=\frac{256z^{12}+768z^{11}+864z^{10}-1808z^9-4959z^8-3780z^7+6111z^6+10584z^5-3234z^4-4802z^3}{4802}. 
\end{align*}
	This function attains two identical maxima, namely $\Phi(0)=\Phi(1)=0$.
	Since the degrees $\vk_i,\vd_i$ are chosen independently subject only to \eqref{deg_sums}, the probability that $\AA$ has more rows than columns works out to be $1/2+o(1)$.
	As a consequence, $\AA$ cannot have full row rank \whp\
	This example shows that the condition that $0$ be the {\em unique} maximiser of $\Phi(x)$ is generally necessary $\AA$ to ensure full row rank.
	The same applies to the rational rank of $\BB$.
\end{example}

\begin{example}[fixed $\vd,\vk$]\upshape\label{ex_fixed}
	Suppose that both $\vd=d,\vk=k\geq3$ are constants rather than genuinely random.
	Then
	\begin{align*}
		\Phi(z)&=\left(1 - z^{k-1}\right)^{d}-\frac dk\bc{1- kz^{k-1}+(k-1)z^{k}}.
	\end{align*}
	Clearly, $\AA$ cannot have full row rank unless $d\leq k$, while \Thm~\ref{thm_main} implies that $\AA$ has full row rank \whp\ if $d<k$.
	This result was previously established via the second moment method~\cite{MillerCohen}.
	But in the critical case $d=k$ the function $\Phi(z)$ attains its identical maxima at $z=0$ and $z=1$.
	Specifically, $0=\Phi(0)=\Phi(1)>\Phi(z)$ for all $0<z<1$.
	Hence, \Thm~\ref{thm_main} does not cover this special case. 
	Nonetheless, Huang~\cite{Huang} proved that the random $\{0,1\}$-matrix $\BB$ has full rational rank \whp\
	The proof is based on a delicate moment computation in combination with a precise local expansion around the equitable solutions.
\end{example}

\begin{example}[power laws]\upshape\label{ex_power}
	Let $\pr(\vd = \ell ) \propto \ell^{-\alpha}$ for some $\alpha > 3$ and $ \vk = k \geq 3$. 
	Thus,
	\begin{align*}
		D(z)& = \frac1{\zeta(\alpha)}\sum_{\ell = 1}^{\infty}\frac{z^\ell}{\ell^{\alpha}},&
		K(z) &= z^k,&
		\Phi(z) = D\bc{1-z^{k-1}} - \frac{\zeta^{-1}(\alpha)\zeta(\alpha-1)}{k} \bc{1 - k z^{k-1}+ (k-1)z^{k} }.
	\end{align*}
	Since
	\begin{align*}
		\Phi'(z) = -(k-1)z^{k-2}D'\bc{1-z^{k-1}}+ \frac{\zeta^{-1}(\alpha)\zeta(\alpha-1)}{k} \bc{k(k-1) (z^{k-1}- z^{k-2}) } < 0,
	\end{align*}
	the function $\Phi(z)$ is strictly decreasing on $(0,1)$.
	Therefore, \eqref{eqmain} is satisfied.
\end{example}

\begin{example}[zero row sums]\upshape\label{ex_zerorowsums}
	\Thm~\ref{thm_main} requires the assumption that $q$ and the g.c.d.\ $\fd$ of the support of $\vd$ be coprime.
	This assumption is indeed necessary.
	To see this, consider the case that $q=2$, $\vec\chi=1$, $\vd=4$ and $\vk=8$ deterministically.
	Then the rows of $\AA$ always sum to zero.
	Hence, $\AA$ cannot have full row rank.
\end{example}

\section{Overview}\label{sec_outline}

\noindent
In contrast to much of the prior work on the rank problem, random $k$-XORSAT and random constraint satisfaction problems generally, the proofs of the main results do not rely on an ``annealed'' second moment computation.
Such arguments appear to be far too susceptible to large deviations effects to extend to as general a random matrix model as we deal with here.
Instead, we proceed by way of a ``quenched'' argument that enables us to discard pathological events.
As a result, it suffices to carry out the moment calculation in the particularly benign case of ``equitable'' solutions.

This proof strategy draws on but substantially generalises tools that were developed towards the approximate rank formula~\eqref{eqMaurice} and variations on random $k$-XORSAT~\cite{Ayre,Maurice}.
In addition, to actually prove that $\AA$ has full rank with {\em high} probability we will need to carry out a meticulous, asymptotically exact calculation of the expected number of equitable solutions.
A key element of this analysis will be a delicate analysis of the lattices generated by certain integer vectors that encode conceivable equitable solutions.
This part of the proof, which generalises a part of Huang's argument for the adjacency matrices of random $d$-regular graphs~\cite{Huang}, combines local limit techniques with a whiff of linear algebra.

To describe the proof strategy in detail let us first explore the ``annealed'' path, discover its pitfalls and then apply the lessons learned to develop a workable ``quenched'' strategy.
The bulk of the proof deals with the random matrix model from \Sec~\ref{sec_results_finite} over the finite field $\FF_q$; the rational case from \Cor~\ref{thm_Q} comes out as an easy consequence.

In order to reduce fluctuations we are going to condition on the $\sigma$-algebra $\fA$ generated by $\vm,(\vk_i)_{i\geq1},(\vd_i)_{i\geq1}$ and by the numbers $\vm(\chi_1,\ldots,\chi_\ell)$ of checks of degree $\ell\geq3$ with coefficients $\chi_1,\ldots,\chi_\ell\in\FF_q^*$.
We write $\pr_{\fA}=\pr\brk{\nix\mid\fA}$ and $\ex_{\fA}=\ex\brk{\nix\mid\fA}$ for brevity.

\subsection{Moments and deviations}\label{sec_grand_outline}
We already alluded to how the full rank problem for the random matrix $\AA$ over $\FF_q$ can be viewed as a random constraint satisfaction problem.
Indeed, suppose we draw a right-hand side vector $\vy\in\FF_q^{\vm}$ independently of $\AA$.
Then $\AA$ has full row rank \whp\ iff the random linear system $\AA x=\vy$ admits a solution \whp\
For if $\rk\AA<\vm$, then the image $\AA \FF_q^n$ is a proper subspace of $\FF_q^{\vm}$ and thus the random linear system $\AA x=\vy$ has a solution with probability at most $1-1/q$.
Naturally, the random linear system is nothing but a random constraint satisfaction problem with $\vm$ constraints and $n$ variables.

Over the past two decades the second moment method has emerged as the default approach to pinpointing satisfiability thresholds of random constraint satisfaction problems~\cite{AM,ANP}.
Indeed, one of the first success stories was the random $3$-XORSAT problem, which boils down directly to a full rank problem over $\FF_2$~\cite{DuboisMandler}.
In fact, as we saw in Example~\ref{ex_XOR}, to mimic $3$-XORSAT we just set $q=2$, $\vd=\Po(d)$ for some $d>0$ and $\vk=3$ deterministically.
In addition, draw $\vy\in\FF_2^{\vm}$ uniformly and independently of everything else.

We try the second moment method on the number $\vZ=\vZ(\AA,\vy)$ of solutions to $\AA x=\vy$ given $\fA$.
Since $\vy$ is independent of $\AA$, for any fixed vector $x\in\FF_2^n$ the event $\AA x=\vy$ has probability $2^{-\vm}$.
Consequently, 
\begin{align}\label{eqXOR1}
	\ex_{\fA}[\vZ]=2^{n-\vm}.
\end{align}
Hence, \eqref{eqXOR1} recovers the obvious condition that we cannot have more rows than columns.
Since $\vm\sim\Po(dn/3)$, \eqref{eqXOR1} boils down to $d<3$.

The second moment method now rests on the hope that we may be able to show that $\ex_{\fA}[\vZ^2]\sim\ex_{\fA}[\vZ]^2$.
Then Chebyshev's inequality would imply $\vZ\sim\ex_{\fA}[\vZ]$ \whp, and thus, in light of \eqref{eqXOR1}, that $\AA x=\vy$ has a solution \whp\

\begin{figure}
	\includegraphics[width=0.4\textwidth]{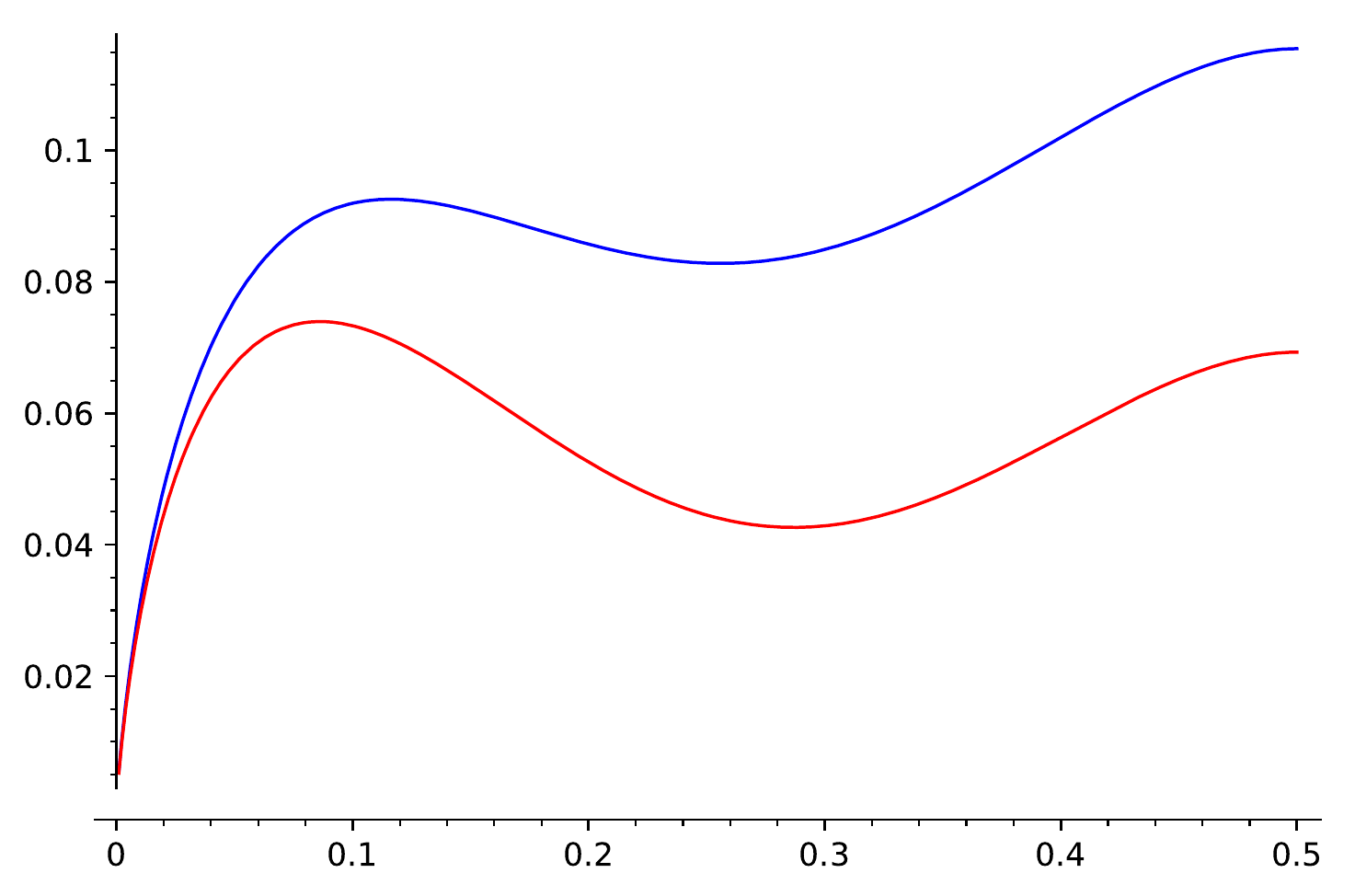}
	\caption{The r.h.s.\ of \eqref{eqXOR4} for $d=2.5$ (blue) and $d=2.7$ (red) in the interval $[0,\frac12]$.}\label{fig_xor_mmt}
\end{figure}

Concerning the computation of $\ex_{\fA}[\vZ^2]$, because the set of solutions is either empty or a translation of the kernel, we obtain
\begin{align}\label{eqXOR2}
\ex_{\fA}[\vZ^2]=\sum_{\sigma,\tau\in\FF_q^n}\pr_{\fA}\brk{\AA\sigma=\AA\tau=\vy}=\sum_{\sigma,\tau\in\FF_q^n}\pr_{\fA}\brk{\AA\sigma=\vy}\pr_{\fA}\brk{\sigma-\tau\in\ker\AA}=\ex_{\fA}\brk{\vZ}\ex_{\fA}|\ker\AA|.
\end{align}
To calculate the expected kernel size we notice that the probability that a vector $x$ is in the kernel depends on its Hamming weight.
For instance, the zero vector always belongs to the kernel, while the all-ones vector $\vecone$ does not \whp\
More systematically, invoking inclusion/exclusion, we find that for a vector $x$ of Hamming weight $w$ we have $\pr_{\fA}\brk{x\in\ker\AA}\sim\brk{(1+(1-2w/n)^3)/2}^{\vm}.$
Since the total number of such vectors comes to $\binom n{w}$, we obtain
	\begin{align}\label{eqXOR3}
		\ex_{\fA}|\ker\AA|&=\sum_{w=0}^n\binom nw\bcfr{1+(1-2w/n)^3}2^{\vm}.
	\end{align}
Taking logarithms, invoking Stirling's formula and parametrising $w=z n$, we simplify \eqref{eqXOR3} to
\begin{align}\label{eqXOR4}
	\log\,\ex_{\fA}|\ker\AA|&\sim n\cdot\max_{z\in[0,1]}-z\log z-(1-z)\log(1-z)+\frac\vm n\log\frac{1+(1-2z)^3}2&&\mbox{(cf.~\cite{DuboisMandler}).}
\end{align}
If we substitute $z=1/2$ into \eqref{eqXOR4}, the expression further simplifies to $(n-\vm)\log2$.
Hence, if the maximum is attained at another value $z\neq1/2$, then \eqref{eqXOR4} yields $\ex_{\fA}|\ker\AA|\gg 2^{n-\vm}$ and the second moment method fails.

Figure~\ref{fig_xor_mmt} displays \eqref{eqXOR4} for $d=2.5$ and $d=2.7$.
While for $d=2.5$ the function takes its maximum at $z=1/2$, for $d=2.7$ the maximum is attained at $z\approx 0.085$.
However, the true random 3-XORSAT threshold is $d\approx2.75$~\cite{DuboisMandler}.
Thus, the naive second moment calculation falls short of the real threshold. 

How so?
The expression \eqref{eqXOR4} does not determine the ``likely'' but the expected size of the kernel, a value prone to large deviations effects.
Indeed, because the number of vectors in the kernel scales exponentially with $n$, an exponentially unlikely event that causes an exceptionally large kernel may end up dominating $\ex_{\fA}|\ker\AA|$.
Precisely such an event manifests itself in the left local maxima in Figure~\ref{fig_xor_mmt}.
Moreover, as we approach the satisfiability threshold such large deviations issues are compounded by a diminishing error tolerance.
Indeed, while for $d=2.5$ the value at $z=1/2$ just swallows the spurious maximum, this is no longer the case for $d=2.7$.

For random $k$-XORSAT Dubois and Mandler managed to identify the precise large deviations effect at work.
It stems from fluctuations of a densely connected sub-graph of $\GG$ called the 2-core, obtained by iteratively pruning nodes of degree less than two along with their neighbours (if any).
Dubois and Mandler pinpointed the 3-XORSAT threshold by applying the second moment method to the minor $\AA^{(2)}$ induced by $\GG^{(2)}$ while conditioning on the 2-core having its typical dimensions.

The technical difficulty is that the rows of $\AA^{(2)}$ are no longer independent.
Indeed, $\AA^{(2)}$ is distributed as a random matrix with a truncated Poisson $\vd^{(2)}\sim\Po_{\geq2}(d')$ with $d'=d'(d,k)>0$ as the distribution of the variable degrees. 
Unfortunately, the given-degrees model leads to a fairly complicated moment computation.
Instead of the humble one-dimensional problem from \eqref{eqXOR4} we now face parameters $(z_i)_{i\geq2}$ that gauge the fraction of variables of each possible degree $i$ set to one.
Additionally, on the constraint side we need to keep track of the number of equations with zero and with two variables set to one.
Of course, these variables are tied together through the constraint that the total Hamming weight on the variable side match that on the constraint side.

With a deal of diligence Dubois and Mandler managed to solve this optimisation problem.
However, even just the step on to check degrees $k>3$ turns out to be tricky because now we need to keep track of all the possible ways in which a $k$-ary parity constraint can be satisfied~\cite{Dietzfelbinger,PittelSorkin}.
Yet even these difficulties are eclipsed by those that result from merely advancing to fields of size $q=3$~\cite{GoerdtFalke}.

Not to mention entirely general degree distributions $\vd,\vk$ and general fields $\FF_q$ as in \Thm~\ref{thm_main}.
The ensuing optimisation problem comes in terms of variables $(z_i)_{i\in\supp\vd}$ that range over the space $\cP(\FF_q)$ of probability distributions on $\FF_q$.
Additionally, there is a second set of variables $(\hat z_{\chi_1,\ldots,\chi_\ell})_{\ell\in\supp\vk,\,\chi_1,\ldots,\chi_\ell\in\supp\vec\chi}$ to go with the rows of $\AA$ whose non-zero entries are precisely $\chi_1,\ldots,\chi_\ell$.
These variables range over probability distributions on solutions $\sigma\in\FF_q^\ell$ to $\chi_1\sigma_1+\cdots+\chi_\ell\sigma_\ell=0$.
In terms of these variables we would need to solve 
\begin{align}\label{eqXOR11}
	\max&\quad\sum_{\sigma\in\FF_q}\ex\brk{(\vd-1)z_{\vd}(\sigma)\log z_{\vd}(\sigma)}
	-\frac{d}{k}\ex\brk{\sum_{\substack{\sigma_1,\dots,\sigma_{\vk}\in\FF_q\\{\vec\chi_{1,1}\sigma_1+\cdots+\vec\chi_{1,\vk}\sigma_{\vk}=0}}}\hat z_{\vec\chi_{1,1},\ldots,\vec\chi_{1,\vk}}(\sigma_1,\ldots,\sigma_{\vk})\log\hat z_{\vec\chi_{1,1},\ldots,\vec\chi_{1,\vk}}(\sigma_1,\ldots,\sigma_{\vk})}
	  \\&\quad \mbox{s.t.}\quad\ex[\vd z_{\vd}(\tau)]=\ex\brk{\sum_{\substack{\sigma_1,\dots,\sigma_{\vk}\in\FF_q\\{\vec\chi_{1,1}\sigma_1+\cdots+\vec\chi_{1,\vk}\sigma_{\vk}=0}}}\vk\vecone\cbc{\sigma_1=\tau}\hat z_{\vec\chi_{1,1},\ldots,\vec\chi_{1,\vk}}(\sigma_1,\ldots,\sigma_{\vk})}\qquad\mbox{for all }\tau\in\FF_q\nonumber.
\end{align}
As in random 3-XORSAT, a simple calculation shows that the value of \eqref{eqXOR11} evaluated at the ``equitable'' solution
\begin{align}\label{eqXOR12}
	z_i(\sigma)&=q^{-1}&\hat z_{\chi_1,\ldots,\chi_\ell}(\sigma_1,\ldots,\sigma_\ell)&=q^{1-\ell}&&\mbox{for all }i, \chi_1,\ldots,\chi_\ell
\end{align}
hits the value $(1-d/k)\log q$, which matches the normalised first moment $n^{-1}\log\ex_{\fA}[\vZ]$.

In summary, the second moment method hardly seems like a promising path towards \Thm~\ref{thm_main}.
Not only does \eqref{eqXOR11} seem unwieldy as even for very special cases of $\vd,\vk$ an analytic solution remains elusive~\cite{GoerdtFalke}.
Even worse, just in the case of ``unabridged'' random $k$-XORSAT large deviations effects may cause spurious maxima.
In effect, even if we could miraculously figure out the precise conditions for \eqref{eqXOR11} being attained at the uniform solution, this would hardly determine for what $\vd,\vk$ the random matrix $\AA$ actually has full row rank \whp\ 

\subsection{Quenching and truncating}\label{sec_quench}
The large deviations issues ultimately result from our attempt at computing the mean of $|\ker\AA|$, a (potentially) exponential quantity.
The mathematical physics prescription is to compute the expectation of its logarithm instead~\cite{MM}.
In the present algebraic setting this comes down to computing the mean of the nullity $\nul\AA=\dim\ker\AA$, or equivalently of the rank $\rk\AA=n-\nul\AA$.
This ``quenched average'' is always of order $O(n)$ and therefore immune to large deviations effects.
In fact, even if on some unfortunate event of exponentially small probability $\exp(-\Omega(n))$ the kernel of $\AA$ were quite large, the ensuing boost to $\ex_{\fA}[\nul\AA]$ remains negligible.

Yet computing the quenched average $\ex_{\fA}[\nul\AA]$ does not suffice to prove \Thm~\ref{thm_main}.
Indeed, \eqref{eqMaurice} already provides an asymptotic formula for $\ex_{\fA}[\nul\AA]$.
But as we saw due to the normalisation on the l.h.s.\ \eqref{eqMaurice} merely implies that $\rk\AA=\vm-o(n)$ \whp\
To actually prove that $\rk\AA=\vm$ \whp\ we will combine the quenched computation with a truncated moment argument calculation.
Specifically, we will harness an enhanced version of \eqref{eqMaurice} to prove that under the assumptions of \Thm~\ref{thm_main} the only combinatorially meaningful solutions to \eqref{eqXOR11} asymptotically coincide with the equitable solution \eqref{eqXOR12}, around which we will subsequently expand \eqref{eqXOR11} carefully.

To carry this programme out, let $\vx_{\AA}=(\vx_{\AA,i})_{i\in[n]}\in\FF_q^n$ be a random vector from the kernel of $\AA$.
Consider the event
\begin{align}\label{eqO}
	\fO&=\cbc{\sum_{\sigma,\tau\in\FF_q}\sum_{i,j=1}^n\abs{\pr\brk{\vx_{\AA,i}=\sigma,\ \vx_{\AA,j}=\tau\mid\AA}-q^{-2}}=o(n^2)}.
\end{align}
Then by Chebyshev's inequality on $\fO$ \whp\ we have 
\begin{align*}
	\sum_{i=1}^n\vecone\cbc{\vd_i=\ell,\,\vx_{\AA,i}=\sigma}=\pr\brk{\vd=\ell}n/q+o(n)&&\mbox{for all }\sigma\in\FF_q,\ \ell\in\supp\vd.
\end{align*}
Hence, on $\fO$ the only combinatorially relevant value of $z_\ell(\sigma)$ from \eqref{eqXOR11} is the uniform $1/q$ for every $\ell,\sigma$, because for every $\ell$ asymptotically almost all kernel vectors set about an equal number of variables of degree $\ell$ to each of the $q$ possible values.
Thanks to this observation will prove that  \whp\ 
\begin{align}\label{eqO1}
	\ex_{\fA}\brk{\vZ\cdot\vecone\cbc{\AA\in\fO}}&\sim\ex_{\fA}\brk{\vZ}\sim q^{n-\vm}&&\mbox{and}\\
	\ex_{\fA}\brk{\vZ^2\cdot\vecone\cbc{\AA\in\fO}}&\sim\ex_{\fA}\brk{\vZ}^2, \label{eqO2}
\end{align}
provided that \eqref{eqmain} is satisfied.
\Thm~\eqref{thm_main} will turn out to be an easy consequence of \eqref{eqO1}--\eqref{eqO2}, and \Cor~\ref{thm_Q} of \Thm~\ref{thm_main}.

Thus, the challenge is to prove \eqref{eqO1}--\eqref{eqO2}.
Specifically, while the second asymptotic equality in \eqref{eqO1} is easy, the proof of the first is where we require knowledge of the ``quenched average'' \eqref{eqMaurice}.
In fact, instead of just applying \eqref{eqMaurice} as is we will need to perform a ``quenched'' computation for a slightly enhanced random matrix from scratch.
Second, the key challenge towards the proof of \eqref{eqO2} is to obtain an exact asymptotic equality here, rather than the weaker estimate $\ex_{\fA}\brk{\vZ^2\cdot\vecone\cbc{\AA\in\fO}}=O(\ex_{\fA}\brk{\vZ}^2)$.
This will require a meticulous expansion of the second moment around the uniform solution, which will involve the detailed analysis of the lattices generated by integer vectors that encode conceivable values of $z_i,\hat z_{\chi_1, \ldots, \chi_\ell}$ from \eqref{eqXOR11}.

\subsection{The truncated first moment}\label{sec_overlap_outline}
Let us begin with \eqref{eqO1}.
Although we know the approximate nullity \eqref{eqMaurice} of $\AA$ already, this does not suffice to actually prove that $\fO$ is a ``likely'' event.
To this end we need to study a slightly modified matrix instead.
Specifically, for an integer $t\geq0$ obtain $\AA_{[t]}$ from $\AA$ by adding $t$ more rows that contain precisely three non-zero entries.
The positions of these non-zero entries are chosen uniformly, mutually independently and independently of everything else, and the non-zero entries themselves are independent copies of $\vec\chi$.
We require the following lower bound on the rank of $\AA_{[t]}$.

\begin{proposition}\label{prop_three}
	If \eqref{eqmain} is satisfied then there exists $\delta_0=\delta_0(\vd,\vk)>0$ such that for all $0<\delta<\delta_0$ we have 
	\begin{align}\label{eqprop_three}
		\liminf_{n\to\infty}\frac1n\Erw[\nul\AA_{\brk{\lfloor\delta n\rfloor}}]&\leq 1-\frac dk-\delta.
	\end{align}
\end{proposition}

The proof of \Prop~\ref{prop_three} relies on the Aizenman-Sims-Starr scheme, a coupling argument inspired by spin glass theory~\cite{Aizenman}.
The technique was also used in~\cite{Maurice} to prove the rank formula~\eqref{eqMaurice}.
While we mostly follow that proof strategy and can even reuse some of the intermediate deliberations, a subtle modification is required to accommodate the additional ternary equations.
The details can be found in \Sec~\ref{sec_prop_three}.

How does \Prop~\ref{prop_three} facilitate the proof of \eqref{eqO1}?
Assuming \eqref{eqmain}, we obtain from \eqref{eqMaurice} that $\nul\AA/n\sim1-d/k$ \whp\
Hence, \eqref{eqprop_three} shows that nearly each one of the of the additional ternary rows added to $\AA_{\brk{\lfloor\delta n\rfloor}}$ reduces the nullity.
We are going to argue that this is possible only if $\AA\in\fO$ \whp\

To see this, let us think about the kernel of a general $M\times N$ matrix $A$ over $\FF_q$ for a short moment.
Draw $\vx_A=(\vx_{A,i})_{i\in[N]}\in\ker A$ uniformly at random.
For any given coordinate $\vx_{A,i}$, $i\in[N]$ there are two possible scenarios: either $\vx_{A,i}=0$ deterministically, or $\vx_{A,i}$ is uniformly distributed over $\FF_q$.
(This is because if we multiply $\vx_A$ by a scalar $t\in\FF_q$ we obtain $t\vx_A\in\ker A$.)
We therefore call coordinate $i$ {\em frozen} if $x_{i}=0$ for all $x \in \ker A$ and unfrozen otherwise.
Let $\fF(A)$ be the set of frozen coordinates.

If $\AA$ had many frozen coordinates then adding an extra random row with three non-zero entries could hardly decrease the nullity \whp\
For if all three non-zero coordinates fall into the frozen set, then we get the new equation ``for free'', i.e., $\nul \AA_{[1]}=\nul \AA$.
Thus, \Prop~\ref{prop_three} implies that $|\fF(\AA)|=o(n)$ \whp\
We conclude that $\vx_{\AA,i}$ is uniformly distributed over $\FF_q$ for all but $o(n)$ coordinates $i\in[n]$.
However, this does not yet imply that $\vx_{\AA,i}$, $\vx_{\AA,j}$ are independent for most $i,j$, as required by $\fO$.
Yet a more careful argument based on the ``pinning lemma'' from~\cite{Maurice} does.
The proof of the following statement can be found in \Sec~\ref{sec_prop_overlap}.

\begin{proposition}\label{prop_overlap}
	Assume that \eqref{eqmain} is satisfied.
	Then \eqref{eqO1} holds \whp\
\end{proposition}

\subsection{Expansion around the equitable solution}\label{sec_moment_outline}
As outlined earlier, now that we know \eqref{eqO1} we can establish \eqref{eqO2} by expanding \eqref{eqXOR11} around the uniform distribution \eqref{eqXOR12}.
At first glance, this may not seem entirely immediate because \eqref{eqO1} only appears to fix the variables $(z_i(\sigma))_{i,\sigma}$ of \eqref{eqXOR11} that correspond to the variable nodes.
But thanks to a certain inherent symmetry property the optimal $\hat z_{\chi_1,\ldots,\chi_\ell}$ to go with the check nodes end up being nearly equitable as well.
This observation by itself now suffices to show without further ado that
\begin{align}\label{eqXOR13}
	\ex_{\fA}[\vZ^2\cdot\vecone\{\AA\in\fO\}]&=O\bc{\ex_{\fA}[\vZ\cdot\vecone\{\AA\in\fO\}]^2}.
\end{align}

Yet the estimate \eqref{eqXOR13} is not quite precise enough to complete the proof of \Thm~\ref{thm_main}.
Indeed, to apply Chebyshev's inequality we would need asymptotic equality as in \eqref{eqO2} rather than just an $O(\nix)$-bound; Huang~\cite{Huang} faced the same issue in the case $\vd=\vk$ constant and $q$ prime.
The proof of this seemingly innocuous improvement actually constitutes one of the main technical obstacles that we need to surmount.

As a first step, using a careful local expansion we will show that the dominant contribution to the second moment actually comes from $(z_\ell)_\ell$ such that
\begin{align}\label{eqGrid1}
	\sum_{\ell\in\supp\vd}\pr\brk{\vd=\ell}\sum_{\sigma\in\FF_q}|z_\ell(\sigma)-q^{-1}|&=O(n^{-1/2}).
\end{align}
But even once we know \eqref{eqGrid1} a critical issue remains because we allow general distributions of degrees $\vd,\vk$ and matrix entries $\vec\chi$.
In effect, to estimate the kernel size accurately we need to investigate the conceivable frequencies of field values that can lead to solutions.
Specifically, for an integer $k_0\geq3$ and $\chi_1,\ldots,\chi_{k_0}\in\FF_q^*$ let
\begin{equation}\label{eqS}
	\cS_q(\chi_1,\ldots,\chi_{k_0})=\cbc{\sigma\in\FF_q^{k_0}:\sum_{i=1}^{k_0}\chi_i\sigma_i=0}
\end{equation}
comprise all solutions to a linear equation with coefficients $\chi_1,\ldots,\chi_{k_0}\in\FF_q$.
For each $\sigma\in\cS_q(\chi_1,\ldots,\chi_{k_0})$ the vector
\begin{align}\label{eqMyFreq}
	\hat\sigma=\bc{\sum_{i=1}^{k_0}\vecone\cbc{\sigma_i=s}}_{s\in\FF_q^*}\in\ZZ^{\FF_q^*}
\end{align}
tracks the frequencies with which the various non-zero field elements appear.
Depending on the coefficients $\chi_1,\ldots,\chi_{k_0}$, the frequency vectors $\hat\sigma$ may be confined to a proper sub-grid of the integer lattice.
For example, in the case $q=k_0=3$ and $\chi_1=\chi_2=\chi_3=1$ they span the sub-lattice spanned by $\binom11$ and $\binom03$.
The following proposition characterises the lattice spanned by the frequency vectors for general $k_0$ and $\chi_1,\ldots,\chi_{k_0}$.

\begin{proposition}\label{prop_module}
	Let $k_0\geq3$, let $\chi_1,\ldots,\chi_{k_0}\in\FF_q^*$ and let $\fM_q(\chi_1,\ldots,\chi_{k_0})\subseteq\ZZ^{\FF_q^*}$ be the $\ZZ$-module generated by the frequency vectors $\hat\sigma$ for $\sigma\in\cS_q(\chi_1,\ldots,\chi_{k_0})$.
	Then $\fM_q(\chi_1,\ldots,\chi_{k_0})$ has a basis $\fb_1,\ldots,\fb_{q-1}$ of non-negative integer vectors with $\|\fb_i\|_1\leq3$ for all $1\leq i\leq q-1$ such that $\det\bc{\fb_1\ \cdots\ \fb_{q-1}}=q^{\vecone\{\chi_1=\cdots=\chi_{k_0}\}}.$
\end{proposition}

\noindent
A vital feature of \Prop~\ref{prop_module} is that the module basis consists of non-negative integer vectors with small $\ell_1$-norm.
In effect, the basis vectors are ``combinatorially meaningful'' towards our purpose of counting solutions.
Perhaps surprisingly, the proof of \Prop~\ref{prop_module} turns out to be rather delicate, with details depending on whether $q$ is a prime or a prime power, among other things.
The details can be found in \Sec~\ref{sec_prop_module}.

In addition to the subgrid constraints imposed by the linear equations themselves, we need to take a divisibility condition into account.
Indeed, for any assignment $\sigma\in\FF_q^n$ of values to variables the frequencies of the various field elements $s\in\FF_q$ are divisible by the g.c.d.\ $\fd$ of $\supp(\vd)$, i.e.
\begin{align}\label{eqGrid2}
	\fd\mid\sum_{i=1}^n\vd_i\vecone\cbc{\sigma_i=s}&&\mbox{for all }s\in\FF_q.
\end{align}
Thus, to compute the expected kernel size we need to study the intersection of the sub-grid \eqref{eqGrid2} with the grid spanned by the frequency vectors $\hat\sigma$ for $\sigma\in\cS_q(\vec\chi_{1,1},\ldots,\vec\chi_{1,\vk})$.
Specifically, by way of estimating the number of assignments represented by each grid point and calculating the ensuing satisfiability probability, we obtain the following.

\begin{proposition}\label{prop_mmt}
	Assume that $q$ and $\fd$ are coprime and that \eqref{eqmain} is satisfied.
	Then \eqref{eqO2} holds \whp
\end{proposition}

\noindent
We prove \Prop~\ref{prop_mmt} in \Sec~\ref{sec_prop_mmt}.
Combining \Prop s~\ref{prop_three}--\ref{prop_mmt}, we now establish the main theorem.

\begin{proof}[Proof of \Thm~\ref{thm_main}]
	The assumption \eqref{eqmain} implies that $1-d/k=\Phi(0)>\Phi(1)=0$.
	Combining \Prop s~\ref{prop_overlap} and~\ref{prop_mmt}, we obtain \eqref{eqO1}--\eqref{eqO2}.
	Hence, Chebyshev's inequality implies that $\vZ\geq q^{n-\vm}=q^{n(1-d/k+o(1))}>0$ \whp\
	Consequently, the random linear system $\AA x=\vy$ has a solution \whp, and thus $\rk\AA=\vm$ \whp
\end{proof}

\begin{proof}[Proof of \Cor~\ref{thm_Q}]
	Let $q$ be a prime that does not divide $\fd$ and let $\vec\chi=1$ deterministically.
Obtain the matrix $\bar\BB\in\FF_q^{\vm\times n}$ by reading the $\{0,1\}$-entries of $\BB$ as elements of $\FF_q$.
Then the distribution of $\bar\BB$ coincides with the distribution of the random $\FF_q$-matrix $\AA$.
Hence, \Thm~\ref{thm_main} implies that $\bar\BB$ has full row rank \whp

Suppose that indeed $\rk\bar\BB=\vm$.
We claim that then the rows of $\BB$ are linearly independent.
Indeed, assume that $z^\trans\BB=0$ for some vector $z=(z_1,\ldots,z_{\vm})^\trans\in\ZZ^{\vm}$.
Factoring out $\gcd(z_1,\ldots,z_{\vm})$ if necessary, we may assume that the vector $\bar z\in\FF_q^{\vm}$ with entries $\bar z_i=z_i+q\ZZ$ is non-zero.
Since $z^\trans\BB=0$ implies that $\bar z^\trans\bar\BB=0$, the rows of $\bar\BB$ are linearly dependent, in contradiction to our assumption that $\bar\BB$ has full row rank.
\end{proof}

\subsection{Discussion and related work}\label{sec_related}
The present proof strategy draws on the prior work~\cite{Ayre,Maurice} on the rank of random matrices.
Specifically, toward the proof of \Prop~\ref{prop_three} we extend the Aizenman-Sims-Starr technique from~\cite{Maurice} and to prove \Prop~\ref{prop_overlap} we generalise an argument from~\cite{Ayre}.
Additionally, the expansion around the centre carried out in the proof of \Prop~\ref{prop_mmt} employs some of the techniques developed in the study of satisfiability thresholds, particularly the extensive use of local limit theorems and auxiliary probability spaces~\cite{Kosta,KostaNAE}.

The principal new proof ingredient is the asymptotically precise analysis of the second moment by means of the study of the sub-grids of the integer lattice induced by the constraints as sketched in \Sec~\ref{sec_moment_outline}.
This issue that was absent in the prior literature on variations on random $k$-XORSAT~\cite{Ayre,Maurice,CFP} and on other random constraint satisfaction problems~\cite{Kosta,KostaNAE}.
However, in the study of the random regular matrix from Example~\ref{ex_fixed} Huang~\cite{Huang} faced a similar issue in the special case $\vd=\vk$ constant and $\vec\chi=1$ deterministically.
\Prop~\ref{prop_module}, whose proof is based on a combinatorial investigation of lattices in the general case, constitutes a generalisation of the case Huang studied.
A further feature of \Prop~\ref{prop_module} absent in~\cite{Huang} is the explicit $\ell_1$-bound on the basis vectors.
This bound facilitates the proof of \Prop~\ref{prop_mmt}, which ultimately carries out the expansion around the equitable solution.

Satisfiability thresholds of random constraint satisfaction problems have been studied extensively in the statistical physics literature via a non-rigorous technique called the ``cavity method''.
The cavity method comes in two installments: the simpler ``replica symmetric ansatz'' associated with the Belief Propagation message passing scheme, and the more intricate ``replica symmetry breaking ansatz''.
The proof of \Thm~\ref{thm_main} demonstrates that the former renders the correct prediction as to the satisfiability threshold of random linear equations.
By contrast, in quite a few problems, notoriously random $k$-SAT, replica symmetry breaking occurs~\cite{Jean,DSS3}.

An intriguing question for future work might be to understand the ``critical'' case of $\Phi$ that attain their global max at $0$ and another point left open by \Thm~\ref{thm_main}.
While Example~\ref{ex_identical} shows that it cannot generally be true that $\AA$ has full row rank \whp, the regular case where $\vd=\vk=d$ are fixed to the same constant provides an intriguing example.
For this scenario Huang proved that the random $\{0,1\}$-matrix $\BB$ has full rank \whp~\cite{Huang}.
The proof, based effectively on a moment computation over finite fields and local limit techniques, also applies to the adjacency matrices of random $d$-regular graphs.

While the present paper deals with sparse random matrices with a bounded average number of non-zero entries in each row and column, the case of dense random matrices has received a great deal of attention, too.
\Komlos~\cite{Komlos} first shows that dense square random $\{0,1\}$-matrices are regular over the rationals \whp; Vu~\cite{Vu} suggested an alternative proof.
The computation of the exponential order of the singularity probability subsequently led to a series of intriguing articles~\cite{KKS,TaoVu,Tikhomirov}.
By contrast, the singularity probability of a dense square matrix over a finite field converges to a value strictly between zero and one~\cite{Kovalenko,Kovalenko2,Lev1,Lev2}.

Apart from the sparse and dense case, the regime of intermediate densities has been studied as well.
Balakin~\cite{Balakin2} and Bl\"omer, Karp and Welzl~\cite{BKW} dealt with the rank of such random matrices of intermediate densities over finite fields.
In addition, Costello and Vu~\cite{costello2008rank,costello2010rank} studied the rational rank of random symmetric matrices of an intermediate density.

Indeed, an interesting open problem appears to be the extension of the present methods to the symmetric case.
In particular, it would be interesting to see if the present techniques can be used to add to the line of works on the adjacency matrices of random graphs, which have been approached by means of techniques based on local weak convergence or Littlewood-Offord techniques~\cite{BLS,Ferber}.

\subsection{Organisation}
After some preliminaries in \Sec~\ref{sec_pre} we begin with the proof of \Prop~\ref{prop_three} in \Sec~\ref{sec_prop_three}.
The proof relies on an Aizenman-Sims-Starr coupling argument, some details of which are deferred to \Sec~\ref{sec_prop_auxphi}.
\Sec~\ref{sec_prop_overlap} deals with the proof of \Prop~\ref{prop_overlap}.
Subsequently we prove \Prop~\ref{prop_module} in \Sec~\ref{sec_prop_module}, thereby laying the ground for the proof of \Prop~\ref{prop_mmt} in \Sec~\ref{sec_prop_mmt}.

\section{Preliminaries}\label{sec_pre}
\noindent
Unsurprisingly, the proofs of the main results involve a few concepts and ideas from linear algebra.
We mostly follow the terminology from~\cite{Maurice}, summarised in the following definition.

\begin{definition}[{\cite[Definition~2.1]{Maurice}}]\label{Def_Alp}
	Let $A$ be an $m\times n$-matrix over a field $\FF$.
	\begin{itemize}
		\item A set $\emptyset\neq I\subset[n]$ is a \bemph{relation} of $A$ if there exists a row vector 
			$y\in\FF^{1\times m}$ such that $\emptyset\neq\supp(y A)\subset I$.
		\item If $I=\cbc{i}$ is a relation of $A$, then we call $i$ \bemph{frozen} in $A$. Let $\fF(A)$ be the set of all frozen $i\in[n]$ and let $$\ff(A)=|\fF(A)|/n.$$
		\item A set $I\subset[n]$ is a \bemph{proper relation} of $A$ if $I\setminus\fF(A)$ is a relation of $A$.
		\item For $\delta>0$, $\ell\geq1$ we say that $A$ is \bemph{$(\delta,\ell)$-free} if there are no more than $\delta n^\ell$ proper relations $I\subset[n]$ of size $|I|=\ell$.
	\end{itemize}
\end{definition}

Thus, a relation is set of column indices such that the support of a non-zero linear combination $yA$ of rows of $A$ is contained in that set of indices.
Of course, every single row induces a relation on the column indices where it has non-zero entries.
An important special case is a relation consisting of one coordinate $i$ only.
If such a relation exists, then $x_i=0$ for all vectors $x\in\ker A$, which is why we call such a coordinate $i$ frozen.
Furthermore, a proper relation is a relation that is not just built up of frozen variables.
Finally, we introduce the term $(\delta,\ell)$-free to express that $A$ has ``relatively few'' relations of size $\ell$ as we will generally employ this term for bounded $\ell$ and small $\delta>0$.

The following observation will aid the Aizenman-Sims-Starr coupling argument, where we will need to study the effect of adding a few extra rows and columns to a random matrix.

\begin{lemma}[{\cite[\Lem~2.4]{Maurice}}]\label{Cor_free}
Let $A,B,C$ be matrices of size $m \times n$, $m'\times n$ and $m'\times n'$, respectively, and let $I\subset[n]$ be the set of all indices of non-zero columns of $B$.
Moreover, obtain $B_*$ from $B$ by replacing for each $i\in I\cap\fF(A)$ the $i$-th column of $B$ by zero.
Unless $I$ is a proper relation of $A$ we have
\begin{align}\label{eqLemma_free}
\nul\begin{pmatrix}A&0\\B&C\end{pmatrix}-\nul A=n'-\rk(B_*\ C).
\end{align}
\end{lemma}

Apart from \Lem~\ref{Cor_free} we will harness an important trick called the ``pinning operation''.
The key insight is that for {\em any} given matrix we can diminish the number of short proper relations by simply expressly freezing a few random coordinates.
The basic idea behind the pinning operation goes back to the work of Montanari~\cite{Montanari} and has been used in other contexts~\cite{CKPZ,Raghavendra}.
The version of the construction that we use here goes as follows.

\begin{definition}[{\cite[Definition~2.2]{Maurice}}]\label{Def_pin}
Let $A$ be an $m\times n$ matrix and let $\theta\geq0$ be an integer.
Let $\vi_1,\vi_2,\ldots,\vi_\theta\in[n]$ be uniformly random and mutually independent column indices.
Then the matrix $A[\theta]$ is obtained by adding $\theta$ new rows to $A$ such that for each $j\in[\theta]$ the $j$-th new row has precisely one non-zero entry, namely a one in the $\vi_j$-th column.
\end{definition}

\begin{proposition}[{\cite[\Prop~2.3]{Maurice}}]\label{Prop_Alp}
For any $\delta>0$, $\ell>0$ there exists $\Theta_0=\Theta_0(\delta,\ell)>0$ such that for all $\Theta>\Theta_0$ and for any matrix $A$ over any field $\FF$ the following is true.
With $\THETA\in[\Theta]$ chosen uniformly at random we have $\pr\brk{\mbox{$A[\THETA]$ is $(\delta,\ell)$-free}}>1-\delta.$
\end{proposition}  

As a fairly immediate application of \Prop~\ref{Prop_Alp} we conclude that if the pinning operation applied to a random matrix over a finite field leaves us with few frozen variables, a decorrelation condition akin to the event $\fO$ from \eqref{eqO} will be satisfied.
For a matrix $A$ we continue to denote by $\vx_A$ a random vector from $\ker A$.

\begin{corollary}[{\cite[\Lem~4.2]{Maurice}}]\label{lem_pinning_RS}
	For any $\zeta>0$ and any prime power $q>0$ there exist $\xi>0$ and $\Theta_0>0$ such that for any $\Theta>\Theta_0$ for large enough $n$ the following is true.
Let $A$ be a $m\times n$-matrix over $\FF_q$.
Suppose that for a uniformly random $\vec\theta\in[\Theta]$  we have $\Erw|\fF(A[\vec\theta])|<\xi n$. Then
$$\sum_{\sigma,\tau\in\FF_q}\sum_{i,j=1}^n\ex\abs{\pr\brk{\vx_i=\sigma,\,\vx_j=\tau\mid A[\vec\theta]}-q^{-2}}<\zeta n^2.$$
\end{corollary}

As mentioned earlier, at a key junction of the moment computation we will need to estimate the number of integer lattice points that satisfy certain linear relations.
The following elementary estimate will prove useful.

\begin{lemma}{{\cite[p.~135]{Lenstra}}}\label{lemma_gridcount}
	Let $\fM\subset\RR^\ell$ be a $\ZZ$-module with basis $b_1,\ldots,b_\ell$.
	Then
	$$\lim_{r\to\infty}\frac{\abs{\cbc{x\in\fM:\|x\|\leq r}}}{\vol\bc{\cbc{x\in\RR^\ell:\|x\|\leq r}}}=\frac1{|\det(b_1\cdots b_\ell)|}.$$
\end{lemma}

The definition of the random Tanner graph in \Sec~\ref{sec_results_finite} provides that $\GG$ is simple.
Commonly it is easier to conduct proofs for an auxiliary random multigraph drawn from a pairing model and then lift the results to the simple random graph.
This is how we proceed as well.
Thus, given \eqref{deg_sums} we let $\vG$ be the random bipartite graph on the set $\{x_1,\ldots,x_n\}$ of variable nodes and $\{a_1,\ldots,a_{\vm}\}$ of check nodes generated by drawing a perfect matching $\vec\Gamma$ of the complete bipartite graph on
	$$\bigcup_{i=1}^n\cbc{x_i}\times[\vd_i]\qquad\mbox{and}\qquad\bigcup_{i=1}^{\vm}\cbc{a_i}\times[\vk_i]$$
	and contracting the sets $x_i\times[\vd_i]$ and $a_i\times[\vk_i]$ of variable/check clones.
	We also let $\vA$ be the random matrix to go with this random multi-graph.
	Hence, 
	$$\vA_{ij}=\vec\chi_{i,j}\sum_{u=1}^{\vk_i}\sum_{v=1}^{\vd_j}\vecone\{\{(a_i,u),(x_j,v)\}\in\vec\Gamma\}\enspace.$$
	Routine arguments show that $\G$ is simple with a non-vanishing probability.

\begin{proposition}[{\cite[\Lem~4.3]{Maurice}}]\label{lemma_contig}
	We have $\pr\brk{\G\mbox{ is simple}\mid\sum_{i=1}^n\vd_i=\sum_{i=1}^{\vm}\vk_i}=\Omega(1)$.
\end{proposition}

When working with the random graphs $\GG$ or $\G$ we occasionally encounter the size-biased versions $\hat\vd,\hat\vk$ of the degree distributions defined by
\begin{align}\label{eqSizeBiasd}
	\pr\brk{\hat\vd=\ell}&=\ell\pr\brk{\vd=\ell}/{d},&\pr\brk{\hat\vk=\ell}&=\ell\pr\brk{\vk=\ell}/{k}&(\ell\geq0).
\end{align}
In particular, these distributions occur in the Aizenman-Sims-Starr coupling argument.
In that context we will also need the following crude but simple tail bound.

\begin{lemma}[{\cite[Lemma 1.8]{Maurice}}]\label{Lemma_sums}
	Let $(\vec\lambda_i)_{i\geq 1}$ be a sequence of independent copies of an integer--valued random variable $\vec\lambda\geq 0$ with $\Erw\brk{\vec\lambda^r} < \infty$ for some $r>2$.  
	Further, let $s$ be a sequence such that $s=\Theta(n)$. 
	Then for all $\delta >0$, $$\pr\brk{\abs{\sum_{i=1}^s (\vec\lambda_i - \Erw\brk{\vec\lambda})}> \delta n} = o(1/n).$$ 
\end{lemma}

Finally, throughout the article we use the common $O(\nix)$-notation to refer to the limit $n\to\infty$.
In addition, we will sometimes need to deal with another parameter $\eps>0$.
In such cases we use $O_\eps(\nix)$ and similar symbols to refer to the double limit $\eps \to 0$ after $n\to \infty$.

\section{Proof of \Prop~\ref{prop_three}}\label{sec_prop_three}

\subsection{Overview}
The first ingredient of the proof of \Prop~\ref{prop_three} is a coupling argument inspired by the Aizenman-Sims-Starr scheme from mathematical physics~\cite{Aizenman}, which also constituted the main ingredient of the proof of the approximate rank formula~\eqref{eqMaurice} from~\cite{Maurice}.
Indeed, the coupling argument here is quite similar to that from~\cite{Maurice}, with some extra bells and whistles to accommodate the additional ternary equations.
We therefore defer that part of the proof to \Sec~\ref{sec_prop_auxphi}.
The Aizenman-Sims-Starr argument leaves us with a variational formula for the rank of $\AA_{[\lfloor\delta n\rfloor]}$.
The second proof ingredient is to solve this variational problem.
Harnessing the assumption \eqref{eqmain}, we will obtain the explicit expression for the rank provided by \Prop~\ref{prop_three}.

Let us come to the details.
As explained in \Sec~\ref{sec_pre}, we will have an easier time working with the pairing model versions $\G,\A$ of the Tanner graph and the random matrix.
Moreover, to facilitate the coupling argument we will need to poke a few holes, known as ``cavities'' in physics jargon, into the random matrix.
More precisely, we will slightly reduce the number of check nodes and tolerate a small number of variable nodes $x_i$ of degree less than $\vd_i$.
The cavities will provide the flexibility needed to set up the coupling argument.

Formally, let $\eps, \delta \in (0,1)$ and let $\Theta\geq 0$ be an integer.
Ultimately $\Theta$ will depend on $\eps$ but not on $n$ or $\delta$.
We then construct the random matrix $\vA\brk{n,\eps,\delta,\Theta}$ as follows. 
Let 
\begin{align}\label{eqms}
	\vm_\eps &\sim \Po((1-\eps)dn/k),& \vm_\delta &\sim \Po(\delta n),&\vec \theta &\sim \text{unif}([\Theta]).
\end{align}
The Tanner multigraph $\vec G\brk{n,\eps, \delta, \Theta}$ has variable nodes $x_1, \ldots, x_n$ and check nodes $a_1, \ldots, a_{\vm_{\eps}}, t_1, \ldots, t_{\vm_\delta}, p_1, \ldots, p_{\vec \theta}$.
To connect them draw a random maximum matching $\vec \Gamma\brk{n,\eps}$ of the complete bipartite graph with vertex classes
\begin{align*}
    V_1= \bigcup_{i=1}^{\vm_\eps} \cbc{a_i} \times [\vk_i] \qquad \text{and} \qquad V_2= \bigcup_{j=1}^{n}\cbc{x_j} \times [\vd_j].
\end{align*}
For every matching edge $\{(a_i,h),(x_j,\ell)\}\in\vec\Gamma[n,\eps]$, $h\in[\vk_i],\ell\in[\vd_j]$, between a clone of $x_h$ and a clone of $a_i$ we insert an $a_i$-$x_j$-edge into $\vec G\brk{n,\eps, \delta, \Theta}$.
Moreover, the check nodes $t_1, \ldots, t_{\vm_\delta}$ each independently choose three neighboring variables uniformly with replacement random among $\cbc{x_1, \ldots, x_n}$. 
Further, check node $p_\ell$ for $\ell \in [\vec \theta]$ is adjacent to $x_\ell$ only. 
Finally, to obtain the random $(\vec \theta +\vm_\eps+\vm_\delta) \times n$-matrix $\vA\brk{n,\eps,\delta,\Theta}$ from $\vec G\brk{n,\eps, \delta, \Theta}$ we let
\begin{align}
	\vA\brk{n,\eps,\delta,\Theta}_{p_i,x_h}&=\vecone\cbc{i=h}&&(i\in[\THETA],h\in[n]),\\
\vA\brk{n,\eps,\delta,\Theta}_{a_{i},x_h}&=\vec\chi_{i, h} \sum_{\ell=1}^{\vk_i}\sum_{s=1}^{\vd_h} \vecone \{(x_h, s), (a_{i}, \ell)\}\in\vec{\Gamma}\brk{n,\eps}\}&&(i\in[\vm_\eps], h\in[n]),\\
\vA\brk{n,\eps,\delta,\Theta}_{t_i,x_h}&=\vec\chi_{\vm_{\eps}+i, h} \sum_{\ell=1}^3\vecone\{\vec i_{i,\ell} = h\} &&(i\in[\vm_\delta],h\in[n]).\label{eq_ternary1}
\end{align}
Applying the Aizenman-Sims-Starr scheme to the matrix $\A[n,\eps,\delta,\Theta]$, we obtain the following variational bound.

\begin{proposition}\label{prop_auxphi}
	There exist $\delta_0>0$, $\Theta_0(\eps)>0$ such that for all $0<\delta<\delta_0$ and any $\Theta=\Theta(\eps)\geq\Theta_0(\eps)$ we have
	\begin{align}\label{eq_prop_auxphi}
		\limsup_{\eps\to0}\limsup_{n \to \infty} \frac{1}{n} \Erw\brk{\nul(\vA\brk{n,\eps,\delta,\Theta})} 
	& \leq \max_{\alpha, \beta \in [0,1]} \Phi(\alpha) + \bc{\exp\bc{-3\delta \beta^2}-1}D(1-K'(\alpha)/k) -\delta + 3\delta\beta^2 -2\delta\beta^3.
	\end{align}
\end{proposition}

\noindent
The proof of \Prop~\ref{prop_auxphi}, carried out in \Sec~\ref{sec_prop_auxphi} in detail, resembles that of the rank formula \eqref{eqMaurice}, except that we have to accommodate the additional ternary checks $t_i$.
Their presence is the reason why the optimisation problem on the r.h.s.\ comes in terms of two variables $\alpha,\beta$ rather than a single variable as \eqref{eqMaurice}.

To complete the proof of \Prop~\ref{prop_three} we need to solve the optimisation problem \eqref{eq_prop_auxphi}.
This is the single place where we require that $\Phi(z)$ take its unique global max at $z=0$, which ultimately implies that the optimiser of \eqref{eq_prop_auxphi} is $\alpha=\beta=0$.
This fact in turn implies the following.

\begin{proposition}\label{prop_aux}
	For any $\vd,\vk$ that satisfy \eqref{eqmain} there exists $\delta_0>0$ such that for all $0<\delta<\delta_0$ we have
	\begin{align*}
		\max_{\alpha, \beta \in [0,1]} \Phi(\alpha) + \bc{\exp\bc{-3\delta \beta^2}-1}D(1-K'(\alpha)/k) -\delta + 3\delta\beta^2 -2\delta\beta^3=1-\frac{d}{k}-\delta.
	\end{align*}
\end{proposition}

\noindent
The proof of \Prop~\ref{prop_aux} can be found in \Sec~\ref{Sec_findmax}.
Finally, in \Sec~\ref{sec_proof_three} we will see that \Prop~\ref{prop_three} is an easy consequence of \Prop s~\ref{prop_auxphi} and~\ref{prop_aux}.

\subsection{Proof of \Prop~\ref{prop_aux}}\label{Sec_findmax}
Let
\begin{align*}
	\tilde{\Phi}_\delta(\alpha, \beta) &=  \Phi(\alpha) + \bc{\exp\bc{-3\delta \beta^2}-1}D(1-K'(\alpha)/k) -\delta + 3\delta\beta^2 -2\delta\beta^3&&(\alpha, \beta \in [0,1]).
\end{align*}
Assuming \eqref{eqmain}, we are going to prove that for small enough $\delta$,
\begin{align}\label{eqprop_aux1}
	\max_{\alpha, \beta \in [0,1]} \tilde{\Phi}_\delta(\alpha, \beta) =\tilde{\Phi}_\delta(0,0) = 1-\frac{d}{k}-\delta, 
\end{align}
whence the assertion is immediate.

	The $C^1$-function $\tilde\Phi_\delta$ attains its maximum either at a boundary point of the compact domain $[0,1]^2$ or at a point where the partial derivatives vanish.
	Beginning with the former, we consider four cases.
	\begin{description}
		\item[Case 1: $\alpha=0$] we have
			\begin{align}\label{eqlemma_aux_1}
				\tilde{\Phi}_\delta(0, \beta) = \tilde{\Phi}_\delta(0,0) + 3\delta \beta^2 - 2\delta \beta^3 -(1-\exp\bc{-3\delta \beta^2}).
			\end{align} 
			Expanding the exponential function, we see that $3\delta \beta^2 - 2\delta \beta^3 -(1-\exp\bc{-3\delta \beta^2}) = - 2\delta \beta^3 + O_\delta(\delta^2\beta^4)$.
			Since $- 2\delta \beta^3 + O_\delta(\delta^2\beta^4)$ is non-positive for all $\beta \in [0,1]$, \eqref{eqlemma_aux_1} yields $\max_\beta\tilde{\Phi}_\delta(0,\beta)=\tilde\Phi_\delta(0,0)$ for small enough $\delta>0$.
		\item[Case 2: $\beta=0$] the assumption \eqref{eqmain} ensures that $\Phi$ is maximised in $0$.
			Therefore, as $\tilde{\Phi}_\delta(\alpha,0) = \Phi(\alpha)-\delta$, the maximum on $\cbc{(\alpha,0): \alpha \in [0,1]}$ is attained in $\alpha =0$.
		\item[Case 3: $\alpha=1$] we obtain
			\begin{align*}
				\tilde{\Phi}_\delta(1,\beta) &= \Phi(1) - \delta + 3 \delta \beta^2 - 2 \delta \beta^3 =  \delta (3\beta^2 - 2 \beta^3-1).
			\end{align*}
			Since $- 2 \beta^3 +  3\beta^2\leq1$ for all $\beta\in[0,1]$ and $d/k<1$, for small enough $\delta$ we obtain $ \tilde{\Phi}_\delta(1,\beta) \leq 1-d/k-\delta = \tilde{\Phi}_\delta(0,0)$.
		\item[Case 4: $\beta=1$] we have
			\begin{equation}\label{eqlemma_aux_2}
				\tilde{\Phi}_\delta(\alpha, 1) = \Phi(\alpha) - (1-\exp\bc{-3\delta})D\bc{1-\frac{K'(\alpha)}{k}}.
			\end{equation}
			Because $D$ and $K'$ are continuous on $[0,1]$ due to the assumption $\ex[\vd^2]+\ex[\vk^2]<\infty$, for any $\zeta>0$ there exists $\hat\alpha>0$ such that $D(1-K'(\alpha)/k) > 1-\zeta$ for all $0<\alpha<\hat\alpha$.
			Therefore, \eqref{eqlemma_aux_2} shows that for small enough $\delta>0$ and $0<\alpha<\hat\alpha$ we have $\tilde{\Phi}_\delta(\alpha, 1) < \tilde{\Phi}_\delta(\alpha,0)\leq\tilde{\Phi}_\delta(0,0)$.  
			On the other hand, for $\hat\alpha\leq\alpha\leq1$ the difference $\Phi(\alpha)-\Phi(0)$ is uniformly negative because of our assumption \eqref{eqmain} that $\Phi$ attains its unique global maximum at $\alpha=0$. 
			Hence, for $\delta$ small enough and $\hat\alpha\le\alpha\leq1$ we obtain $\tilde{\Phi}_\delta(\alpha, 1) < \tilde{\Phi}_\delta(0,0)$.
	\end{description}
	Combining Cases 1--4, we obtain
	\begin{align}\label{eqlemma_aux_3}
		\max_{(\alpha,\beta)\in\partial[0,1]^2}\tilde{\Phi}_\delta(\alpha,\beta) = \tilde{\Phi}_\delta(0,0).
	\end{align}

	Moving on to the interior of $[0,1]^2$, we calculate the derivatives
	\begin{align*}
		\frac{\partial \tilde\Phi_\delta}{\partial \alpha} &= \Phi'(\alpha) + \bc{1-\exp\bc{-3\delta \beta^2}} \frac{K''(\alpha)}{k}D'(1-K'(\alpha)/k) = \frac{K''(\alpha)}{k} \bc{d(1-\alpha)-\exp\bc{-3\delta \beta^2}D'(1-K'(\alpha)/k)},\\
		\frac{\partial \tilde\Phi_\delta}{\partial \beta} &= 6 \delta \beta \bc{1-\beta-\exp\bc{-3\delta \beta^2}D(1-K'(\alpha)/k)}. 
	\end{align*}
	Hence, potential maximisers $(\alpha,\beta)$ in the interior of $[0,1]^2$ satisfy
	\begin{equation}\label{int_max}
		d(1-\alpha)=D'(1-K'(\alpha)/k)\exp\bc{-3\delta \beta^2} \qquad \text{and} \qquad 1-\beta = \exp\bc{-3\delta \beta^2}D(1-K'(\alpha)/k). 
	\end{equation}
	Substituting \eqref{int_max} into $\tilde\Phi_\delta$, we obtain
	\begin{align}\nonumber
		\tilde{\Phi}_\delta(\alpha, \beta)&= \Phi(\alpha) - \delta + \bc{\exp\bc{-3\delta \beta^2}-1}D(1-K'(\alpha)/k) + 3\delta\beta^2 -2\delta\beta^3\\
		 &=\Phi(\alpha) - \delta + (1-\beta)(1-\exp\bc{3\delta \beta^2})+ 3\delta\beta^2 -2\delta\beta^3 \nonumber\\
		 &\leq   \Phi(\alpha) - \delta - 3\delta \beta^2 (1-\beta)+ 3\delta\beta^2 -2\delta\beta^3 = \Phi(\alpha)-\delta +\delta \beta^3.
		 \label{int_max2}
	\end{align}

	To estimate the r.h.s.\ we consider the cases of small and large $\alpha$ separately. 
	Specifically, by continuity for any $\zeta>0$ there is $0<\hat\alpha<\delta$ such that $D(1-K'(\alpha)/k)>1-\zeta$ for all $0<\alpha<\hat\alpha$.
	\begin{description}
		\item[Case 1: $0<\alpha<\hat\alpha$] 
		Since $D(1-K'(\alpha)/k)>1-\zeta$, \eqref{int_max} implies that for $\beta >0$
			$$1-\beta > (1-3\delta \beta^2) (1-\zeta) = 1 - \zeta - 3\delta \beta^2(1-\zeta).$$
			In particular, small $\hat \alpha$ implies that also $\beta$ is small. More precisely, after choosing $\delta,\zeta$ small enough, we may assume that $\beta<\hat\beta$ for any fixed $\hat\beta>0$.
			In this case, we may thus restrict to solutions $(\alpha, \beta) \in (0,1)^2$ to \eqref{int_max} where \textit{both} coordinates are sufficiently small. Also here, we distinguish three cases that all lead to contradictions. 
			\begin{itemize}
			\item[(A)] If the solution satisfies $\alpha=\beta$, consider the function
			$$x \mapsto 1-x-\exp\bc{-3\delta x^2} D(1-K'(x)/k)$$
			whose zeros determine the solutions to the right equation in \eqref{int_max} under the assumption $\alpha=\beta$.
			Its value is zero at $x=0$ and it has derivative
			$$-1+6\delta x \exp\bc{-3\delta x^2} D(1-K'(x)/k) + \exp\bc{-3\delta x^2} D'(1-K'(x)/k)\frac{K''(x)}{k},$$
			which is negative in a neighbourhood of $x=0$. Thus $(\alpha, \alpha)$ cannot be a solution to \eqref{int_max} for $\alpha \in (0, \hat\alpha)$. 
			\item[(B)] Assume now that $\alpha<\beta$. Then the right equation of \eqref{int_max} yields
			\begin{align*}
			1-\beta > \exp\bc{-3\delta \beta^2}D\bc{1-K'(\beta)/k} > \bc{1-3 \delta \beta^2} \bc{1- \frac{d}{k}K'(\beta)}.
			\end{align*}
			Now since $\vk \geq 3$, $K'(\beta) = O_\beta(\beta^2)$. But then the above equation yields a contradiction for $\beta$ small enough and thus $(\alpha, \beta) \in (0,\hat \alpha) \times (0,\hat \beta)$ with $\alpha < \beta$ is no possible solution. 
			\item[(C)] Finally, if $\alpha > \beta$, the left equation of \eqref{int_max} yields
			\begin{align*}
			d\bc{1-\alpha} > \exp\bc{-3\delta \alpha^2}D'\bc{1-K'(\alpha)/k} > d\bc{1-3 \delta \alpha^2} \bc{1- \frac{\Erw\brk{\vd^2}}{dk}K'(\alpha)}.
			\end{align*}			
			Now since $\vk \geq 3$, $K'(\beta) = O_\beta(\beta^2)$. But then the above equation yields a contradiction for $\beta$ small enough and thus $(\alpha, \beta) \in (0,\hat \alpha) \times (0,\hat \beta)$ with $\alpha > \beta$ is no possible solution.			\end{itemize} 
			Hence, \eqref{int_max} has no solution with $0<\alpha<\hat\alpha$.
		\item[Case 2: $\hat\alpha\leq\alpha<1$] because $\Phi(\alpha)<\Phi(0)$ for all $0<\alpha\leq1$, \eqref{int_max2} shows that we can choose $\delta$ small enough so that $\tilde\Phi_\delta(\alpha,\beta)<\tilde{\Phi}_\delta(0,0)$ for all $\alpha\geq\hat\alpha$ and all $\beta\in[0,1]$.
	\end{description}
	Combining both cases and recalling \eqref{eqlemma_aux_3}, we obtain \eqref{eqprop_aux1}.

\subsection{Proof of \Prop~\ref{prop_three}}\label{sec_proof_three}

Combining \Prop s~\ref{prop_auxphi} and~\ref{prop_aux}, we see that
\begin{align}\label{eqprop_three_1}
	\frac{1}{n} \Erw\brk{\nul(\vA\brk{n,\eps,\delta,\Theta})} & \leq 1-\frac{d}{k}-\delta+o_{\eps}(1).
\end{align}
The only (small) missing piece is that we still need to extend this result to the original random matrix $\AA_{[\lfloor\delta n\rfloor]}$ based on the simple random factor graph $\GG$.
To this end we apply the following lemma.

\begin{lemma}[{\cite[\Lem~4.8]{Maurice}}]\label{Lemma_JanesCoupling}\label{lem_epscoupling}
For any fixed $\Theta>0$ there exists a coupling of $\vA$ and $\vA\brk{n,\eps,0,\Theta}$ such that 
	$$\ex|\nul\vA-\nul\vA\brk{n,\eps,0,\Theta}|=O_{\eps}(\eps n). $$
\end{lemma}

Let $\A_{[\lfloor\delta n\rfloor]}$ be the matrix obtained from $\A$ by adding $\lfloor\delta n\rfloor$ random ternary equations.
Combining \eqref{eqprop_three_1} with \Lem~\ref{Lemma_JanesCoupling} and observing that each of the unary checks $p_i$ can alter the nullity by at most one, we obtain
\begin{align}\label{eqprop_three_2}
	\frac{1}{n} \Erw\brk{\nul(\vA_{[\lfloor \delta n\rfloor]})} & \leq 1-\frac{d}{k}-\delta+o(1).
\end{align}
Furthermore, since changing a single edge of the Tanner graph $\G$ or a single entry of $\A$ can change the rank by at most one, the Azuma--Hoeffding inequality shows that $\nul(\vA_{[\lfloor \delta n\rfloor]})$ is tightly concentrated.
Thus, \eqref{eqprop_three_2} implies
\begin{align}\label{eqprop_three_3}
	\pr\brk{\frac{1}{n}\nul(\vA_{[\lfloor \delta n\rfloor]}) \leq 1-\frac{d}{k}-\delta+o(1)}&=1-o(1/n).
\end{align}
Finally, combining \eqref{eqprop_three_3} with \Lem~\ref{lemma_contig}, we conclude that
	$$\pr\brk{\frac{1}{n}\nul(\AA_{[\lfloor \delta n\rfloor]}) \leq 1-\frac{d}{k}-\delta+o(1)}=1-o(1/n),$$
	which implies the assertion because $\nul(\AA_{[\lfloor \delta n\rfloor]})\leq n$ deterministically.

\section{Proof of \Prop~\ref{prop_overlap}}\label{sec_prop_overlap}

\noindent
We recall that $\vA$ is the random $\vm\times n$-matrix generated by way of the pairing model.
Moreover, we continue to let $\vA[n,\eps,\delta,\Theta]$ be the matrix from \Sec~\ref{sec_prop_three} with $\vm_\eps\sim\Po((1-\eps)dn/k)$ checks with independent degrees $\vk_i$, another $\vm_\delta\sim\Po(\delta n)$ ternary checks and further $\vec\theta\sim\unif([\Theta])$ unary checks (cf.~\eqref{eqms}).
Lemma \ref{lem_epscoupling} shows that the nullity of the second model approaches that of the first as $\eps\to0$.

We now go on to prove that if the matrix $\vA[\vec\theta_0]$ obtained from $\vA$ by adding a few random unary checks has many frozen coordinates, then the nullity of $\vA[n,\eps,\delta,\Theta]$ would be greater than permitted by \Prop~\ref{prop_three}; we use an argument similar to~\cite[proof of \Prop~2.7]{Ayre}.
Invoking \Cor~\ref{lem_pinning_RS} will then complete the proof of \Prop~\ref{prop_overlap}.

\begin{lemma}\label{Rank_too_small}
	Assume that for some $\Theta_0>0$ and $\vec\theta_0\sim\unif([\Theta_0])$ we have
\begin{align*}
	\limsup_{n \to \infty}\frac{1}{n}\ex\abs{\fF(\vA[\vec\theta_0])}>0.
\end{align*}
Then for any $\delta_0>0$ there exists $0<\delta<\delta_0$ and $\Theta_1=\Theta_1(\eps)$ such that for any $\Theta=\Theta(\eps)>\Theta_1(\eps)$ we have
\begin{align*}
		\limsup_{\eps\to0}\limsup_{n \to \infty} \frac{1}{n} \Erw\brk{\nul(\vA\brk{n,\eps, \delta, \Theta})} > 1-\frac{d}{k}-\delta.
\end{align*}
\end{lemma}
\begin{proof}
	For an integer $\ell\geq0$ obtain  $\vA_{[\ell]}[\vec\theta_0]$ from $\vA[\vec\theta_0]$ by adding $\ell$  random ternary equations.
	Moreover, let $\vec\lambda=\Po(\delta n)$.
	Since $\nul\vA_{[\vec\lambda]}[\vec\theta_0]\geq\nul\vA-\vec\lambda-\vec\theta_0$, \Lem~\ref{lem_epscoupling} implies that for any $\Theta=\Theta(\eps)$,
	\begin{align}\label{eqRank_too_small0}
		\ex\abs{\nul\vA_{[\vec\lambda]}[\vec\theta_0]-\nul\vA\brk{n,\eps,\delta,\Theta}}=\ex\abs{\nul\vA_{[\vec\lambda]}[\vec\theta_0]-\nul\vA\brk{n,\eps,0,\Theta}}+O_\eps(\delta n)&=O_{\eps}((\eps+\delta) n+\Theta)=O_\eps(\eps n).
	\end{align}
	We now estimate the nullity of $\vA_{[\vec\lambda]}[\vec\theta_0]$ under the assumption that for a large $n$,
	\begin{align}\label{eqRank_too_small1}
		\pr\brk{\abs{\fF(\vA[\vec\theta_0])}>\zeta n}&>\zeta&&\mbox{for some }\zeta>0.
	\end{align}
	Because adding equations can only increase the set of frozen variables, we have $\fF(\vA_{[\ell]}[\vec\theta_0])\subset \fF(\vA_{[\ell+1]}[\vec\theta_0])$ for all $\ell\geq0$.
Therefore, \eqref{eqRank_too_small1} implies that
	\begin{align}\label{eqRank_too_small3}
		\pr\brk{\abs{\fF(\vA_{[\ell]}[\vec\theta_0])}>\zeta n}&>\zeta&&\mbox{ for all }\ell\geq0.
	\end{align}

	We now claim that
	\begin{align}\label{eqRank_too_small4}
		\frac1n\ex\brk{\nul\vA_{[\vec\lambda]}[\vec\theta_0]}&\geq 1-d/k-\delta+\delta\zeta^4+o(1).
	\end{align}
	To prove \eqref{eqRank_too_small4} it suffices to show that for any $\ell\geq0$,
	\begin{align}\label{eqRank_too_small5}
		\ex\brk{\nul\vA_{[\ell+1]}[\vec\theta_0]-\nul\vA_{[\ell]}[\vec\theta_0]}&\geq  \zeta^4-1+o(1).
	\end{align}
	Indeed, we obtain \eqref{eqRank_too_small4} from \eqref{eqRank_too_small5} and the nullity formula $n^{-1}\ex[\nul\vA_{[0]}[\vec\theta_0]]=n^{-1}\ex[\nul\vA]+o(1)=1-d/k+o(1)$ from \eqref{eqMaurice} by writing a telescoping sum.

	To establish \eqref{eqRank_too_small5} we observe that $\nul\vA_{[\ell+1]}[\vec\theta_0]-\nul\vA_{[\ell]}[\vec\theta_0]\geq-1$ because we obtain $\vA_{[\ell+1]}[\vec\theta_0]$ from $\vA_{[\ell]}[\vec\theta_0]$ by adding a single ternary equation.
	Furthermore, if $|\fF(\vA_{[\ell]}[\vec\theta_0])|\geq\zeta n$, then with probability $\zeta^3+o(1)$ all three variables of the new ternary equation are frozen in $\vA_{[\ell]}[\vec\theta_0]$, in which case $\nul\vA_{[\ell+1]}[\vec\theta_0]=\nul\vA_{[\ell]}[\vec\theta_0]$.
	Hence, \eqref{eqRank_too_small4} follows from \eqref{eqRank_too_small5}, which follows from \eqref{eqRank_too_small3}.
	Finally, combining \eqref{eqRank_too_small0} and \eqref{eqRank_too_small4} completes the proof.
\end{proof}

\begin{proof}[Proof of \Prop~\ref{prop_overlap}]
	The proposition follows from \Cor~\ref{lem_pinning_RS} and \Lem~\ref{Rank_too_small}.
\end{proof}

\section{Proof of \Prop~\ref{prop_module}}\label{sec_prop_module}

\noindent
The proof proceeds very differently depending on whether the coefficients $\chi_1,\ldots,\chi_{k_0}$ are identical or not.
The following two lemmas summarise the analyses of the two cases.

\begin{lemma}\label{lemma_module}
	For any prime power $q$ and any $\chi\in\FF_q^*$ the $\ZZ$-module $\fM_q(\chi,\chi,\chi)$ possesses a basis $(\fb_1,\ldots,\fb_{q-1})$ of non-negative integer vectors $\fb_i\in\ZZ^{\FF_q^*}$ for all $i\in[q-1]$ such that
	\begin{align*}
		\|\fb_i\|_1\leq3\quad\mbox{ and }\quad\sum_{s\in\FF_q^*}\fb_{i,s}s=0\quad\mbox{for all $i\in[q-1]$, and}\quad\det\bc{\fb_1\ \cdots\ \fb_{q-1}}=q.
	\end{align*}
	Furthermore, for any $k_0>3$ we have
	$\fM_q\underbrace{\bc{\chi,\ldots,\chi}}_{\mbox{$k_0$ times}}=\fM_q(\chi,\chi,\chi).$
\end{lemma}

\begin{lemma}\label{lemma_module'}
	Suppose that $q$ is a prime power, that $k_0\geq3$ and that $\chi_1,\ldots,\chi_{k_0}\in\FF_q^*$ satisfy $|\{\chi_1,\ldots,\chi_{k_0}\}|\geq2$.
	Then $$\fM_q(\chi_1,\ldots,\chi_{k_0})=\ZZ^{\FF_q^*}.$$
	Furthermore, $\fM_q(\chi_1,\ldots,\chi_{k_0})$ possesses a basis $(\fb_1,\ldots,\fb_{q-1})$ of non-negative integer vectors $\fb_i\in\ZZ^{\FF_q^*}$ such that
	\begin{align*}
		\|\fb_i\|_1\leq3\quad\mbox{and}\quad\sum_{s\in\FF_q^*}\fb_{i,s}s=0\quad\mbox{for all }i\in[q-1].
	\end{align*}
\end{lemma}

Clearly, \Prop~\ref{prop_module} is an immediate consequence of \Lem s~\ref{lemma_module} and \ref{lemma_module'}.
We proceed to prove the former in \Sec~\ref{sec_module} and the latter in \Sec~\ref{sec_module'}.

\subsection{Proof of \Lem~\ref{lemma_module}}\label{sec_module}
Because we can just factor out any scalar, it suffices to consider the module
$$\fM=\fM_q\underbrace{(1,\ldots,1)}_{\mbox{$k_0$ times}}.$$
Being a $\ZZ$-module, $\fM$ is free, but it is not entirely self-evident that a basis with the additional properties stated in \Lem~\ref{lemma_module} exists.
Indeed, while it is easy enough to come up with $q-1$ linearly independent vectors in $\fM$ that all have an $\ell_1$-norm bounded by $3$, it is more difficult to show that these vectors generate $\fM$. 
In the proof of \Lem~\ref{lemma_module}, we sidestep this difficulty by working with two sets of vectors $\cB_1$ and $\cB_2$.
The first set $\cB_1$ is easily seen to generate $\fM$, while $\cB_2$ is a set of linearly independent vectors in $\fM$ with $\ell_1$-norms bounded by $3$. 
To argue that $\cB_2$ generates $\fM$, too, it then suffices to show that the determinant of the change of basis matrix equals one. 

To interpret the bases as subsets of $\ZZ^{q-1}$ rather than $\ZZ^{\FF_q^\ast}$ in the following, we fix some notation for the elements of $\FF_q$. 
Throughout this section, we let $q=p^\ell$ for a prime $p$ and $\ell \in \NN$. 
If $\ell=1$, we regard $\FF_q$ as the set $\{0, \ldots, p-1\}$ with \hspace{-0.25 cm} $\mod p$ arithmetic. 
If $\ell \geq 2$, the field elements can be written as
\begin{align*}
\{a_0 + a_1\var + \ldots + a_{\ell-1}\var^{\ell-1}: a_j \in \FF_p \text{ for } j = 0, \ldots, \ell-1\}, 
\end{align*}
with  \hspace{-0.3 cm} $\mod g(\var)$ arithmetic for a prime polynomial $g(\var)\in \FF_p[\var]$ of degree $\ell$. 
Exploiting this representation of the field elements as polynomials, we define the length len$(a_0 + a_1\var + \ldots + a_{\ell-1}\var^{\ell-1})$ of an element of $\FF_q$ to be the number of its non-zero coefficients. 
Finally, let
\begin{align}\label{eq_length}
\mathbb{F}_q^{(\geq 2)} = \cbc{h \in \FF_q: \text{len}(h) \geq 2}
\end{align}
be the set of all elements of $\FF_q$ with length at least two. Of course, if $\ell = 1$, $\FF_q^{(\geq 2)}$ is empty. 

Recall that we view $\fM$ as a subset of $\ZZ^{\FF_q^\ast}$ that is generated by the vectors
\begin{align*}
\bc{\sum_{i=1}^{k_0}\vecone\cbc{\sigma_i=s}}_{s\in\FF_q^*}, \quad \sigma\in\cS_0(1, \ldots, 1).
\end{align*}
In the above representation, the generators are indexed by $\FF_q^*$ rather than by the set $[q-1]$. But to carry out the determinant calculation, it is immensely useful to represent both $\cB_1$ and $\cB_2$ as matrices with a convenient structure. 
Hence, there is ambiguity in the choice of a bijection $f:\FF_q^\ast \to \{1, \ldots, q-1\}$ that maps the non-zero elements of $\FF_q$ to coordinates in $\ZZ^{\FF_q^*}$. 
To put a clear structure to the matrices in this subsection, we will soon choose $f$ in a particular way. With the above notation, we will from now on fix a bijection $f$ that is monotonically decreasing with respect to the length function on $\FF_q^\ast$: If len$(h_1) <$len$(h_2)$ for $h_1, h_2 \in \FF_q^\ast$, then $f(h_1) > f(h_2)$. 
More precisely, $f$ maps the $(p-1)^\ell$ elements in $\FF_q^\ast$ of maximal length $\ell$ to the interval $[(p-1)^\ell]$, the $\ell (p-1)^{\ell-1}$ elements of length $\ell-1$ to the interval $\{(p-1)^\ell+1, \ldots, (p-1)^\ell+\ell(p-1)^{\ell-1}\}$, and so on. 
For elements of length one, we further specify that
\begin{align*}
f(a \var^{i})= q-1-(\ell-i)(p-1) + a \quad \text{for } i \in \{0, \ldots \ell-1\} \text{ and } a \in [p-1].
\end{align*}
For our purposes, there is no need to fully specify the values of $f$ within sets of constant length greater than one, but one could follow the lexicographic order, for example. 
The benefit of such an ordering will become apparent in the next two subsections.

\subsubsection{First basis $\cB_1$}
The idea behind the first set $\cB_1$ is that it consists of vectors whose coordinates can be easily seen to correspond to element statistics of a valid solution while ignoring the $\ell_1$-restriction formulated in \Lem~\ref{lemma_module}. 
We build $\cB_1$ from frequency vectors of solutions of the form
\begin{align*}
\bc{a_0 + a_1\var + \ldots + a_{\ell-1}\var^{\ell-1}} + \sum_{i=0}^{\ell-1} a_i \cdot ((p-1)\var^{i}) = 0.
\end{align*}
That is, we take any element $a_0 + a_1\var + \ldots + a_{\ell-1}\var^{\ell-1}$ from $\FF_q^\ast$ and cancel it by a linear combination of elements from $\{p-1, (p-1)\var, \ldots, (p-1)\var^{\ell-1}\} \subset \FF_q^\ast$. 
Formally, let $e_1, \ldots, e_{q-1}$ denote the canonical basis of $\ZZ^{q-1}$. 
The set of statistics of all frequency vectors of the form described above then reads
\begin{align*}
\cB_1 =  \cbc{e_{f(\sum_{i=0}^{\ell-1} a_i\var^{i})} + \sum_{i=0}^{\ell-1} a_i e_{f(-\var^{i})}: \sum_{i=0}^{\ell-1} a_i\var^{i} \in \FF_q^\ast}.
\end{align*}
A moment of thought shows that $|\cB_1| = q-1$. 
Indeed, it is helpful to notice that for any $h \in \FF_q^\ast \setminus\{-1, \ldots, -\var^{\ell-1}\}$, there is exactly one element with a non-zero position in coordinate $f(h)$, and this coordinate is $1$. 
That is, there is basically exactly one element in $\cB_1$ associated with each element of $\FF_q^\ast$. 
Generally, the elements of $\cB_1$ can be ordered to yield a lower triangular matrix $M_q$. 
To sketch this matrix, we first consider the case $\ell=1$.  
In this case, with our choice of indexing function $f$, the elements of $\B_1$ can be ordered to give the matrix displayed in Figure~\ref{fig_Mp}.
For the case of fields of prime order, this basis is already implicitly mentioned in~\cite{Huang}.

\begin{figure}
\begin{align*}
M_p= \begin{pNiceArray}{cccccccc}[first-row, first-col]
		& 1& 2 &       &. &. &. &. &  p-1 \\
	1	&1 &  &  &  &  &  &  & \\
	2	 && 1 &  &  &  &&& \\
		& &  & 1 &  &  & & &\\
		& &  &  & \ddots &  & &&\\
	\vdots	& &  &  & &  &  &  & \\
		& &  &  &  &  & \ddots&& \\
		& &  &  &  &  & & 1&\\
	p-1	& 1 & 2 & 3 &  & \cdots &  & p-2 & p
		\end{pNiceArray}.
\end{align*}
\caption{The matrix $M_p$.}\label{fig_Mp}
\end{figure}

Note that this reduces to $M_2 = (2)$ in the case $p=2$. 
In this representation, rows are indexed by the field elements they represent, while columns are indexed by the field elements they are associated with. 
For $\ell \geq 2$, we can use the matrix $M_p$ for the compact representation of $M_q$ displayed in Figure~\ref{fig_Mq}.

\begin{figure}
\begin{align}
M_q = \begin{pNiceArray}{cccccccc|ccccccccccc}[first-row, first-col]
& & & & &\rotate{\FF_q^{(\geq 2)}} & & & & \rotate 1&\hdots &\rotate{p-1} &\rotate{\var} &\hdots&\rotate{(p-1)\var}&\hdots&\rotate{\var^{\ell-1}} &\hdots&\rotate{(p-1)\var^{\ell-1}}\\
&1 &  &  &  &  &  &  & & & & &&&&&&&\\
&		 & 1 &  &  &  &&& & & & & &&&&&&\\
&		 &  & \ddots &  &  & & && & & &\Block{4-5}{\bigzero} &&&&&&\\
&		 &  &  & 1 &  & &&& & & & & &&&&&\\
\FF_q^{(\geq 2)}&		 &  &  &  & 1 &  &  & & & & &&&&&&&\\
&		 &  &   & &  &1  &  & & & & &&&&&&&\\
&		 &  &  & &   & &  & & & & &&&&&&&\\
&		 &  &  & & & &  & 1 & & & &&&&&&&\\
		\hline
1 &		0 &  & \Cdots & &  && &0  & \Block[draw]{3-3}{M_p}  & &&&&&\Block{3-5}{\bigzero}&&\\
 \vdots &		0& & \Cdots &  &  & && 0 &  && &&&&&&&\\
 p-1 &		 \ast& & \Cdots &   && &&\ast & & & & &&&&&&\\
  \var &		 0 & & \Cdots &  & &&&0 & && & \Block[draw]{3-3}{M_p}  &&&&&&\\
      \vdots &	0 & & \Cdots &  &  & &&0 & \Block{6-3}{\bigzero}& & & &  &&&&&\\  
 (p-1)\var &		\ast & & \Cdots & &  &&& \ast & &   & & && &&&&\\
 \vdots &		\ast & & \Cdots & & &&&\ast & &   & & && &\ddots &&&\\
 \var^{\ell-1}&  0& & \Cdots & &  &&&0  & &   & & && && \Block[draw]{3-3}{M_p} &&\\
       \vdots &	0	 & & \Cdots &  &  & & & 0 & && & &&&&&&  &\\
         (p-1)\var^{\ell-1 }&	\ast	&  & \Cdots &  &   &  & &\ast  && & &&&&& && 
	\end{pNiceArray}.
\end{align}
	\caption{The matrix $M_q$ for $\ell\geq2$.}\label{fig_Mq}
\end{figure}

In the matrix $M_q$, the upper left block is an identity matrix of the appropriate dimension, the upper right is a zero matrix, the lower left is a matrix that only has non-zero entries in rows that correspond to $-1, \ldots, -\var^{\ell-1}$ while the lower right is a block diagonal matrix whose blocks are given by $M_p$. 
In particular, $M_p$ is a lower triangular matrix.
Because $M_p$ has determinant $p$ the following is immediate.

\begin{claim}\label{claim_detM}
We have $ \det(M_q) = p^\ell = q.  $
\end{claim}

Let $\fB_1$ denote the $\ZZ$-module generated by the elements of $\cB_1$. 
Then the lower triangular structure of $M_q$ also implies the following.

\begin{claim}
The rank of $\cB_1$ is $q-1$.
\end{claim}

The following lemma shows that the module $\fM$ is contained in $\fB_1$.
\begin{lemma}
The  $\ZZ$-module $\fM$ is contained in the $\ZZ$-module $\fB_1$. 
\end{lemma}
\begin{proof}
We show that each element of $\fM$ can be written as a linear combination of elements of $\cB_1$. 
To this end it is sufficient to show that every frequency vector of a solution to an equation with exactly $k_0$ non-zero entries and all-one coefficients can be written as a linear combination of the elements of $\cB_1$. 
Let thus $x \in \NN^{q-1}$ be such a frequency vector, that is $\| x \|_1 \leq k_0$ and $\sum_{i=1}^{q-1} x_i f^{-1}(i) = 0$ in $\FF_q$. 
Before we state a linear combination of $x$ in terms of $\cB_1$, observe that for each $j \in [q-1] \setminus \{q-1-(\ell-1)(p-1), q-1-(\ell-2)(p-1), \ldots, q-1\}$, there is exactly one basis vector with a non-zero entry in position $j$. 
Moreover, the entry of this basis vector in position $j$ is $1$. 
On the other hand, the basis vectors corresponding to the remaining $\ell$ columns $q-1-(\ell-1)(p-1), q-1-(\ell-2)(p-1), \ldots, q-1$ of $M_q$ are actually integer multiples of the standard unit vectors, as
\begin{align*}
e_{f((p-1)\var^{i})} + (p-1) e_{f(-\var^{i})} = p e_{f((p-1)\var^{i})}
\end{align*}
for $i=0, \ldots, \ell-1$. 
With these observations, the only valid candidate for a linear combination of $x$ in terms of the elements of $\cB_1$ is given by
\begin{align*}
	x &= \sum_{\sum_{i=0}^{\ell-1} a_i\var^{i} \in \FF_q^\ast \setminus \{-1,\ldots, -\var^{\ell-1}\}} x_{f(\sum_{i=0}^{\ell-1} a_i\var^{i})}\bc{e_{f(\sum_{i=0}^{\ell-1} a_i\var^{i})} + \sum_{j=0}^{\ell-1} a_j e_{f(-\var^{j})} }\\
	  & \qquad \qquad + \sum_{j=0}^{\ell-1}\frac{x_{f(-\var^{j})} -  \sum_{\sum_{i=0}^{\ell-1} a_i\var^{i} \in \FF_q^\ast \setminus\{-1, \ldots, -\var^{\ell-1}\}} a_j x_{f(\sum_{i=0}^{\ell-1} a_i\var^{i})}}{p} \cdot pe_{f(-\var^{j})}.
\end{align*}
It remains to argue why the coefficients of the basis vectors $pe_{f(-1)}, \ldots, pe_{f(-\var^{\ell-1})}$ in the second sum are integers.
At this point, we will use that $x$ is a solution statistic: Because 
\begin{align*}
\sum_{\sum_{i=0}^{\ell-1} a_i\var^{i} \in \FF_q^\ast} x_{f(\sum_{i=0}^{\ell-1} a_i\var^{i} )} \sum_{j=0}^{\ell-1} a_j\var^{j} = 0 \qquad \text{in} \quad \FF_q
\end{align*}
 and the additive group $(\FF_q,+)$ is isomorphic to $((\FF_p)^\ell, +)$, all ``components'' in the above sum must be zero and thus
\begin{align*}
\sum_{\sum_{i=0}^{\ell-1} a_i\var^{i} \in \FF_q^\ast} x_{f(\sum_{i=0}^{\ell-1} a_i\var^{i} )}  a_j= 0 \qquad \text{in} \quad  \FF_p
\end{align*}
for all $j=0, \ldots, \ell-1$. 
However, isolating the contribution from $\{-1, \ldots, -\var^{\ell-1}\}$ yields
\begin{align}\label{eq_division}
0 = \sum_{\sum_{i=0}^{\ell-1} a_i\var^{i} \in \FF_q^\ast} x_{f(\sum_{i=0}^{\ell-1} a_i\var^{i} )}  a_j =  - x_{f(-\var^{j})} + \sum_{\sum_{i=0}^{\ell-1} a_i\var^{i} \in \FF_q^\ast \setminus \{-1, \ldots, -\var^{\ell-1}\}} a_j x_{f(\sum_{i=0}^{\ell-1} a_i\var^{i} )} \qquad \text{in} \quad  \FF_p,
\end{align}
as the coefficient $a_j$ of $\var^{j}$ in $-\var^{i}$ is zero unless $i=j$. 
Therefore, the right hand side in (\ref{eq_division}) is divisible by $p$ and the claim follows.
\end{proof}

\subsubsection{Second basis $\cB_2$}
In this subsection, we define a candidate set for the vectors $(\fb_1, \ldots, \fb_{q-1})$ in the statement of \Lem~\ref{lemma_module}. 
That is, we define a set $\cB_2$ all whose elements have non-negative components and $\ell_1$-norm at most three. 
In other words, we are looking for solutions to 
\begin{align}\label{eq_eq}
x_1 + \ldots + x_{k_0} = 0
\end{align}
with at most three different non-zero components. 

Here again, our construction basically associates one basis vector to each element of $\FF_q^\ast$. 
However, due to the $\ell_1$-restriction, there is less freedom in choosing the remaining non zero-coordinates. 
Our approach to design a set that satisfies this restriction while retaining a representation in a convenient block lower triangular matrix structure is to distinguish between elements of length one and of length at least two. 
We will therefore construct $\cB_2$ via two sets $\cB^{(1)}$ and $\cB^{(\geq 2)}$ such that $\cB_2$ is given as 
\begin{equation}
\cB_2 = \cB^{(1)} \cup \cB^{(\geq 2)}.
\end{equation}
Let us start with an element $h = \sum_{i=0}^{\ell-1}a_i\var^{i}$ of length at least two in $\FF_q$. 
Assume that its leading coefficient is $a_r$ for $r \in [\ell-1]$. 
If a variable in (\ref{eq_eq}) takes value $h$, we may cancel its contribution to an equation by subtracting the two elements $a_r\var^r$ and $h - a_r\var^r$, both of which are shorter than $h$:
\begin{align*}
 \sum_{i=0}^{\ell-1} a_i\var^{i} - a_r\var^r - \bc{\sum_{i=0}^{\ell-1}a_i\var^{i} - a_r\var^r}=0.
 \end{align*}
This solution corresponds to the vector
 \begin{align*}
 e_{f(h)} + e_{f(-a_r\var^r)} + e_{f(-h+a_r\var^r)}.
 \end{align*}
This idea for field elements $h \in \FF_q^{(\geq 2)}$ of length at least two then yields the $q-1-\ell(p-1)$ integer vectors
 \begin{align*}
 \cB^{(\geq 2)} =  \cbc{e_{f(h)} + e_{f(-a_r\var^r)} + e_{f(-h+a_r\var^r)}: r \in [\ell-1] \text{ and } h = \sum_{i=0}^{r}a_i \var^{i} \in \FF_q^{(\geq 2)} \text{ with } a_r\not=0}.
 \end{align*}
 
For a field element $h$ of length one, an analogous shortening operation would correspond to the vector
\begin{align*}
e_{f(h)} + e_{f(-h)}.
\end{align*}
If $p=2$, this procedure applied to all field elements of length one yields $\ell$ distinct vectors and we are done. 
However, if $p > 2$, employing this idea for all elements of length one would only lead to $\ell(p-1)/2$ rather than $\ell(p-1)$ additional vectors, as $h$ and $-h$ are distinct and obviously give rise to the same statistic. 
As a consequence, for $p>2$, we need to deviate from the above construction and come up with a modified ``short-solution'' scheme. 
Let $h = a_r\var^r$ be an element of length one. 
If $a_r \in \{1, \ldots (p-1)/2\}$, we simply associate the vector $e_{f(h)} + e_{f(-h)}$ to it, as indicated. 
If on the other hand  $a_r \in \{(p+1)/2, \ldots, p-1\}$, we let $h$ correspond to the vector
\begin{align*}
e_{f(h)} + e_{f(-\var^r)} + e_{f(-h+\var^r)}.
\end{align*}
With this, for $p>2$, the part of $\cB_2$ that corresponds to field elements of length one is given by the set
\begin{align}
\cB^{(1)}= \bigcup_{r=0}^{\ell-1}\bc{\cbc{e_{f(a_r \var^r)} + e_{f(-a_r\var^r)}: a_r \in [(p-1)/2]} \cup \cbc{e_{f(-a_r\var^r)} + e_{f(-\var^r)} + e_{f(a_r\var^r+\var^r)}: a_r \in [(p-1)/2]}}.
\end{align}
If $p=2$, in line with the above discussion, we simply let
\begin{align}
\cB^{(1)}= \bigcup_{r=0}^{\ell-1}\cbc{2e_{f( \var^r)} }.
\end{align}
Again, a moment of thought shows that in any case, $|\cB_2| = |\cB_1| = q-1$. 
Let $\fB_2$ denote the $\ZZ$-module generated by the elements of $\cB_2$. 
Our choice of $\cB_2$ has the advantage that again, its elements may be represented in a block lower triangular matrix. 
For this representation, it is instructive to consider the case $\ell=1$ first. 
In this case and with our choice of $f$, the elements of $\cB_2$ can be arranged as the columns of a matrix $A_p$ as in Figure~\ref{fig_Ap}.

\begin{figure}
\begin{align}\label{A_p}
A_p= \begin{pmatrix}
		1 & 0 & \cdots & \cdots & 0 & 0 & \cdots &\cdots & \cdots & 0\\
		0 & \ddots & \ddots &  & \vdots &\vdots  & &  & \reflectbox{$\ddots$}& 1\\
		\vdots &\ddots  &\ddots &\ddots  &\vdots   & \vdots  &  & \reflectbox{$\ddots$} & \reflectbox{$\ddots$}& 0  \\
		 \vdots&  & \ddots&  \ddots & 0 & \vdots & \reflectbox{$\ddots$} &  \reflectbox{$\ddots$}& \reflectbox{$\ddots$}  & \vdots\\
		0 & \cdots &\cdots &  0 & 1 & 0 & 1&0& \cdots & 0 \\
		0 & \cdots & \cdots &  0 & 1& 2 &0  & \cdots  &\cdots & 0\\
		\vdots &  & \reflectbox{$\ddots$}  & \reflectbox{$\ddots$} & 0 & 0 & 1& \ddots & &\vdots \\
		\vdots & \reflectbox{$\ddots$} & \reflectbox{$\ddots$} & \reflectbox{$\ddots$} & \vdots  & \vdots  &\ddots  & \ddots& \ddots & \vdots\\
		0 & \reflectbox{$\ddots$} & \reflectbox{$\ddots$}   & & \vdots & 0 & \cdots &0 &1  &0 \\
		1 & 0 & \cdots  & \cdots &0&1& \cdots & \cdots & 1 &  2
		\end{pmatrix}.
\end{align}	\caption{The matrix $A_p$.}\label{fig_Ap}
\end{figure}

Here, as in the construction of $M_p$, column $i$ corresponds to the unique vector associated to $i \in \FF_q$. In the special case $p=2$, this matrix reduces to 
\begin{align*}
A_2 = (2).
\end{align*}
For $\ell \geq 2$, the elements of $\cB_2$ may then be visualised in the matrix from Figure~\ref{fig_Aq}.

\begin{figure}
\begin{align}\label{A_q}
A_q = \begin{pNiceArray}{cccccccc|ccccccccccc}[first-row, first-col]
& & & & \rotate{\FF_q^{(\geq 2)}} & & & & &\rotate 1 &\cdots &\rotate{p-1} &\rotate{\var} &\cdots&\rotate{(p-1)\var}&\cdots&\rotate{\var^{\ell-1}} &\cdots&\rotate{(p-1)\var^{\ell-1}}\\
&1 &  &  &  &  &  &  & & & & &&&&&&&\\
&	\ast	 & 1 &  &  &  &&& & & & & &&&&&&\\
&		\ast & \ast & \ddots &  &  & & && & & &\Block{4-5}{\bigzero} &&&&&&\\
&	\ast	 & \cdots &  \ast& 1 &  & &&& & & & & &&&&&\\
\FF_q^{(\geq 2)}&\ast & \Cdots &  & \ast & 1 &  &  & & & & &&&&&&&\\
&	\ast	 &  & \Cdots & & \ast &1  &  & & & & &&&&&&&\\
&	\ast	 &  &  & \Cdots&   & \ast & \ddots & & & & &&&&&&&\\
&	\ast	 &  &  &\Cdots &  & &\ast  & 1 & & & &&&&&&&\\
		\hline
1 &	\ast	 &  & \Cdots &  &   && & \ast & \Block[draw]{3-3}{A_p}  & &&&&&\Block{3-5}{\bigzero}&&\\
 \vdots &	\ast	& & \Cdots &  &  & &&\ast &  && &&&&&&&\\
 p-1 &	\ast	 & & \Cdots &   &  & && \ast& & & & &&&&&&\\
  \var&	\ast	 & & \Cdots &  &  &&& \ast & && & \Block[draw]{3-3}{A_p}  &&&&&&\\
      \vdots & \ast	 & & \Cdots &  &   & &&\ast & \Block{6-3}{\bigzero}& & & &  &&&&&\\  
(p-1)\var &	\ast	 & & \Cdots & &  &&& \ast & &   & & && &&&&\\
 \vdots &	\ast	 & & \Cdots & &  &&&\ast & &   & & && &\ddots &&&\\
\var^{\ell-1}&  & & \Cdots & &  &&&\ast & &   & & && && \Block[draw]{3-3}{A_p} &&\\
       \vdots &	\ast	 & & \Cdots &&  & & &\ast& && & &&&&&&  &\\
       (p-1)\var^{\ell-1 }&	\ast	&  & \Cdots &  &  &  & & \ast && & &&&&& && 
	\end{pNiceArray}.
\end{align}
\caption{The matrix $A_q$ for $\ell\geq2$.}\label{fig_Aq}
\end{figure}
In $A_q$, column $i \in [q-1]$ corresponds to the unique vector that is associated with the field element $f^{-1}(i)$.
Moreover, at this point, a moment of appreciation of our indexing choice $f$ is in place: Because $f$ is monotonically decreasing with respect to length, there are no entries above the diagonal in the first $|\FF_q^{(\geq 2)}|$ columns, as we only cancel field elements by strictly shorter ones. 
Moreover, the remaining $\ell(p-1)$ columns are governed by a simple block structure. 
As a concrete example, \eqref{A_p} with $p=7$ reads
\begin{equation*}
A_7 = \begin{pmatrix}
		1 & 0 & 0 & 0 & 0 & 0\\
		0 & 1 & 0 & 0 & 0 & 1\\
		0 & 0 & 1 & 0 & 1 & 0\\
		0 & 0 & 1 & 2 & 0 & 0\\
		0 & 1 & 0 & 0 & 1 & 0\\
		1 & 0 & 0 & 1 & 1 & 2
		\end{pmatrix}
\end{equation*} and $A_7$ would be used as a block matrix in any field of order $7^\ell$ as shown in \eqref{A_q}.

As each element of $\cB_2$ corresponds to a solution with at most $3 \leq k_0$ non-zero components, we obtain the following.
\begin{claim}
The  $\ZZ$-module $\fB_2$ is contained in the $\ZZ$-module $\fM$. 
\end{claim}

Thus far we know $ \fB_2 \subseteq \fM \subseteq \fB_1.  $
Moreover, $\cB_2$ has the desired $\ell_1$-property. 
On the other hand, in comparison to $\cB_1$, it is less clear that $\cB_2$ generates $\fM.$ 
It thus remains to show that in fact $\fB_2 = \fB_1$. 
We will do so by using the following fact, which is an immediate consequence of the adjugate matrix representation of the inverse matrix.

\begin{fact}\label{Lem_Conrad}
If $M$ is a free $\ZZ$-module with basis $x_1, \ldots, x_n$, a set of elements $y_1, \ldots, y_n\in M$ is a basis of $M$ if and only if the change of basis matrix $(c_{ij})$ has determinant $\pm1$.
\end{fact}

We will apply Fact \ref{Lem_Conrad} to $M=\fB_1$ with $\{x_1, \ldots, x_n\} = \cB_1$ and $\{y_1, \ldots, y_n\} = \cB_2$. 
Let $C_q \in \ZZ^{(q-1) \times (q-1)}$ be the matrix whose entries comprise the coefficients when we express the elements of $\cB_2$ by $\cB_1$ (recall that $\fB_2 \subseteq \fB_1$) when we order the elements of $\cB_1, \cB_2$ as done in the construction of $M_q$ and $A_q$. 
Thus $A_q = M_q C_q.$
As 
\begin{align*}
\det(A_q) = \det(M_q C_q) = \det(M_q) \cdot \det (C_q),
\end{align*}
we do not need to compute $C_q$ explicitly to apply Fact \ref{Lem_Conrad}, but instead it suffices to compute $\det(M_q)$ and $\det(A_q)$. 
From Claim \ref{claim_detM}, $\det(M_q)$ is already known. 
Moreover, for $A_q$, the computation will not be too hard, as $A_q$ is a block lower triangular matrix.
Therefore, we are just left to calculate the determinant of the non-trivial diagonal blocks.

\begin{lemma} \label{Lem_detAp}
For any prime $p$ we have $\det(A_p) = p.$
\end{lemma}
\begin{proof}
The case $p=2$ is immediate.
We thus assume that $p>2$ in the following. 
We transform $A_p$ into a lower triangular matrix by elementary column operations. 
To this end, let $a_1, \ldots, a_{q-1}$ be the columns of $A_p$. 
The first $(p+1)/2$ columns already have the right form, so we do not alter this part of the matrix. 
For any $j=(p+3)/2, \ldots, p-1$, subtract column $a_{p+1-j}$ from column $a_j$. 
This yields the matrix
\begin{align*}
\begin{pmatrix}
		1 & 0 & 0 & 0 & 0 & 0 & 0 & 0\\
		0 & 1 & 0 & 0 & 0 &0&0& 0\\
		0 & 0 & \ddots & 0 & 0 & 0&  \reflectbox{$\ddots$}&0\\
		0 & 0 & 0 & 1 & 0 & 0&0&0 \\
		0 & 0 & 0 & 1& 2 &-1  & 0 & 0\\
		0 & 0 & \reflectbox{$\ddots$} & 0 & 0 & 1&-1&0 \\
		0 & 1 & 0 & 0 & 0 & 0&\ddots&-1 \\
		1 & 0 & 0 &0&1& 1 & 1 & 2
		\end{pmatrix}.
\end{align*}
Next, we swap column $(p+1)/2$ successively with columns $(p+3)/2, \ldots$ up to $p-1$, yielding
\begin{align*}
\begin{pmatrix}
		1 & 0 & 0 & 0 & 0 & 0 & 0 & 0\\
		0 & 1 & 0 & 0 & 0 &0&0& 0\\
		0 & 0 & \ddots & 0 & 0 & 0&  \reflectbox{$\ddots$}&0\\
		0 & 0 & 0 & 1 & 0 & 0&0&0 \\
		0 & 0 & 0 & 1& -1  & 0 & 0& 2\\
		0 & 0 & \reflectbox{$\ddots$} & 0 & 1&-1&0 & 0 \\
		0 & 1 & 0 & 0 & 0 & 0&\ddots& 0 \\
		1 & 0 & 0 &0&1& 1 & 2 & 1
		\end{pmatrix}.
\end{align*}
This changes the determinant by a factor of $(-1)^{(p-3)/2}$. 
Finally, in order to erase the entry $2$ in row $(p+1)/2$ and column $p-1$, we add twice the sum of columns $(p+1)/2, \ldots, p-2$ to column $p-1$. 
We thus obtain the matrix
\begin{align*}
\begin{pmatrix}
		1 & 0 & 0 & 0 & 0 & 0 & 0 & 0\\
		0 & 1 & 0 & 0 & 0 &0&0& 0\\
		0 & 0 & \ddots & 0 & 0 & 0&  \reflectbox{$\ddots$}&0\\
		0 & 0 & 0 & 1 & 0 & 0&0&0 \\
		0 & 0 & 0 & 1& -1  & 0 & 0& 0\\
		0 & 0 & \reflectbox{$\ddots$} & 0 & 1&-1&0 & 0 \\
		0 & 1 & 0 & 0 & 0 & 0&\ddots& 0 \\
		1 & 0 & 0 &0&1& 1 & 2 & p
		\end{pmatrix}.
\end{align*}
with determinant $(-1)^{(p-3)/2}p$. 
Multiplying with $(-1)^{(p-3)/2}$ from the column swaps yields the claim.
\end{proof}

\begin{corollary} \label{Cor_detAq}
For any prime $p$ and $\ell \geq 1$, we have $ \det(A_q) = q.  $
\end{corollary}

Finally, Claim \ref{claim_detM} and Corollary \ref{Cor_detAq} imply that $\det(C_q) = 1.$
Thus, by Fact \ref{Lem_Conrad}, $\cB_2$ is a basis of $\fB_1$, which implies that $ \fB_1 = \fB_2 = \fM.  $
The column vectors $\fb_1,\ldots,\fb_{q-1}$ of $A_q$ therefore enjoy the properties stated in \Lem~\ref{lemma_module}.

\subsection{Proof of \Lem~\ref{lemma_module'}}\label{sec_module'}

Assume w.l.o.g.\ that $\chi_1=1$. Moreover, by assumption, the set $\{\chi_1, \ldots, \chi_{k_0}\}$ contains at least two different elements, and so we may also assume that $\chi_3 \not= 1$ (recall that $k_0 \geq 3$).


We define $(\fb_1, \ldots, \fb_{q-1})$ by distinguishing between three cases:

\noindent
\textbf{Case 1:} $p=2$ and $\chi_2=1$.

Denote the order of $\chi_3^{-1}$ in $(\FF_q^\ast, \cdot)$ by $\fo$, so that the elements $1, \chi_3^{-1}, \ldots, \chi_3^{-(\fo-1)}$ are pairwise distinct. Since $p=2$ and $\fo \mid q-1$, $\fo$ is an odd number. Moreover, because $\chi_3^{-1} \not=1$, $\fo \geq 3$.  
We now partition $\FF_q^\ast$ into orbits of the action of $(\{1, \chi_3^{-1}, \ldots, \chi_3^{-(\fo-1)} \}, \cdot)$ on $\FF_q^\ast$ such that 
$$\FF_q^\ast =\dot \bigcup_{j=1}^{(q-1)/\fo} \fO_j,$$ 
where each orbit $\fO_j$ contains exactly $\fo$ elements. Suppose that $\fO_j=\{g^{(j)}_1, \ldots, g^{(j)}_{\fo}\}$, where the elements are indexed such that $g_{i+1}^{(j)} = \chi_3^{-1}g_{i}^{(j)}$. 

To each $\fO_j$, we associate a set of potential basis vectors whose union over different $j$ then yields the full set $(\fb_1, \ldots, \fb_{q-1})$. More precisely, the set corresponding to $\fO_j$ is defined as
\begin{align*}
	\fB_j = \bigcup_{i=1}^{\fo-1} \cbc{e_{g^{(j)}_i}+ e_{g^{(j)}_{i+1}}} \cup  \cbc{e_{g_{1}^{(j)}} +  e_{g_{2}^{(j)}} + e_{g_{2}^{(j)}+g_3^{(j)}}} .
\end{align*}
In this definition, we have used that for $\chi_1 = -\chi_2=1$ and any $h \in \FF_q$,
\begin{align*}
	\chi_1 \cdot h + \chi_2 \cdot 0  + \chi_3 \cdot \chi_3^{-1}h  = 0 \qquad \text{as well as} \qquad \chi_1 \cdot h + \chi_2 \cdot \chi_3^{-1}h  + \chi_3 \cdot (\chi_3^{-1}h+\chi_3^{-2}h)  = 0.
\end{align*}
Note that the element
\begin{align*}
	g_2^{(j)}+g_3^{(j)} = (1+ \chi_3^{-1})g_2^{(j)}
\end{align*}
is nonzero and distinct from both $g_2^{(j)}$ and $g_3^{(j)}$. It might be one of $g_{1}^{(j)},g_4^{(j)}, \ldots, g_{\fo}^{(j)}$.

We next argue that the union of the different $\fB_j$ generates $\ZZ^{\FF_q^\ast}$. By linear transformation and using that $\fo$ is odd, $\fB_j$ has the same span as
\begin{align*}
	\cbc{e_{g^{(j)}_1} + e_{g_2^{(j)}}, e_{g^{(j)}_1} - e_{g_3^{(j)}}, e_{g^{(j)}_1} + e_{g_4^{(j)}}, \ldots, e_{g^{(j)}_1} - e_{g_{\fo}^{(j)}}} \cup  \cbc{e_{g_{1}^{(j)}+g_2^{(j)}}}. 
\end{align*}
Now, there are two cases.
\begin{enumerate}
\item For all $j \in [(q-1)/\fo]$, $g_2^{(j)}+g_3^{(j)} \in \{g_{1}^{(j)}, g_4^{(j)}, \ldots, g_{\fo}^{(j)}\}$. In this case, either $e_{g_2^{(j)}+g_3^{(j)}} =e_{g_{1}^{(j)}}$, or we can subtract $e_{g_2^{(j)}+g_3^{(j)}}$ from or add it to the element $e_{g^{(j)}_1} \pm e_{g_2^{(j)}+g_3^{(j)}}$ to obtain $e_{g_{1}^{(j)}}$. After isolating $e_{g_1^{(j)}}$, a straightforward linear transformation yields a set of $\fo$ distinct unit vectors whose non-zero components are given by $\fO_j$. 
Thus, the union over all $\fB_j$ constitutes a set of linearly independent elements that generates $\ZZ^{\FF_q^*}$. 
\item For all $j \in [(q-1)/\fo]$, $g_2^{(j)}+g_3^{(j)} \notin \{g_1^{(j)}, g_4^{(j)}, \ldots, g_{\fo}^{(j)}\}$. In this case, consider the union $\bigcup_{j=1}^{(q-1)/\fo}\fB_j$, which has the same span as
\begin{align*}
	\bigcup_{j=1}^{(q-1)/\fo} \cbc{e_{g^{(j)}_1} + e_{g_2^{(j)}}, e_{g^{(j)}_1} - e_{g_3^{(j)}}, e_{g^{(j)}_1} + e_{g_4^{(j)}}, \ldots, e_{g^{(j)}_1} - e_{g_{\fo}^{(j)}}} \cup  \cbc{e_{g_{1}^{(j)}+g_2^{(j)}}}.
\end{align*}
Since for each $j$, the element $g_1^{(j)}+g_2^{(j)}$ must be contained in some $\fO_{j'}$ for $j\not=j'$, as in case (1), $e_{g_{1}^{(j)}+g_2^{(j)}}$ can be used to isolate $e_{g_1^{(j')}}$. After isolating $e_{g_1^{(j')}}$ for all $j'$, these elements can be straightforwardly used to linearly transform the union over all $\fB_j$ into the standard basis $(e_{h})_{h \in \FF_q^\ast}$ of $\ZZ^{\FF_q^*}$.
\end{enumerate}
Finally, set $\bigcup_{j=1}^{(q-1)/\fo} \fB_j = \{\fb_1, \ldots, \fb_{q-1}\}$.

\textbf{Case 2:} $p \not=2$ and $\chi_2=-1$. 

We proceed almost exactly as before, only the choice of the ``acyclic'' basis vectors is different: 

Denote the order of $\chi_3^{-1}$ in $(\FF_q^\ast, \cdot)$ by $\fo$, so that the elements $1, \chi_3^{-1}, \ldots, \chi_3^{-(\fo-1)}$ are pairwise distinct. Then $\fo \mid q-1$, and since $\chi_3^{-1} \not=1$, $\fo \geq 2$. 
We now partition $\FF_q^\ast$ into orbits of the action of $(\{1, \chi_3^{-1}, \ldots, \chi_3^{-(\fo-1)} \}, \cdot)$ on $\FF_q^\ast$ such that 
$$\FF_q^\ast =\dot \bigcup_{j=1}^{(q-1)/\fo} \fO_j,$$ 
where each orbit $\fO_j$ contains exactly $\fo$ elements. Suppose that $\fO_j=\{g^{(j)}_1, \ldots, g^{(j)}_{\fo}\}$, where the elements are indexed such that $g_{i+1}^{(j)} = \chi_3^{-1}g_{i}^{(j)}$. 

To each $\fO_j$, we associate a set of potential basis vectors whose union over different $j$ then yields the full set $(\fb_1, \ldots, \fb_{q-1})$. More precisely, the set corresponding to $\fO_j$ is defined as
\begin{align*}
	\fB_j = \bigcup_{i=1}^{\fo-1} \cbc{e_{g^{(j)}_i}+ e_{g^{(j)}_{i+1}}} \cup  \cbc{e_{g_{1}^{(j)}} +  e_{g_{2}^{(j)}} + e_{2g_{1}^{(j)}}} .
\end{align*}
Here, we have used that for $\chi_1 = -\chi_2=1$ and $p \not=2$,
\begin{align*}
	\chi_1 \cdot 0 + \chi_2 \cdot h  + \chi_3 \cdot \chi_3^{-1}h  = 0 \qquad \text{and} \qquad \chi_1 \cdot h + \chi_2 \cdot 2 h  + \chi_3 \cdot \chi_3^{-1}h  = 0.
\end{align*}
Note that the element $2g_1^{(j)}$
is distinct from $g_1^{(j)}$. It might be one of $g_2^{(j)}, \ldots, g_{\fo}^{(j)}$.

We next argue that the union of the different $\fB_j$ generates $\ZZ^{\FF_q^\ast}$. By linear transformation, $\fB_j$ has the same span as
\begin{align*}
	\cbc{e_{g^{(j)}_1} + e_{g_2^{(j)}}, e_{g^{(j)}_1} - e_{g_3^{(j)}}, e_{g^{(j)}_1} + e_{g_4^{(j)}}, \ldots, e_{g^{(j)}_1} \pm e_{g_{\fo}^{(j)}}} \cup  \cbc{e_{2g_{1}^{(j)}}}. 
\end{align*}
Now, there are two cases.
\begin{enumerate}
\item For all $j \in [(q-1)/\fo]$, $2g_1^{(j)} \in \{g_2^{(j)}, \ldots, g_{\fo}^{(j)}\}$. As in case 1, we can then subtract $e_{2g_2^{(j)}}$ from or add it to $e_{g_1^{(j)}} \pm e_{2g_2^{(j)}}$ to isolate $e_{g_1^{(j)}}$. After isolating $e_{g_1^{(j)}}$, a straightforward linear transformation yields a set of $\fo$ distinct unit vectors whose non-zero components are given by $\fO_j$. 
Thus, the union over all $\fB_j$ constitutes a set of linearly independent elements that generates $\ZZ^{\FF_q^*}$.
\item For all $j \in [(q-1)/\fo]$, $2g_1^{(j)} \notin \{g_2^{(j)}, \ldots, g_{\fo}^{(j)}\}$. In this case, consider the union $\bigcup_{j=1}^{(q-1)/\fo}\fB_j$, which has the same span as
\begin{align*}
	\bigcup_{j=1}^{(q-1)/\fo} \cbc{e_{g^{(j)}_1} + e_{g_2^{(j)}}, e_{g^{(j)}_1} - e_{g_3^{(j)}}, e_{g^{(j)}_1} + e_{g_4^{(j)}}, \ldots, e_{g^{(j)}_1} \pm e_{g_{\fo}^{(j)}}} \cup  \cbc{e_{2g_{1}^{(j)}}}.
\end{align*}
Since for each $j$, the element $2g_1^{(j)}$ must be contained in some $\fO_{j'}$ for $j\not=j'$, as in case (1), $e_{2g_{1}^{(j)}}$ can be used to isolate $e_{g_1^{(j')}}$. After isolating $e_{g_1^{(j')}}$ for all $j'$, these elements can be straightforwardly used to linearly transform the union over all $\fB_j$ into the standard basis $(e_{h})_{h \in \FF_q^\ast}$ of $\ZZ^{\FF_q^*}$.
\end{enumerate}
In any case, set $\bigcup_{j=1}^{(q-1)/\fo} \fB_j = \{\fb_1, \ldots, \fb_{q-1}\}$.

\noindent
\textbf{Case 3:} $\chi_2\not= -1$. 

Denote the order of $-\chi_2^{-1}$ in $(\FF_q^\ast, \cdot)$ by $\fo$, so that the elements $1, -\chi_2^{-1}, \ldots, (-\chi_2^{-1})^{\fo-1}$ are pairwise distinct. Then $\fo \mid q-1$, and since $-\chi_2^{-1} \not=1$, $\fo \geq 2$. 
We now partition $\FF_q^\ast$ into orbits of the action of $(\{1, -\chi_2^{-1}, \ldots,(-\chi_2^{-1})^{\fo-1} \}, \cdot)$ on $\FF_q^\ast$ such that 
$$\FF_q^\ast =\dot \bigcup_{j=1}^{(q-1)/\fo} \fO_j,$$ 
where each orbit $\fO_j$ contains exactly $\fo$ elements. Suppose that $\fO_j=\{g^{(j)}_1, \ldots, g^{(j)}_{\fo}\}$, where the elements are indexed such that $g_{i+1}^{(j)} = -\chi_2^{-1}g_{i}^{(j)}$. 

To each $\fO_j$, we associate a set of potential basis vectors whose union over different $j$ then yields the full set $(\fb_1, \ldots, \fb_{q-1})$. More precisely, the set corresponding to $\fO_j$ is defined as
\begin{align*}
	\fB_j= \bigcup_{i=1}^{\fo-1} \cbc{e_{g^{(j)}_i}+ e_{g^{(j)}_{i+1}}} \cup \cbc{e_{g_{1}^{(j)}} + e_{g_{2}^{(j)}} + e_{(1-\chi_3)g_{1}^{(j)}} }.
\end{align*}
In the above, we have used that for $\chi_1=1$,
\begin{align*}
	\chi_1 \cdot h + \chi_2 \cdot (-\chi_2^{-1}) h  + \chi_3 \cdot 0  = 0 \qquad \text{and} \qquad   \chi_1 \cdot (1-\chi_3)h + \chi_2 \cdot (-\chi_2^{-1})h  + \chi_3 \cdot h  = 0.
\end{align*}
Note that the element $(1-\chi_3)g_1^{(j)}$
is distinct from $g_1^{(j)}$. It might be one of $g_2^{(j)}, \ldots, g_{\fo}^{(j)}$.

We next argue that the union of the different $\fB_j$ generates $\ZZ^{\FF_q^\ast}$. By linear transformation, $\fB_j$ has the same span as
\begin{align*}
	\cbc{e_{g^{(j)}_1} + e_{g_2^{(j)}}, e_{g^{(j)}_1} - e_{g_3^{(j)}}, e_{g^{(j)}_1} + e_{g_4^{(j)}}, \ldots, e_{g^{(j)}_1} \pm e_{g_{\fo}^{(j)}}} \cup  \cbc{ e_{(1-\chi_3)g_{1}^{(j)}} }.
\end{align*}
Now, there are two cases.
\begin{enumerate}
\item For all $j \in [(q-1)/\fo]$, $(1-\chi_3)g_{1}^{(j)}$ is one of the elements $g_2^{(j)}, \ldots, g_{\fo}^{(j)}$. As in case 1, we can then subtract $e_{2g_2^{(j)}}$ from or add it to $e_{g_1^{(j)}} \pm e_{(1-\chi_3)g_1^{(j)}}$ to isolate $e_{g_1^{(j)}}$. After isolating $e_{g_1^{(j)}}$, a straightforward linear transformation yields a set of $\fo$ distinct unit vectors whose non-zero components are given by $\fO_j$. 
Thus, the union over all $\fB_j$ constitutes a set of linearly independent elements that generates $\ZZ^{\FF_q^*}$.
\item For all $j \in [(q-1)/\fo]$, $(1-\chi_3)g_{1}^{(j)}$ is none of the elements $g_2^{(j)}, \ldots, g_{\fo}^{(j)}$. In this case, consider the union $\bigcup_{j=1}^{(q-1)/\fo}\fB_j$, which has the same span as
\begin{align*}
	\bigcup_{j=1}^{(q-1)/\fo} \cbc{e_{g^{(j)}_1} + e_{g_2^{(j)}}, e_{g^{(j)}_1} - e_{g_3^{(j)}}, e_{g^{(j)}_1} + e_{g_4^{(j)}}, \ldots, e_{g^{(j)}_1} \pm e_{g_{\fo}^{(j)}}} \cup  \cbc{e_{(1-\chi_3)g_{1}^{(j)}}}.
\end{align*}
Since for each $j$, the element $(1-\chi_3)g_1^{(j)}$ must be contained in some $\fO_{j'}$ for $j\not=j'$, as in case (1), $e_{(1-\chi_3)g_{1}^{(j)}}$ can be used to isolate $e_{g_1^{(j')}}$. After isolating $e_{g_1^{(j')}}$ for all $j'$, these elements can be straightforwardly used to linearly transform the union over all $\fB_j$ into the standard basis $(e_{h})_{h \in \FF_q^\ast}$ of $\ZZ^{\FF_q^*}$.
\end{enumerate}
In any case, set $\bigcup_{j=1}^{(q-1)/\fo} \fB_j = \{\fb_1, \ldots, \fb_{q-1}\}$.

\section{Proof of \Prop~\ref{prop_mmt}}\label{sec_prop_mmt}

\subsection{Overview}
The aim in this section is to bound the expected size of the kernel of $\A$ on $\fO$ from \eqref{eqO}, i.e., $|\ker\A|\cdot\vecone\fO$.
As in \Sec~\ref{sec_grand_outline} we let $\fA$ be the $\sigma$-algebra generated by $\vm,(\vk_i)_{i\geq1},(\vd_i)_{i\geq1}$ and by the numbers $\vm(\chi_1,\ldots,\chi_\ell)$ of equations of degree $\ell\geq3$ with coefficients $\chi_1,\ldots,\chi_\ell\in\FF_q^*$.
Thus, the total degree $\vec\Delta=\sum_{i=1}^n\vd_i$ is $\fA$-measurable.

Let us first observe that it suffices to count ``nearly equitable'' kernel vectors, in the following sense.
For a vector $\sigma\in\FF_q^n$ and $s\in\FF_q$ define the empirical frequency
\begin{align}\label{eqrho}
	\rho_\sigma(s)=\sum_{i=1}^n \vd_i\vecone\cbc{\sigma_i=s}
\end{align}
and let $\rho_\sigma=(\rho_\sigma(s))_{s\in\FF_q}$.
If $\fO$ occurs, then $\rho_\sigma$ is nearly uniform for most kernel vectors.
Formally, we have the following statement.

\begin{fact}\label{Cor_Maurice}
For any $\eps>0$ \whp\ given $\fA$ we have
	$\vecone\fO\cdot|\ker\A|\leq(1+\eps)\abs{\cbc{\sigma\in\ker\A:\|\rho_\sigma-q^{-1}\vec\Delta\vecone\|_1<\eps\vec\Delta}}.$
\end{fact}
\begin{proof}
	Choose $\delta=\delta(\eps,q)>0$ small enough.
	Since $0<\ex[\vd^2]<\infty$ we find a constant $d^*>0$ such that 
	\begin{align}\label{eqCorMaurice1}
		\vec\Delta>\sqrt\delta n\qquad\mbox{ and }\qquad\sum_{i=1}^n\vecone\{\vd_i>d^*\}\vd_i<\delta \vec\Delta\qquad\mbox{\whp}
	\end{align}
	Now, the definition \eqref{eqO} of $\fO$ implies that for any degree $\ell\leq d^*$ a random vector $\vx_{\A}\in\ker\A$ satisfies
		\begin{align}\label{eqCorMaurice3}
			\sum_{s,t\in\FF_q}\sum_{i,j=1}^n\vecone\{\vd_i=\vd_j=\ell\}\abs{\pr\brk{\vx_{\A,i}=s,\ \vx_{\A,j}=t\mid\A}-q^{-2}}=o(n^2)\quad\mbox{on }\fO.
		\end{align}
		By Chebyshev's inequality \whp\ $\sum_{j=1}^n\vecone\{\vd_j=\ell\}=\Omega(n)$ and consequently \eqref{eqCorMaurice3} shows that \whp\ for a random vector $\vx_{\A}$ we have
		\begin{align}\label{eqCorMaurice2}
			\abs{\sum_{i=1}^n\vecone\cbc{\vd_i=\ell}\bc{ \vecone\cbc{\vx_{\A,i}=s}-1/q}}=o(n)\qquad\mbox{for all }s\in\FF_q,\ \ell\leq d^*\mbox{ on the event }\fO.
		\end{align}
		Combining \eqref{eqCorMaurice1} and \eqref{eqCorMaurice2} with the definition \eqref{eqrho} of $\rho_\sigma$ completes the proof.
\end{proof}

We proceed to contemplate different regimes of ``nearly equitable'' frequency vectors and employ increasingly subtle estimates to bound their contributions.
To this end let $\fP_q$ be the set of all possible frequency vectors, i.e.,
\begin{align*}
	\fP_q&=\cbc{\rho_\sigma:\sigma\in\FF_q^n}.
\end{align*}
Moreover, for $\eps>0$ let
\begin{align*}
	\fP_q(\eps)&= \cbc{\rho\in\fP_q: \|\rho - q^{-1}\vec\Delta \vecone\| < \eps\vec\Delta}.
\end{align*}
In addition, we introduce
\begin{align*}
	\cZ_\rho&=\abs{\cbc{\sigma\in\ker\vA:\rho_\sigma=\rho}}&&(\rho\in\fP_q),\\
	\cZ_\eps&= \sum_{\rho\in\fP_q(\eps)}\cZ_\rho&&(\eps\geq0),\\
	\cZ_{\eps,\eps'} &= \cZ_{\eps'}-\cZ_\eps&&(\eps,\eps'\geq0).
\end{align*}

The following lemma sharpens the $\eps\vec\Delta$ error bound from Fact~\ref{Cor_Maurice} to $\omega n^{-1/2}\vec\Delta$.

\begin{lemma}\label{lemma_far}
	For any fixed $\eps>0$ for large enough $\omega=\omega(\eps)>1$ \whp\ we have $\ex\brk{\cZ_{\omega n^{-1/2},\eps}\mid\alg}<\eps q^{n-\vm}.$
\end{lemma}

\noindent
The proof of \Lem~\ref{lemma_far}, which can be found in \Sec~\ref{sec_lemma_far}, is based on an expansion to the second order of the optimisation problem \eqref{eqXOR11} around the equitable solution.
Similar arguments have previously been applied in the theory of random constraint satisfaction problems, particularly random $k$-XORSAT (e.g.~\cite{ANP,Ayre,DuboisMandler}).

For $\rho$ that are within $O(n^{-1/2}\vec\Delta)$ of the equitable solution such relatively routine arguments do not suffice anymore.
Indeed, by comparison to examples of random CSPs that have been studied previously, sometimes by way of the small subgraph conditioning technique, a new challenge arises.
Namely, due to the algebraic nature of our problem the conceivable empirical distributions $\rho_{\vx}$ given that $\vx\in\ker\A$ are confined to a proper sub-lattice of $\ZZ^q$.
The same is true of $\fP_q$ unless $\fd=1$.
Hence, we need to work out how these lattices intersect.
Moreover, for $\rho\in\fP_q$ we need to calculate the number of assignments $\sigma$ such that $\rho_\sigma=\rho$ as well as the probability that such an assignment satisfies all equations.
Seizing upon \Prop~\ref{prop_module} and local limit theorem-type techniques, we will deal with these challenges in \Sec~\ref{sec_lemma_near}, where we prove the following.

\begin{lemma}\label{lemma_near}
	For any $\eps>0$ for large enough $\omega=\omega(\eps)>1$ we have $\ex[\cZ_{\omega n^{-1/2}}\mid\alg]\leq (1+\eps)q^{n-\vm}$ \whp
\end{lemma}

\begin{proof}[Proof of \Prop~\ref{prop_mmt}]
	This is an immediate consequence of Fact~\ref{Cor_Maurice}, \Lem~\ref{lemma_far} and \Lem~\ref{lemma_near}.
\end{proof}

\subsection{Proof of \Lem~\ref{lemma_far}}\label{sec_lemma_far}
As we just saw, on the one hand we need to count $\sigma\in\FF_q^n$ such that $\rho_\sigma$ hits a particular attainable $\rho\in\fP_q(\eps)$.
On the other hand, we need to estimate the probability that such a given $\sigma$ satisfies all equations.
The first of these, the entropy term, increases as $\rho$ becomes more equitable.
The second, probability term takes greater values for non-uniform $\rho$.
Roughly, the more zero entries $\rho$ contains, the better.
The thrust of the proofs of \Lem s~\ref{lemma_far} and~\ref{lemma_near} is to show that the drop in entropy is an order of magnitude stronger than the boost to the success probability.

Toward the proof of \Lem~\ref{lemma_far} we can get away with relatively rough bounds, mostly disregarding constant factors.
The first claim bounds the entropy term.
Instead of counting assignments we will take a probabilistic viewpoint.
Hence, let $\vec\sigma\in\FF_q^n$ be a uniformly random assignment.

\begin{claim}\label{lemma_entropy_rough}
There exists $C>0$ such that \whp\ $\pr\brk{\|\rho_{\vec \sigma}-q^{-1}\vec\Delta\vecone\|_1>t\sqrt{\vec\Delta}\mid\alg}\leq C\exp(-nt^2/C)$ for all $t>0$.
\end{claim}
\begin{proof}
	Since $\ex[\vd^2]<\infty$, this is an immediate consequence of Azuma--Hoeffding.
\end{proof}

Let us move on to the probability term.
We proceed indirectly by way of Bayes' rule.
Hence, fix $\rho\in\fP_q(\eps)$ and let $\vec\xi=(\vec\xi_{ij})_{i,j\geq1}$ be an infinite array of $\FF_q$-valued random variables with distribution $\vec\Delta^{-1}\rho$, mutually independent and independent of all other randomness.
Moreover, let
\begin{align}\label{eqRS}
	\fR(\rho)&=\bigcap_{s\in\FF_q}\cbc{\sum_{i=1}^{\vm}\sum_{j=1}^{\vk_i}\vecone\cbc{\vec\xi_{ij}=s}=\rho(s)},&
	\fS&=\cbc{\forall i \in [\vm]: \sum_{j=1}^{\vk_i}\vec\chi_{ij}\vec\xi_{ij}=0}.
\end{align}
In words, $\fR(\rho)$ is the event that the empirical distribution induced by the random vector $\vec\xi_{ij}$, truncated at $i=\vm$ and $j=\vk_i$ for every $i$, works out to be $\rho\in\fP_q$.
Furthermore, $\fS$ is the event that all $\vm$ checks are satisfied if we substitute the independent values $\vec\xi_{ij}$ for the variables.

Crucially, $\fS$ ignores that the various equations share variables, or conversely that variables may appear in several distinct checks.
Hence, the {\em unconditional} event $\fS$ effectively just deals with a linear system whose Tanner graph consists of $\vm$ checks with degrees $\vk_1,\ldots,\vk_{\vm}$ and $\sum_{i=1}^{\vm}\vk_i$ variable nodes of degree one each.
However, the {\em conditional} probability $\pr_{\alg}\brk{\fS\mid\fR(\rho)}$ equals the probability that a random assignment $\vec\sigma$ lies in the kernel of $\vA$ given that $\rho_{\vec\sigma}=\rho$; in symbols,
\begin{align}\label{eqJane}
	\pr_{\alg}\brk{\fS\mid\fR(\rho)}&=\pr_{\alg}\brk{\vec\sigma\in\ker\vA\mid\rho_{\vec\sigma}=\rho}.
	\end{align}
	Indeed, given $\alg$ and given $\rho_{\vec\sigma}=\rho$ the randomness that remains amounts to just how the variable clones are matched to the check clones to form the random Tanner graph $\G$.
	We can think of this matching as randomly distributing $\sum_{i=1}^n\vd_i\rho(s)$ ``pebbles'' with value $s$ onto the $\vm$ equations.
	The probability that the pebbles happen to satisfy all the equations is precisely equal to $\pr_{\alg}\brk{\fS\mid\fR(\rho)}$.

We are going to see momentarily that the unconditional probabilities of $\fR(\rho)$ and $\fS$ are easy to calculate.
In addition, we will be able to calculate the conditional probability $\pr_{\alg}\brk{\fS\mid\fR(\rho)}$ by way of the local limit theorem for sums of independent random variables.
Finally, \Lem~\ref{lemma_far} will follow from these estimates via Bayes' rule.

\begin{claim}\label{lemma_rough_S}
	We have $\pr_{\fA}\brk{\fS}=q^{\vm(O(\sum_{s\in\FF_q}|\vec\Delta^{-1}\rho(s)-1/q|^3)-1)}$.
\end{claim}
\begin{proof}
	For any $h\geq3$ and any $\chi_1,\ldots,\chi_h\in\supp\chi$ we aim to calculate
	\begin{align*}
		P_h&=\log\sum_{\sigma\in\FF_q^h}\vecone\cbc{\sum_{i=1}^h\chi_i\sigma_i=0}\prod_{i=1}^h\frac{\rho(\sigma_i)}{\vec\Delta}.
	\end{align*}
	The derivatives of this expression work out to be
	\begin{align*}
		\frac{\partial P_h}{\partial\rho_s}&=\frac{\sum_{j=1}^h\sum_{\sigma\in\FF_q^h}\vecone\cbc{\sum_{i=1}^h\chi_i\sigma_i=0,\,\sigma_j=s}\prod_{i\neq j}\frac{\rho(\sigma_i)}{\vec\Delta}}{\vec\Delta \eul^{P_h}}\qquad(s\in\FF_q),\\
		\frac{\partial^2 P_h}{\partial\rho_s\partial\rho_{s'}}&=\frac{\sum_{j,j'=1}^h\sum_{\sigma\in\FF_q^h}\vecone\cbc{\sum_{i=1}^h\chi_i\sigma_i=0,\,\sigma_j=s,\sigma_{j'}=s'}\prod_{i\neq j,j'}\frac{\rho(\sigma_i)}{\vec\Delta}}{\vec\Delta^2 \eul^{P_h}}-\frac{\partial P_h}{\partial\rho_s}\frac{\partial P_h}{\partial\rho_s'}&&(s,s'\in\FF_q,\ s\neq s'),\\
		\frac{\partial^2 P_h}{\partial\rho_s^2}&=\frac{\sum_{j\neq j'}\sum_{\sigma\in\FF_q^h}\vecone\cbc{\sum_{i=1}^h\chi_i\sigma_i=0,\,\sigma_j=\sigma_{j'}=s}\prod_{i\neq j,j'}\frac{\rho(\sigma_i)}{\vec\Delta}}{\vec\Delta^2 \eul^{P_h}}-\bcfr{\partial P_h}{\partial\rho_s}^2.
	\end{align*}
	Evaluating the derivatives at the equitable $\bar\rho=q^{-1}\vec\Delta\vecone$ we obtain for any $i\geq3$,
\begin{align*}
	\frac{\partial P_h}{\partial\rho_s}\bigg|_{\bar\rho}&=\frac{hq^{-1}}{\vec\Delta q^{-1}}=\frac{h}{\vec\Delta},\\
	\frac{\partial^2 P_h}{\partial\rho_s\partial\rho_{s'}}\bigg|_{\bar\rho}&=\frac{h(h-1)q^{-1}}{\vec\Delta^2 q^{-1}}-\frac{h^2}{\vec\Delta^2}=-\frac{h}{\vec\Delta^2},&&(s\neq s') \\
	\frac{\partial^2 P_h}{\partial\rho_s^2}\bigg|_{\bar\rho}&=\frac{h(h-1)q^{-1}}{\vec\Delta^2 q^{-1}}-\frac{h^2}{\vec\Delta^2}=-\frac{h}{\vec\Delta^2}.
	\end{align*}
	Hence, the Jacobi matrix and the Hessian work out to be
	\begin{align}\label{eq2ndDeriv7}
	DP_h\bc{\bar\rho}&=\frac{h}{\vec\Delta}\vecone_q,&
	D^2P_h\bc{\bar\rho}&=-\frac{h}{\vec\Delta^2}\vecone_{q\times q}.
	\end{align}
	Furthermore, the third partial derivatives are clearly bounded, i.e.,
	\begin{align}\label{eq2ndDeriv8}
		\frac{\partial^3 P_h}{\partial\rho_s\partial\rho_{s'}\partial\rho_{s''}}&=O_\eps(1).
	\end{align}
	Since $\rho-\bar\rho\perp\vecone$, \eqref{eq2ndDeriv7}, \eqref{eq2ndDeriv8} and Taylor's theorem imply the assertion.
\end{proof}

\begin{claim}\label{lemma_rough_R}
	\Whp\ we have $\pr_{\fA}\brk{\fR(\rho)}=\Omega_\eps(n^{(1-q)/2})$.
\end{claim}
\begin{proof}
	Since the $\vec\xi_{ij}$ are mutually independent, the probability of $\fR(\rho)$ given $\fA$ is nothing but
	\begin{align*}
		\pr_{\fA}\brk{\fR(\rho)}=\binom{\vec\Delta}{(\rho(s))_{s\in\FF_q}}\prod_{s\in\FF_q}\bcfr{\rho(s)}{\vec\Delta}^{\rho(s)}.
	\end{align*}
	The claim therefore follows from Stirling's formula.
\end{proof}

\begin{claim}\label{lemma_rough_RS}
	\Whp\ we have $\pr_{\fA}\brk{\fR(\rho)\mid\fS}=O_\eps(n^{(1-q)/2})$.
\end{claim}
\begin{proof}
	The claim follows from the local limit theorem for the sums of independent random variables (e.g.~\cite{DavMc}).
	To elaborate, even once we condition on the event $\fS$ the random {\em vectors} $(\vec\xi_{ij})_{j\in[\vk_i]}$, $1\leq i\leq\vm$, remain independent for different $i\in[\vm]$ due to the independence of the $(\vec\xi_{ij})_{i,j}$.
	Indeed, $\fS$ only asks that each check be satisfied separately, without inducing dependencies among different checks.
	Thus, the vector
	\begin{align*}
		\bc{\sum_{i=1}^{\vm}\sum_{j=1}^{\vk_i}\vecone\cbc{\vec\xi_{ij}=s}}_{s\in\FF_q}\qquad\mbox{given }\fS
	\end{align*}
	is a sum of $\vm$ independent random vectors.
	The local limit theorem therefore implies that the probability of the most likely outcome of this random vector is of order $n^{(1-q)/2}$; in symbols,
	\begin{align}\label{eq_lemma_rough_RS}
		\max_{r\in\fP_q(\eps)}\pr_{\fA}\brk{\fR(r)\mid\fS}=O(n^{(1-q)/2}).
	\end{align}
	The assertion is an immediate consequence of \eqref{eq_lemma_rough_RS}.
\end{proof}

\begin{proof}[Proof of \Lem~\ref{lemma_far}]
	Fix $\rho\in\fP_q(\eps)$ such that $\omega\sqrt{\vec\Delta}\leq\sum_{s\in\FF_q}|\rho(s)-\vec\Delta/q|\leq\eps\vec\Delta$.
	Combining Claims~\ref{lemma_rough_S}--\ref{lemma_rough_RS} with Bayes' rule, we conclude that \whp\
	\begin{align}\label{eqlemma_far1}
		\pr_{\fA}\brk{\fS\mid\fR(\rho)}&=\frac{\pr_{\fA}\brk{\fS}\pr_{\fA}\brk{\fR(\rho)\mid\fS}}{\pr_{\fA}\brk{\fR(\rho)}}=O(\pr_{\fA}\brk{\fS})=q^{\vm(O(\sum_{s\in\FF_q}|\rho(s)/\vec\Delta-1/q|^3)-1)+O(1)}.
	\end{align}
	Consequently, \eqref{eqJane} and \eqref{eqlemma_far1} imply that
	\begin{align}\label{eqlemma_far2}
		\pr_{\fA}\brk{\vec\sigma\in\ker\A\mid\rho_{\vec\sigma}=\rho}=\pr_{\alg}\brk{\fS\mid\fR(\rho)}=q^{\vm(O(\sum_{s\in\FF_q}|\rho(s)/\vec\Delta-1/q|^3)-1)+O(1)}.
	\end{align}
	Hence, combining Claim~\ref{lemma_entropy_rough} with \eqref{eqlemma_far2} and using the bound $\sum_{s\in\FF_q}|\rho(s)-\vec\Delta/q|\leq\eps\vec\Delta$, we obtain
		\begin{align}\label{eqlemma_far3}
			\pr_{\fA}\brk{\vec\sigma\in\ker\A,\,\rho_{\vec\sigma}=\rho}=q^{\vm(O(\sum_{s\in\FF_q}|\rho(s)/\vec\Delta-1/q|^3)-(\Omega(\sum_{s\in\FF_q}|\rho(s)/\vec\Delta-1/q|^2)-1)+O(1)} =q^{\vm(-1-\Omega(\sum_{s\in\FF_q}|\rho(s)/\vec\Delta-1/q|^2)+O(1)}.
	\end{align}
	Multiplying \eqref{eqlemma_far3} with $q^n$ and summing on $\rho\in\fP_q(\eps)$ such that $\omega n^{-1/2} \vec\Delta \leq \sum_{s\in\FF_q}|\rho(s)-\vec \Delta/q|$, we finally obtain
	\begin{align*}
		\ex_{\fA}\brk{\cZ_{\omega n^{-1/2},\eps}}&=q^{n-\vm+O(1)}\sum_{\substack{\rho\in\fP\\\omega n^{-1/2} \vec\Delta\leq\sum_{s\in\FF_q}|\rho(s)-\vec\Delta/q| < \eps\vec\Delta}}\exp\bc{-\Omega\bc{n\sum_{s\in\FF_q}|\rho(s)/\vec\Delta-1/q|^2}}
		<\eps q^{n-\vm},
	\end{align*}
provided $\omega=\omega(\eps)>0$ is chosen large enough.
\end{proof}

\subsection{Proof of \Lem~\ref{lemma_near}}\label{sec_lemma_near}

By comparison to the proof of \Lem~\ref{lemma_far}, the main difference here is that we need to be more precise.
Specifically, while in Claims~\ref{lemma_rough_R} and~\ref{lemma_rough_RS} we got away with disregarding constant factors, here we need to be accurate up to a multiplicative $1+o(1)$.
Working out the probability term turns out to be delicate.
As in \Sec~\ref{sec_lemma_far} we introduce auxiliary $\FF_q$-valued random variables $\vec\xi=(\vec\xi_{ij})_{i,j\geq1}$.
These random variables are mutually independent as well as independent of all other randomness.
But this time all $\vec\xi_{ij}$ are {\em uniform} on $\FF_q$.
Let $\fR(\rho)$ and $\fS$ be the events from \eqref{eqRS}.

Similarly as in \Sec~\ref{sec_lemma_far} we will ultimately apply Bayes' rule to compute the probability of $\fS$ given $\fR(\rho)$ and hence the conditional mean of $\cZ_\rho$.
The individual probability $\fR(\rho)$ is easy to compute.

\begin{claim}\label{lemma_precise_R}
	For any $\rho\in\fP_q$ we have $\pr_{\fA}\brk{\fR(\rho)}=\binom{\vec\Delta}{\rho}q^{-\vec\Delta}$.
\end{claim}
\begin{proof}
	This is similar to the proof of Claim~\ref{lemma_rough_R}.
	As the $\vec\xi_{ij}$ are uniformly distributed and independent, we obtain
	\begin{align*}
		\pr_{\fA}\brk{\fR(\rho)}=\binom{\vec\Delta}{(\rho(s))_{s\in\FF_q}}\prod_{s\in\FF_q}q^{-\sum_{s\in\FF_q}\rho(s)}=\binom{\vec\Delta}{\rho}q^{-\vec\Delta},
	\end{align*}
	as claimed.
\end{proof}

As a next step we calculate the conditional probability of $\fS$ given $\fR(\rho)$.
Similar to (\ref{eqrho}), for $s\in\FF_q$ define the empirical frequency
\begin{align}\label{eqdefRHO}
	\vec\rho(s)=\sum_{i=1}^{\vec m}\sum_{j=1}^{\vec k_i} \vecone\cbc{\vec \xi_{ij}=s}
\end{align}
and let $\vec\rho=(\vec\rho(s))_{s\in\FF_q}$ as well as $\hat{\vec\rho}=(\vec\rho(s))_{s\in\FF_q^{\ast}}$. Of course, \Prop~\ref{prop_module} implies that for some $\rho\in\fP_q$ the event $\fS$ may be impossible given $\fR(\rho)$.
Hence, to characterise the distributions $\rho$ for which $\fS$ can occur at all, we let 
\begin{align}
	\fL&=\cbc{r\in\ZZ^{\FF_q^*}:\pr_{\fA}\brk{\mycheck{\vec\rho} = r}>0\mbox{ and }\norm{r-q^{-1}\vec\Delta\vecone}_1\leq\omega n^{-1/2}\vec\Delta}\label{eqL},\\
	\fL_0&=\cbc{r\in\fL:\pr_{\fA}\brk{\mycheck{\vec\rho}=r\mid\fS}>0}\label{eqL0},\\
	\fL_*&=\cbc{r\in\fL:\pr_{\fA}\brk{\mycheck\rho_{\vec\sigma}=r}>0}.\label{eqL*}
\end{align}
Thus, $\fL$ contains all conceivable outcomes of truncated frequency vectors.
Moreover, $\fL_0$ comprises those frequency vectors that can occur given $\fS$, and $\fL_*$ those that can result from random assignments $\vec\sigma$ to the variables.
Hence, $\fL_0$ is a finite subset of the $\ZZ$-module generated by those sets $\cS_q(\chi_1,\ldots,\chi_\ell)$ from \eqref{eqS} with $\vm(\chi_1,\ldots,\chi_\ell)>0$.
The following lemma shows that actually the conditional probability $\fS$ given $\fR(\rho)$ is asymptotically the same for all $\rho\in\fL_0$, i.e., for all conceivably satisfying $\rho$ that are nearly equitable.

\begin{lemma}\label{lemma_uniformly}
	\Whp\ uniformly for all $r \in\fL_0$ we have $\pr_{\alg}\brk{\fS\mid\mycheck{\vec\rho}=r}\sim q^{\vecone\{|\supp\vec\chi|=1\}-\vm}$. 
\end{lemma}

We complement \Lem~\ref{lemma_uniformly} by the following estimate of the probability that a uniformly random assignment $\vec\sigma\in\FF_q^n$ hits the set $\fL_0$ in the first place.

\begin{lemma}\label{lemma_hit}
	\Whp\ we have $\pr_{\alg}\brk{\hat{\rho}_{\vec\sigma}\in\fL_0}\leq(1+o(1))q^{-\vecone\{|\supp\vec\chi|=1\}}.$
\end{lemma}

\noindent
We prove \Lem s~\ref{lemma_uniformly} and~\ref{lemma_hit} in \Sec s~\ref{sec_uniformly} and~\ref{sec_hit}, respectively.

\begin{proof}[Proof of \Lem~\ref{lemma_near}]
	The formula \eqref{eqJane} extends to the present auxiliary probability space with uniformly distributed and independent $\vec\xi_{ij}$ (for precisely the same reasons given in \Sec~\ref{sec_lemma_far}).
	Hence, \eqref{eqJane}, \eqref{eqL} and \eqref{eqL0} show that
	\begin{align}\label{eqlemma_near1}
		\ex[\cZ_{\omega n^{-1/2}}\mid\alg]&\leq\sum_{\sigma\in\FF_q^n}\vecone\cbc{\hat\rho_\sigma\in\fL}\pr_{\alg}\brk{\fS\mid\hat{\vec\rho}=\hat\rho_{\sigma}}=\sum_{\sigma\in\FF_q^n}\vecone\cbc{\hat\rho_\sigma\in\fL_0}\pr_{\alg}\brk{\fS\mid\hat{\vec\rho}=\hat\rho_{\sigma}}.
		\end{align}
	Finally, combining \eqref{eqlemma_near1} with \Lem~\ref{lemma_uniformly} and \Lem~\ref{lemma_hit}, we obtain
\begin{align*}
		\ex[\cZ_{\omega n^{-1/2}}\mid\alg]&\leq (1+o(1))q^{\vecone\cbc{|\supp\vec\chi|=1}-\vec m}\sum_{\sigma\in\FF_q^n}\vecone\cbc{\hat\rho_\sigma\in\fL_0}\\
										  &=(1+o(1))q^{n-\vm+\vecone\cbc{|\supp\vec\chi|=1}}\pr_{\alg}\brk{\hat{\rho}_{\vec\sigma}\in\fL_0}\leq(1+o(1))q^{n-\vm},
		\end{align*}
	as desired.
\end{proof}

\subsection{Proof of \Lem~\ref{lemma_uniformly}}\label{sec_uniformly}

\noindent
Given $\omega>0$ (from \eqref{eqL}) we choose $\eps_0=\eps_0(\omega,q)$ sufficiently small and let $0<\eps<\eps_0$. 
Moreover, recall that we assume the existence of a constant $\eta>0$ such that $\Erw[\vec d^{2+\eta}] + \Erw[\vec k^{2+\eta}] < \infty$.
The proof hinges on a careful analysis of the conditional distribution of $\mycheck{\vec\rho}$ given $\fS$.
We begin by observing that the vector $\mycheck{\vec\rho}$ is asymptotically normal given $\fS$.
Let $\vec I_{q-1}$ the $(q-1)\times(q-1)$-identity matrix and let $\vN\in\RR^{\FF_q^*}$ be a Gaussian vector with zero mean and covariance matrix 
	\begin{align}\label{eqCmatrix}
		\cC=q^{-1}\vec I_{q-1}-q^{-2}\vecone_{(q-1)\times(q-1)}.
	\end{align}

\begin{claim}\label{fact_boxS}
	There exists a function $\alpha=\alpha(n,q,\eta)=o(1)$ such that for all axis-aligned cubes $U\subset\RR^{\FF_q^*}$ we have
	$$\ex\abs{\pr_{\fA}\brk{\vec\Delta^{-1/2}(\mycheck{\vec\rho}-q^{-1}\vec\Delta\vecone)\in U\mid\fS}-\pr\brk{\vN\in U}}\leq\alpha.$$
\end{claim}
\begin{proof}
	The conditional mean of $\mycheck{\vec\rho}$ given $\fS$ is uniform.
	To see this, consider any $i\in[\vm]$ and $h\in[\vk_i]$.
	We claim that for any vector $(\tau_j)_{j\in[\vk_i]\setminus\{h\}}$,
		\begin{align}\label{eqfact_boxS}
			\pr_{\fA}\brk{\forall j\in[\vk_i]\setminus\{h\}:\vec\xi_{ij}=\tau_j\mid\fS}&\sim q^{1-\vk_i}.
		\end{align}
	Indeed, for any such vector $(\tau_j)_{j\in[\vk_i]\setminus\{h\}}$ there is exactly one value $\vec\xi_{ih}$ that will satisfy the constraint, namely
		\begin{align*}
\vec\xi_{ih}=-\vec\chi_{ih}^{-1}\sum_{j\in[\vk_i]\setminus\{h\}}\vec\chi_{ij}\tau_j.
			\end{align*}
	Hence, given $\fS$ the events $\{\forall j\in[\vk_i]\setminus\{h\}:\vec\xi_{ij}=\tau_j\}$ are equally likely for all $\tau$, which implies \eqref{eqfact_boxS}.
	Furthermore, together with the definition \eqref{eqdefRHO} of $\vec\rho$, \eqref{eqfact_boxS} readily implies that $\ex_{\fA}\brk{\mycheck{\vec\rho}}=q^{-1}\vec\Delta\vecone$.
	Similarly, \eqref{eqfact_boxS} also shows that $\vec\Delta^{-1/2}\hat{\vec\rho}$ has covariance matrix $\cC$.

	Finally, we are left to prove the desired uniform convergence to the normal distribution.
	To this end we employ the multivariate Berry-Esseen theorem (e.g., \cite{Raic}).
	Specifically, given a small $\alpha>0$ choose $K=K(q,\eta,\alpha)>0$ and $m_0=m_0(K)$, $n_0=n_0(K,m_0)$ sufficiently large.
	Assuming $n>n_0$, we can ensure that \whp\ $\vm>m_0$.
	Also let
	\begin{align*}
		\vk_i'&=\vecone\{\vk_i\leq K\}\vk_i,&\vk_i''&=\vk_i-\vk_i',\\
		\hat{\vec\rho}'(s)&=\sum_{1\leq i\leq \vm: \vk_i\leq K}\sum_{j=1}^{\vk_i}\vecone\{\vec\xi_{ij}=s\},&
		\hat{\vec\rho}''(s)&=\sum_{1\leq i\leq \vm: \vk_i>K}\sum_{j=1}^{\vk_i}\vecone\{\vec\xi_{ij}=s\},\\
		\vec\Delta'&=\sum_{i=1}^n\vk_i',&\vec\Delta''&=\sum_{i=1}^n\vk_i''.
	\end{align*}
	Then the assumtion $\ex[\vk^{2+\eta}]<\infty$ and Markov's inequality ensure that \whp
	\begin{align}\label{eqkpedestrian3}
		\vec\Delta''&<\alpha^8\vec\Delta.
	\end{align}
	Moreover, by the same reasoning as in the previous paragraph the random vectors $\hat{\vec\rho}'$ and $\hat{\vec\rho}''$ have means $q^{-1}\vec\Delta'$ and $q^{-1}\vec\Delta''$ and covariances $\vec\Delta'\cC$ and $\vec\Delta''\cC$, respectively.
	Thus, \eqref{eqkpedestrian3} and Chebyshev's inequality show that \whp
	\begin{align}\label{eqkpedestrian5}
		\pr_{\fA}\brk{\norm{\frac{\mycheck{\vec\rho}''-q^{-1}\vec\Delta''\vecone}{\vec\Delta^{1/2}}}>\alpha^2}&<\alpha^2.
	\end{align}
	Further, the Berry--Esseen theorem shows that \whp\ 
	\begin{align}\label{eqkpedestrian4}
		\pr_{\fA}\brk{\frac{\mycheck{\vec\rho}'-q^{-1}\vec\Delta'\vecone}{\sqrt{\vec\Delta'}}\in U}-\pr\brk{\vN\in U}&=O(n^{-1/2})\qquad\mbox{for all cubes }U.
	\end{align}
	Combining \eqref{eqkpedestrian4} and \eqref{eqkpedestrian5}, we see that \whp
	\begin{align}\label{eqkpedestrian6}\abs{\pr_{\fA}\brk{\frac{\mycheck{\vec\rho}-q^{-1}\vec\Delta\vecone}{\sqrt{\vec\Delta}}\in U}-\pr\brk{\vN\in U}}\leq \alpha.\end{align}
	The assertion follows from \eqref{eqkpedestrian6} by taking $\alpha\to0$ slowly as $n\to\infty$.
\end{proof}

The following claim states that the normal approximation from Claim~\ref{fact_boxS} also holds for the unconditional random vector $\hat{\vec\rho}$.

\begin{claim}\label{fact_box}
	There exists a function $\alpha=\alpha(n,q,\eta)=o(1)$ such that \whp\ for all convex sets $U\subset\RR^{\FF_q^*}$ we have
	$$\abs{\pr_{\fA}\brk{\vec\Delta^{-1/2}(\mycheck{\vec\rho}-q^{-1}\vec\Delta\vecone)\in U}-\pr\brk{\vN\in U}}\leq\alpha.$$
\end{claim}
\begin{proof}
	This is an immediate consequence of Claim~\ref{lemma_precise_R} and Stirling's formula.
\end{proof}

Let $k_0=\min\supp\vk$.
In the case that $|\supp\vec\chi|=1$ we set $\chi_1=\cdots=\chi_{k_0}$ to the single element of $\supp\vec\chi$.
Moreover, in the case that $|\supp\vec\chi|>1$ we pick and fix any $\chi_1,\ldots,\chi_{k_0}\in\supp\vec\chi$ such that $|\{\chi_1,\ldots,\chi_{k_0}\}|>1$.
Let $\fI_0$ be the set of all $i\in[\vm]$ such that $\vk_i=k_0$ and $\vec\chi_{ij}=\chi_j$ for $j=1,\ldots,k_0$ and let $\fI_1=[\vm]\setminus\fI_0$.
Then $|\fI_0|,|\fI_1|=\Theta(n)$ \whp\
Further, set
\begin{align*}
	\vr_0(s)&=\sum_{i\in\fI_0}\sum_{j\in[\vk_i]}\vecone\cbc{\vec\xi_{ij}=s},&
	\vr_1(s)&=\sum_{i\in\fI_1}\sum_{j\in[\vk_i]}\vecone\cbc{\vec\xi_{ij}=s}&&(s\in\FF_q^*).
\end{align*}
Then $\mycheck{\vec\rho}=\vr_0+\vr_1$.

Because the vectors $\vec\xi_i=(\vec\xi_{i,1},\ldots,\vec\xi_{i,\vk_i})$ are mutually independent, so are $\vr_0=(\vr_0(s))_{s\in\FF_q^*}$ and $\vr_1=(\vr_1(s))_{s\in\FF_q^*}$.
To analyse $\vr_0$ precisely, let 
\begin{align*}
	\cS_0&=\cbc{\sigma\in\FF_q^{k_0}:\sum_{i=1}^{k_0}\chi_i\sigma_i=0}.
\end{align*}
Moreover, for $\sigma\in\cS_0$ let $\vec R_\sigma$ be the number of indices $i\in\fI_0$ such that $\vec\xi_i=\sigma$.
Then conditionally on $\fS$, we have 
\begin{align*}
	\vec r_0(s)&=\sum_{i\in\fI_0}\sum_{j\in[\vk_i]}\vecone\cbc{\vec\xi_{ij}=s}=\sum_{\sigma\in\cS_0}\sum_{j=1}^{k_0}\vecone\cbc{\sigma_j=s}\vR_\sigma\qquad\mbox{given }\fS,
\end{align*}
which reduces our task to the investigation of $\vR=(\vR_\sigma)_{\sigma\in\cS_0}$.

This is not too difficult because given  $\fS$ the random vector $\vR$ has a multinomial distribution with parameter $|\fI_0|$ and uniform probabilities $|\cS_0|^{-1}$.
In effect, the individual entries $\vR(\sigma)$, $\sigma\in\fS_0$, will typically differ by only a few standard deviations, i.e., their typically difference will be of order $O(\sqrt{\vec\Delta})$.
We require a precise quantitative version of this statement.

Recalling the sets from \eqref{eqL}--\eqref{eqL*}, for $r_*\in\fL_0$ and $0<\eps<\eps_0$ we let
\begin{align*}
	\fL_0(r_*,\eps)&=\cbc{r\in\fL_0:\|r-r_*\|_\infty<\eps\sqrt{\vec\Delta}}.
\end{align*}
Furthermore, we say that $\vR$ is {\em $t$-tame} if $|\vR_\sigma-|\cS_0|^{-1}|\fI_0||\leq t\sqrt{\vec\Delta}$ for all $\sigma\in\cS_0$.
Let $\fT(t)$ be the event that $\vR$ is $t$-tame.

\begin{lemma}\label{claim_tame3}
	\Whp\ for every $r_*\in\fL_0$ there exists $r^*\in\fL_0(r_*,\eps)$ such that
	\begin{align}\label{eq_claim_tame3}
	\pr_{\fA}\brk{\mycheck{\vec\rho}=r^*\mid\fS}&\geq\frac{1}{2|\fL_0(r_*,\eps)|}&\mbox{and}&& \pr_{\fA}\brk{\fT(-\log\eps)\mid\fS,\,\mycheck{\vec\rho}=r^*}&\geq1-\eps^4.
\end{align}
\end{lemma}
\begin{proof}
	Recall that the event $\{\mycheck{\rho}=r\}$ is the same as $\fR(r')$ with $r'(s)=r(s)$ for $s\in\FF_q^*$ and $r'(0) = \vec\Delta -\| r\|_1$. 
	As a first step we observe that $\vR$ given $\fS$ is reasonably tame with a reasonably high probability.
	More precisely, since $\vR$ has a multinomial distribution given $\fA$ and $\fS$, the Chernoff bound shows that \whp\ 
	\begin{align}\label{claim_tame}
		\pr_{\fA}\brk{\fT(-\log\eps)\mid\fS}\geq1-\exp(-\Omega_\eps(\log^2(\eps))).
	\end{align}
	Further, Claim~\ref{fact_boxS} implies that  $\pr_{\fA}\brk{\mycheck{\vec\rho}\in\fL_0(r_*,\eps)\mid \fS}\geq \Omega_\eps(\eps^{q-1})\geq\eps^q$ \whp, provided $\eps<\eps_0=\eps_0(\omega)$ is small enough.
	Combining this estimate with \eqref{claim_tame} and Bayes' formula, we conclude that \whp\ for every $r_*\in\fL_0$,
	\begin{align}\label{claim_tame2}
	\pr_{\fA}\brk{\fT(-\log\eps)\mid\fS,\ \mycheck{\vec\rho}\in\fL_0(r_*,\eps)}\geq1-\eps^{5}.
\end{align}

To complete the proof, assume that there does not exist $r^*\in\fL_0(r_*,\eps)$ that satisfies \eqref{eq_claim_tame3}.
	Then for every $r\in\fL_0(r_*,\eps)$ we either have
	\begin{align}\label{eqX0}
		\pr_{\fA}\brk{\mycheck{\vec\rho}=r\mid\fS}&<\frac{1}{2|\fL_0(r_*,\eps)|}&&\mbox{or}\\
		\pr_{\fA}\brk{\fT(-\log\eps)\mid\fS,\,\mycheck{\vec\rho}=r}&<1-\eps^4.\label{eqX1}
	\end{align}
	Let $\fX_0$ be the set of all $r\in\fL_0(r_*,\eps)$ for which \eqref{eqX0} holds, and let $\fX_1=\fL_0(r_*,\eps)\setminus\fX_0$.
	Then \eqref{eqX0}--\eqref{eqX1} yield
	\begin{align*}\nonumber
	\pr_{\fA}\brk{\fT(-\log\eps)\mid\fS,\ \mycheck{\vec\rho}\in\fL_0(r_*,\eps)}
		&\leq \frac{\pr_{\fA}\brk{\mycheck{\vec\rho}\in\fX_0\mid\fS}+\sum_{r\in\fX_1}\pr_{\fA}\brk{\fT(-\log\eps)\mid\fS,\,\mycheck{\vec\rho}=r}\pr_{\fA}\brk{\mycheck{\vec\rho}=r\mid\fS}}{\pr_{\fA}\brk{\fL_0(r_*,\eps)\mid\fS}}\\
		&\leq\frac{\pr_{\fA}\brk{\mycheck{\vec\rho}\in\fX_0\mid\fS}+(1-\eps^4)\pr_{\fA}\brk{\mycheck{\vec\rho}\in\fX_1\mid\fS}}{\abs{\fL_0(r_*,\eps)}}< 1-\eps^4,
	\end{align*} 
	provided that $1-\eps^4>\frac{1}{2}$, in contradiction to \eqref{claim_tame2}.
\end{proof}

Let $\fM=\fM_q(\chi_1,\ldots,\chi_{k_0})$ and let $\fb_1,\ldots,\fb_{q-1}$ be the basis of $\fM$ supplied by \Prop~\ref{prop_module}.
Let us fix vectors $\tau^{(1)},\ldots,\tau^{(q-1)}\in\cS_0$ whose frequency vectors as defined in \eqref{eqMyFreq} coincide with $\fb_,\ldots,\fb_{q-1}$, i.e., 
	$$\mycheck\tau^{(i)}=\fb_i\qquad\mbox{ for $i=1,\ldots,q-1$.}$$
Also let $\fT(r,t)$ be the event that $\mycheck{\vec\rho}=r$ and that $\vR$ is $t$-tame.
The following lemma summarises the key step of the proof of \Lem~\ref{lemma_uniformly}.

\begin{lemma}\label{lemma_transform}
	\Whp\ for any $r_*\in\fL_0$, any $1\leq t\leq\log n$ and any $r,r'\in\fL_0(r_*,\eps)$ there exists a one-to-one map $\psi:\fT(r,t)\to\fT(r',t+O_\eps(\eps))$ such that for all $(R,r_1)\in\fT(r,t)$ we have 
	\begin{align}\label{eqlemma_transform}
		\log\frac{\pr_\fA\brk{(\vR,\vr_1)=(R,r_1)\mid\fS}}{\pr_\fA\brk{(\vR,\vr_1)=\psi(R,r_1)\mid\fS}}=O_\eps(\eps(\omega+t)).
	\end{align}
\end{lemma}
\begin{proof}
	Since $r,r'\in\fM$, we have $r-r'\in\fM$ \whp\
	Indeed, if $\supp\vec\chi>1$, then \Prop~\ref{prop_module} shows that $\fM=\ZZ^{\FF_q^*}$ \whp\
	Moreover, if $\supp\vec\chi=1$, then $\fM$ is a proper subset of the integer lattice $\ZZ^{\FF_q^*}$.
	Nonetheless, \Prop~\ref{prop_module} shows that the modules $$\fM_q(\underbrace{1,\ldots,1}_{\mbox{$\ell$ times}})$$ coincide for all $\ell\geq3$, and therefore $\fM$ coincides with the $\ZZ$-module generated by $\fL_0$.
	Hence, in either case there is a unique representation 
	\begin{align}\label{eqlemma_transform00}
		r-r'&=\sum_{i=1}^{q-1}\lambda_i\fb_i&&(\lambda_i\in\ZZ)
	\end{align}
	in terms of the basis vectors.
	Because $r,r'\in\fL_0(r_*,\eps)$ and 
\begin{align*}
\begin{pmatrix}\lambda_1\\\vdots\\\lambda_{q-1}\end{pmatrix}=\bc{\fb_1\ \cdots\ \fb_{q-1}}^{-1}(r-r'),
	\end{align*}
	the coefficients satisfy 
	\begin{align}\label{eqlemma_transform0}
		|\lambda_i|&=O_\eps(\eps\sqrt{\vec\Delta})\qquad\mbox{for all $i=1,\ldots,q-1$.}
	\end{align}
	Now let $\lambda_0=-\sum_{i=1}^{q-1}\lambda_i$, obtain the vector $R'$ from $R$ by amending the entry $R_0'$ corrsponding to the zero solution $0\in\cS_0$ to
	\begin{align*}
		R'_0&=R_0+\lambda_0,\qquad\mbox{and setting}&
		R_{\tau^{(i)}}'&=R_{\tau^{(i)}}+\lambda_i\qquad\mbox{for all }\sigma\not\in\{0,\tau^{(1)},\ldots,\tau^{(q-1)}\}.
	\end{align*}
	Further, define $\psi(R,r)=(R',r')$.
	Then $\psi(R,r)\in\fT(r',t+O_\eps(\eps))$ due to \eqref{eqlemma_transform00} and \eqref{eqlemma_transform0}.
	Moreover, Stirling's formula and the mean value theorem show that
	\begin{align}
		\frac{\pr_\fA\brk{(\vR,\vr_1)=(R,r_1)\mid\fS}}{\pr_\fA\brk{(\vR,\vr_1)=\psi(R,r_1)\mid\fS}}&=\binom{|\fI_0|}{R|\fI_0|}\binom{|\fI_0|}{R'|\fI_0|}^{-1}=\exp\brk{\sum_{\sigma\in\cS_0}O_\eps\bc{R_\sigma\log R_\sigma-R'_\sigma\log R'_\sigma}}\nonumber\\
																								   &=\exp\brk{O_\eps(|\fI_0|)\sum_{\sigma\in\cS_0}\abs{\int_{R'_\sigma/|\fI_0|}^{R_\sigma/|\fI_0|}\log z\dd z}}\nonumber\\
																								   &=\exp\brk{O_\eps(|\fI_0|)\sum_{\sigma\in\cS_0}\bc{\frac{R_\sigma}{|\fI_0|}-\frac{R'_\sigma}{|\fI_0|}}\log\bc{\frac1q+O_\eps\bcfr{(\omega+t)\sqrt{\vec\Delta}}{|\fI_0|}}}\nonumber\\
																								   &=\exp\brk{O_\eps(|\fI_0|)\sum_{\sigma\in\cS_0}O_\eps\bc{\frac{(\omega+t)\sqrt{\vec\Delta}}{|\fI_0|}\bc{\frac{R_\sigma}{|\fI_0|}-\frac{R'_\sigma}{|\fI_0|}}}}.\label{eqlemma_transform1}
	\end{align}
	Since $|\fI_0|=\Theta_\eps(\vec\Delta)=\Theta_\eps(n)$ \whp, \eqref{eqlemma_transform1} implies \eqref{eqlemma_transform}.
	Finally, $\psi$ is one-to-one because each vector has a unique representation with respect to the basis $(\fb_1,\ldots,\fb_{q-1})$.
\end{proof}

Roughly speaking, \Lem~\ref{lemma_transform} shows that any two tame $r,r'\in\fL_0(r_*,\eps)$ close to a conceivable $r_*\in\fL_0$ are about equally likely.
However, the map $\psi$ produces solutions that are a little less tame than the ones we start from.
The following corollary, which combines \Lem s~\ref{claim_tame3} and~\ref{lemma_transform}, remedies this issue.

\begin{corollary}\label{cor_transform4}
	\Whp\ for all $r_*\in\fL_0$ and all $r,r'\in\fL_0(r_*,\eps)$ we have
	\begin{align*}
		\pr_\fA\brk{\fT(r,-3\log\eps)\mid\fS}=(1+o_\eps(1))\pr_\fA\brk{\fT(r',-3\log\eps)\mid\fS}.
	\end{align*}
\end{corollary}
\begin{proof}
	Let $r^*$ be the vector supplied by \Lem~\ref{claim_tame3}.
	Applying \Lem~\ref{lemma_transform} to $r^*$ and $r\in\fL_0(r_*,\eps)$, we see that \whp
\begin{align}\label{cor_transform}
\pr_\fA\brk{\fT(r,-2\log\eps)\mid\fS}\geq(1+O_\eps(\eps\log\eps))\pr_\fA\brk{\fT(r^*,-\log\eps)\mid\fS}\geq\frac1{3|\fL_0(r_*,\eps)|}\qquad\mbox{for all }r\in\fL_0(r_*,\eps).
\end{align}

In addition, we claim that \whp
	\begin{align}\label{cor_transform2}
		\pr_{\fA}\brk{\fT(r,-4\log\eps)\setminus\fT(r,-3\log\eps)\mid\fS}\leq\eps\pr_{\fA}\brk{\fT(r^*,-\log\eps)\mid\fS}\qquad\mbox{for all }r\in\fL_0(r_*,\eps).
	\end{align}
Indeed, applying \Lem~\ref{lemma_transform} twice to $r$ and $r^*$ and invoking \eqref{eq_claim_tame3}, we see that \whp
	\begin{align}\nonumber
		\pr_{\fA}\brk{\fT(r,-2\log\eps)\mid\fS}&\geq\exp(O_\eps(\eps\log\eps))\pr_{\fA}\brk{\fT(r^*,-3\log\eps)\mid\fS}\\&\geq\bc{1-O_\eps(\eps\log\eps)}\pr_{\fA}\brk{\hat{\vec\rho}=r^*\mid\fS},\label{eq_cor_transform2_1}\\
			\pr_{\fA}\brk{\fT(r,-4\log\eps)\setminus\fT(r,-3\log\eps)\mid\fS}&\leq\exp(O_\eps(\eps\log\eps))\pr_{\fA}\brk{\fT(r^*,-3\log\eps)\setminus\fT(r^*,-2\log\eps)\mid\fS}\nonumber\\
																			 &\leq O_\eps(\eps^4)\pr_{\fA}\brk{\hat{\vec\rho}=r^*\mid\fS}.\label{eq_cor_transform2_2}
		\end{align}
		Combining \eqref{eq_cor_transform2_1} and \eqref{eq_cor_transform2_2} yields \eqref{cor_transform2}.

		Finally, \eqref{eq_claim_tame3}, \eqref{cor_transform} and \eqref{cor_transform2} show that \whp\ 
\begin{align}\label{cor_transform3}
		\pr_\fA\brk{\fT(-3\log\eps)\mid\hat{\vec\rho}=r,\,\fS}\geq1-\sqrt\eps,\quad\pr_\fA\brk{\fT(-3\log\eps)\mid\hat{\vec\rho}=r',\,\fS}\geq1-\sqrt\eps
		\qquad\mbox{for all }r,r'\in\fL_0(r_*,\eps),
	\end{align}
	and combining \eqref{cor_transform3} with \Lem~\ref{lemma_transform} completes the proof.
\end{proof}

\begin{proof}[Proof of \Lem~\ref{lemma_uniformly}]
	We are going to show that the conditional probability $\pr_\fA\brk{\hat{\vec\rho}=r\mid\fS}$ of hitting some particular $r\in\fL_0$ coincides with the unconditional probability $\pr_\fA\brk{\hat{\vec\rho}=r}$ up to a factor of $1+o_\eps(1)$.
	Then the assertion follows from Bayes' formula.

	The unconditional probability $\pr_\fA\brk{\hat{\vec\rho}=r}$ is given precisely by Claim~\ref{lemma_precise_R}.
	Hence, recalling the $(q-1)\times(q-1)$-matrix $\Sigma=q\id^{-1}-q^{-2}\vecone$ and applying Stirling's formula, we obtain 
	\begin{align}\label{eq_lemma_uniformly_1}
		\pr_\fA\brk{\mycheck{\vec\rho}=r}&\sim\frac1{(2\pi\vec\Delta q^{-1}(1-q^{-1}))^{(q-1)/2}}
		\exp\brk{-\frac{(r-q^{-1}\vec\Delta\vecone)^\trans(q^{-1}\vecone-q^{-2}\vecone)^{-1}(r-q^{-1}\vec\Delta\vecone)}{2\vec\Delta}}
	\end{align}
	\whp\ 

	Next we will show that the conditional probability $\pr_\fA\brk{\hat{\vec\rho}=r\mid\fS}$ works out to be asymptotically the same.
	Indeed, Claim~\ref{fact_boxS} shows that for any $r\in\fL_0$ the proability that $\hat{\vec\rho}$ hits the set $\fL_0(r,\eps)$ is asymptotically equal to the probability of the event $\{\|\vN-\vec\Delta^{-1/2}(r-q^{-1}\vec\Delta\vecone)\|_\infty<\eps\}$ \whp\
	Moreover, \Cor~\ref{cor_transform4} implies that given $\hat{\vec\rho}\in\fL_0(r,\eps)$, $\hat{\vec\rho}$ is within $o_\eps(1)$ of the uniform distribution on $\fL_0(r,\eps)$.
	Furthermore, \Lem~\ref{lemma_gridcount} and \Prop~\ref{prop_module} show that the number of points in $\fL_0(r,\eps)$ satisfies
	\begin{align*}
\frac{|\fL_0(r,\eps)|}{\abs{\cbc{z\in\ZZ^{q-1}:\|z-r\|_\infty\leq\eps\sqrt{\vec\Delta}}}}\sim q^{-\vecone\{|\supp\vec\chi|=1\}}.
		\end{align*}
	Therefore, \whp\ for all $r\in\fL_0$ we have
	\begin{align}\label{eq_lemma_uniformly_2}
		\pr_\fA\brk{\mycheck{\vec\rho}=r\mid\fS}&=(1+o_\eps(1))\frac{q^{\vecone\{|\supp\vec\chi|=1\}}}{(2\pi\vec\Delta q^{-1}(1-q^{-1}))^{(q-1)/2}} \exp\brk{-\frac{(r-q^{-1}\vec\Delta\vecone)^\trans(q^{-1}\vecone-q^{-2}\vecone)^{-1}(r-q^{-1}\vec\Delta\vecone)}{2\vec\Delta}}.
	\end{align}
	Finally, we observe that
	\begin{align}\label{eq_lemma_uniformly_3}
		\pr_{\fA}\brk{\fS}\sim q^{-\vm}.
		\end{align}
	Indeed, since the $\vec\xi_{ij}$ are uniform and independent, for each $i\in[\vm]$ we have $\sum_{j=1}^{\vk_i}\chi_{i,j}\xi_{ij}=0$ with probability $1/q$ indepdenently.
	Combining \eqref{eq_lemma_uniformly_1}--\eqref{eq_lemma_uniformly_3} completes the proof.
\end{proof}

\subsection{Proof of \Lem~\ref{lemma_hit}}\label{sec_hit}
We continue to denote by $\vec\sigma\in\FF_q^n$ a uniformly random assignment and by $\vec I_{q-1}$ the $(q-1)\times(q-1)$-identity matrix.
Also recall $\rho_\sigma$ from \eqref{eqrho} and for $\rho=(\rho(s))_{s\in\FF_q}$ obtain $\hat\rho=(\rho(s))_{s\in\FF_q^*}$ by dropping the $0$-entry.
The following claim, which we prove via the local limit theorem for sums of independent random variables, determines the distribution of $\rho_{\vec\sigma}$.
Let $\bar\rho=q^{-1}\vec\Delta \vecone_{q-1}$.

\begin{claim}\label{lemma_entropy_llt}
	Let $\cC$ be the $(q-1)\times(q-1)$-matrix from \eqref{eqCmatrix}.
	Then \whp\ for all $\rho\in\fP_q$ we have
	\begin{align*}
		\pr\brk{\rho_{\vec\sigma}=\rho \mid\alg}&=\frac{q^{q/2}\fd^{q-1}}{(2\ex[\vd^2]\pi n)^{(q-1)/2}}\exp\bc{-\frac{(\hat\rho-\bar\rho)^\trans\cC^{-1}(\hat\rho-\bar\rho)}{2n\ex[\vd^2]}}+o(n^{(1-q)/2}).
	\end{align*}
\end{claim}

\noindent
The proof of Claim~\ref{lemma_entropy_llt} is based on local limit theorem techniques similar to but simpler than the ones from \Sec~\ref{sec_uniformly}.
In fact, the proof strategty is somewhat reministcent of that of the well-known local limit theorem for sums of independent random vectors from~\cite{DavMc}.
However, the local theorem from that paper does not imply Claim~\ref{lemma_entropy_llt} directly because a key assumption (that increments of vectors in each direction can be realised) is not satisfied here.
We therefore carry the details out in the appendix.

Claim~\ref{lemma_entropy_llt} demonstrates that $\rho_{\vec\sigma}$ satisfies a local limit theorem.
Hence, let $\vec N'\in\RR^{q-1}$ be a mean-zero Gaussian vector with covariance matrix $q^{-1}\id-q^{-2}\vecone$.
Moreover, fix $\eps>0$ and let $U\subset\RR^{q-1}=v+[-\eps,\eps]^{q-1}$ be a box of side length $2\eps$.
Then \whp\ we have
\begin{align}\label{eqlemma_hit_0}
	\pr_{\fA}\brk{\bc{n\ex[\vd^2]}^{-1/2}(\mycheck{\rho}_{\vec\sigma}-q^{-1}\vec\Delta\vecone)\in U}=\pr_{\fA}\brk{\vN'\in U}+o(1).
\end{align}
Indeed, Claim~\ref{lemma_entropy_llt} implies that $\mycheck{\rho}_{\vec\sigma}$ is asymptotically uniformly distributed on the lattice points of the box $U$ whose coordinates are divisible by $\fd$ \whp\
Thus, \whp\ for any $z,z'\in\vec\Delta U\cap\fd\ZZ^{\FF_q^*}$ we have
\begin{align}\label{eqlemma_hit_1}
	\pr_{\fA}\brk{\mycheck{\rho}_{\vec\sigma}=z}&=(1+o_\eps(1))\pr_{\fA}\brk{\mycheck{\rho}_{\vec\sigma}=z'}.
\end{align}
Moreover, we claim that 
\begin{align}\label{eqlemma_hit_2}
	\pr_{\fA}\brk{\mycheck{\rho}_{\vec\sigma}\in\fL_0\mid\mycheck{\rho}_{\vec\sigma}\in U}\sim\frac{\abs{U\cap\fL_0\cap\fd\ZZ^{\FF_q^*}}}{\abs{U\cap\fd\ZZ^{\FF_q^*}}}\leq\frac{\abs{U\cap\fM\cap\fd\ZZ^{\FF_q^*}}}{\abs{U\cap\fd\ZZ^{\FF_q^*}}}&\leq(1+o(1))q^{-\vecone\{|\supp\vec\chi|=1\}}. 
\end{align}
Indeed, if $|\supp\vec\chi|>1$, then \eqref{eqlemma_hit_2} is satisfied \whp\ for the trivial reason that the r.h.s.\ equals $1+o(1)$.
Hence, suppose that $|\supp\vec\chi|=1$, let $\fM\supset\fL_0$ be the module from \Prop~\ref{prop_module} and let $\fb_1,\ldots,\fb_{q-1}$ be its assorted basis.
Clearly, $\fM\cap\fd\ZZ^{\FF_q^*}\supseteq\fd\fM$.
Conversely,  Cramer's rule shows that any $y\in \fM\cap\fd\ZZ^{\FF_q^*}$ can be expressed as
	$$(\fb_1\ \cdots\ \fb_{q-1})z,\qquad\mbox{with }z_i=\frac{\det(\fb_1\cdots\fb_{i-1}\ y\ \fb_{i+1}\cdots\fb_{q-1})}{q}.$$
	In particular, all coordinates $z_i$ are divisible by $\fd$ because $y\in\fd\ZZ^{\FF_q^*}$.
	Hence, $y\in\fd\fM$ because $\fd$ and $q$ are coprime.
	\Lem~\ref{lemma_gridcount} therefore implies~\eqref{eqlemma_hit_2}.
Finally, the assertion follows from \eqref{eqlemma_hit_0}--\eqref{eqlemma_hit_2}.

\section{Proof of \Prop~\ref{prop_auxphi}}\label{sec_prop_auxphi}

\noindent
We prove \Prop~\ref{prop_auxphi} by way of a coupling argument inspired by the Aizenman-Sims-Starr scheme from spin glass theory~\cite{Aizenman}.
The proof is a close adaptation of the coupling argument used in~\cite{Maurice} to prove the approximate rank formula~\eqref{eqMaurice}.
We will therefore be able to reuse some of the technical steps from that paper.
The main difference is that we need to accommodate the extra ternary equations $t_i$.
Their presence gives rise to the second parameter $\beta$ in \eqref{eq_prop_auxphi}.

\subsection{Overview}
The basic idea behind the Aizenman-Sims-Starr scheme is to compute the expected difference $\ex[\nul\A[n+1,\eps,\delta,\Theta]]-\ex[\nul\A[n+1,\eps,\delta,\Theta]]$ of the nullity upon increasing the size of the matrix.
We then obtain \eqref{eq_prop_auxphi} by writing a telescoping sum.
In order to estimate the expected change of the nullity, we set up a coupling of $\A[n,\eps,\delta,\Theta]$ and $\A[n+1,\eps,\delta,\Theta]$.

To this end it is helpful to work with a different description of the random matrix model.
Specifically, let $\vM=(\vM_j)_{j\geq1}$, $\DELTA=(\DELTA_j)_{j\geq1}$, $\vec \lambda$ and $\vec \eta$ be Poisson variables with means
\begin{align}\label{eqPoissons}
	\Erw[\vM_j]=(1-\eps)\pr\brk{\vk=j}dn/k, \qquad \quad \Erw[\DELTA_j]=(1-\eps)\pr\brk{\vk=j}d/k, \qquad \quad \Erw\brk{\vec \lambda}=\delta n, \qquad \quad \Erw\brk{\vec \eta}= \delta.
\end{align}
All these random variables are mutually independent and independent of $\THETA$ and the $(\vd_i)_{i\geq1}$.
Further, let
\begin{align}\label{eqm}
	\vM_j^+&=\vM_j+\DELTA_j,&
	\vm_{\eps,n}&=\sum_{j\geq1}\vM_j,&
	\vm_{\eps,n}^+&=\sum_{j\geq1}\vM_j^+,& 
	\vec \lambda^+ &= \vec \lambda + \vec \eta.
\end{align}
Since  $\sum_{j\geq1}\vM_j\disteq\Po((1-\eps)dn/k)$, (\ref{eqm}) is consistent with \eqref{eqms}. 

We define a random Tanner (multi-)graph $\G\brk{n,\vM,\vec \lambda}$ with variable nodes $x_1,\ldots,x_n$ 
and check nodes $a_{i,j}$, $i\geq1$, $j\in[\vM_i]$, $t_1, \ldots, t_{\vec \lambda}$ and $p_1,\ldots,p_{\THETA}$. 
The edges between variables and the check nodes $a_{i,j}$ are induced by a random maximal matching  $\vec\Gamma\brk{n,\vM}$ of the complete bipartite graph with vertex classes
\begin{align*}
	\bigcup_{h=1}^n\{x_h\}\times[\vd_h]\quad\mbox{and}\quad\bigcup_{i\geq1}\bigcup_{j=1}^{\vM_i}\{a_{i,j}\}\times[i].
\end{align*}
Moreover, for each $j \in [\vec \lambda]$ we choose $\vec i_{j,1}, \vec i_{j,2}, \vec i_{j,3}$ uniformly and independently from $[n]$ and add edges between $x_{\vec i_{j,1}}$, $x_{\vec i_{j,2}}$, $x_{\vec i_{j,3}}$ and $t_j$.
In addition, we insert an edge between $p_i$ and $x_i$ for every $i\in[\THETA]$.

To define the random matrix $\vA\brk{n,\vM,\vec\lambda}$ to go with $\G\brk{n,\vM,\vec\lambda}$, let
\begin{align}
\vA\brk{n,\vM,\vec \lambda}_{p_i,x_h}&=\vecone\cbc{i=h}&&(i\in[\THETA],h\in[n]),\\
\vA\brk{n,\vM,\vec \lambda}_{a_{i,j},x_h}&=\vec\chi_{i, h} \sum_{\ell=1}^i\sum_{s=1}^{\vd_h} \vecone \{(x_h, s), (a_{i,j}, \ell)\}\in\vec{\Gamma}_{n, \vM}\}&&(i\ge 1, j\in[\vM_i], h\in[n]),\\
\vA\brk{n,\vM,\vec\lambda}_{t_i,x_h}&=\vec\chi_{\vm_{\eps,n}+i, h} \sum_{\ell=1}^3\vecone\{\vec i_{i,\ell} = h\} &&(i\in[\vec{\lambda}],h\in[n]).\label{eq_ternary1}
\end{align}
The Tanner graph $\G\brk{n+1, \vM^+,\vec\lambda^+}$  and its associated random matrix $\vA\brk{n+1,\vM^+,\vec\lambda^+}$ are defined analogously using $n+1$ variable nodes instead of $n$, $\vM^+$ instead of $\vM$ and $\vec \lambda^+$ instead of $\vec \lambda$.

\begin{fact}\label{Fact_ComplicatedModel}
For any $\eps, \delta>0$ we have
\begin{align*}
	\Erw[\nul\vA\brk{n,\eps,\delta}]=\Erw[\nul\vA\brk{n,\vM,\vec\lambda}],&& 
	\Erw[\nul\vA\brk{n+1,\eps,\delta}]=\Erw\brk{\nul\vA\brk{n+1,\vM^+,\vec\lambda^+}}.
\end{align*}
\end{fact}
\begin{proof}
Because the check degrees $\vk_i$ of the random factor graph $\vec G\brk{n,\eps, \delta}$ are drawn independently,  the only difference between $\vG\brk{n,\eps,\delta}$ and $\vG\brk{n,\vM,\vec\lambda}$ is the bookkeeping of the number of checks of each degree.
The same is true of $\vG\brk{n+1,\eps,\delta}$ and $\vG\brk{n+1,\vM,\vec\lambda}$.
\end{proof}

To construct a coupling of $\vA\brk{n,\vM,\vec\lambda}]$ and $\vA\brk{n+1,\vM^+,\vec\lambda^+}$ we introduce a third, intermediate random matrix.
Hence, let $\GAMMA_i\ge0$ be the number of checks $a_{i,j}$, $j\in[\vM_i^+]$, adjacent to the last variable node $x_{n+1}$ in $\G\brk{n+1,\vM^+, \vec\lambda^+}$.
Set $\GAMMA=(\GAMMA_i)_{i\geq3}$.
Also let 
\begin{align}\label{eqlambda'}
	\lambda^- = \delta(n+1) - 3\delta\cdot\frac{n^2+n+1/3}{n^2+2n+1}
\end{align}
be the expected number of extra ternary checks of $\G\brk{n+1,\vM^+,\vec \lambda^+}$ in which $x_{n+1}$ does not appear. 
Let
\begin{equation}\label{eqminus}
	\vM_i^-=(\vM_i-\GAMMA_i)\vee 0, \qquad \qquad \text{as well as} \qquad \qquad \vec\lambda^- \sim \Po(\lambda^-).
\end{equation} 
Consider the random Tanner graph $\G'=\G\brk{n,\vM^-,\vec\lambda^-}$ induced by a random maximal matching $\vec\Gamma'=\vec\Gamma\brk{n, \vec M^-}$ of the complete bipartite graph with vertex classes
\begin{align}\label{cavities_creation}
	\bigcup_{h=1}^n\{x_h\}\times[\vd_h]\quad\mbox{and}\quad\bigcup_{i\geq1}\bigcup_{j=1}^{\vM_i^-}\{a_{i,j}\}\times[i].
\end{align}
Each matching edge $\{(x_h,s),(a_{i,j},\ell)\}\in \vec\Gamma\brk{n,\vM^-}$ induces an edge between $x_h$ and $a_{i,j}$ in the Tanner graph. 
For each $j \in [\vec \lambda^-]$ and $\vec i^-_{j,1}, \vec i^-_{j,2}, \vec i^-_{j,3}$ uniform and independent in $[n]$, we add the edges between $x_{\vec i^-_{j,1}}$, $x_{\vec i^-_{j,2}}$, $x_{\vec i^-_{j,3}}$ and $t_j$.
In addition, there is an edge between $p_i$ and $x_i$ for every $i\in[\THETA]$. 
Let $\vec A'$ denote the corresponding random matrix.

For each variable $x_i$, $i=1,\ldots,n$, let $\cC$ be the set of clones from $\bigcup_{i\in[n]}\{x_i\}\times[\vd_i]$ that $\vec\Gamma\brk{n,\vM^-}$ leaves unmatched.
We call the elements of $\mathcal{C}$ \textit{cavities}. 

From $\G'$, we finally construct two further Tanner graphs.
Obtain the Tanner graph $\G''$ from $\G'$ by adding new check nodes $a''_{i,j}$ for each $i\geq3$, $j\in[\vM_i-\vM_i^-]$ and ternary check nodes $t''_i$ for $i \in [\vec\lambda'']$, where 
\begin{align}\label{eqlambda''}
	\vec\lambda'' \sim \Po(\delta n - \lambda^-) =\Po\bc{2\delta\cdot\frac{n^2+n/2}{n^2+2n+1}}
\end{align}
The new checks $a_{i,j}''$ are joined by a random maximal matching $\vec\Gamma''$ of the complete bipartite graph on
	$$\mbox{$\cC\qquad$ and}\qquad\bigcup_{i\geq1}\bigcup_{j\in[\vM_i-\vM_i^-]}\{a_{i,j}''\}\times[i].$$
	Moreover, for each $j \in [\vec \lambda'']$ we choose $\vec i''_{j,1}, \vec i''_{j,2}, \vec i''_{j,3}\in[n]$ uniformly and independently of everything else and add the edges between $x_{\vec i''_{j,1}}$, $x_{\vec i''_{j,2}}$, $x_{\vec i''_{j,3}}$ and  $t_j''$.
Let $\vec A''$ denote the corresponding random matrix, where as before, each new edge is represented by an independent copy of $\vec \chi$.

Finally, let 
\begin{align}\label{eqlambda'''}
	\vec\lambda''' \sim \Po(\delta(n+1) - \lambda^-)=\Po\bc{ 3\delta\cdot\frac{n^2+n+1/3}{n^2+2n+1} }.
\end{align} 
	We analogously obtain $\G'''$ by adding one variable node $x_{n+1}$ as well as check nodes $a_{i,j}'''$, $i\geq1$, $j\in[\GAMMA_i]$, $b_{i,j}'''$, $i\geq1$, $j\in[\vM_i^+-\vM_i^--\GAMMA_i]$, $t_i''', i \in [\vec \lambda''']$.
The new checks $a_{i,j}'''$ and $b_{i,j}'''$ are connected to $\G'$ via a random maximal matching $\vec\Gamma'''$ of the complete bipartite graph on
 $$\mbox{ $\cC\qquad$ and }\qquad\bigcup_{i\geq1}\bc{\bigcup_{j\in[\GAMMA_i]}\{a_{i,j}'''\}\times[i-1]
			\cup\bigcup_{j\in[\vM_i^+-\vM_i^--\GAMMA_i]}\{b_{i,j}'''\}\times[i]}.$$
For each matching edge we insert the corresponding variable-check edge and in addition each of the check nodes $a_{i,j}'''$ gets connected to $x_{n+1}$ by exactly one edge. 
Then we connect each $t_i'''$ to $x_{\vec i'''_{i,1}}, x_{\vec i'''_{i,2}}$ and $x_{n+1}$, with $\vi'''_{i,1},\vi'''_{i,2}\in[n+1]$ chosen uniformly and independently. 
Once again each edge is represented by an independent copy of $\vec\chi$.
Let $\vec A'''$ denote the resulting random matrix. 

The following lemma connects $\vA'',\vA'''$ with the random matrices $\vA\brk{n,\vM, \vec \lambda}$, $\vA\brk{n+1,\vM^+, \vec \lambda^+}$ and thus, in light of Fact~\ref{Fact_ComplicatedModel}, with $\vA\brk{n,\eps,\delta}$ and $\vA\brk{n+1,\eps,\delta}$.

\begin{lemma}\label{Lemma_valid}
We have  $\Erw[\nul(\vA'')]=\Erw[\nul(\vA_{n,\vM, \lambda})]+o(1)$ and  $\Erw[\nul(\vA''')]=\Erw[\nul(\vA_{n+1,\vM^+, \lambda^+})]+o(1).$
\end{lemma}

\noindent
We defer the simple proof of \Lem~\ref{Lemma_valid} to \Sec~\ref{Sec_Lemma_valid}.

The core of the proof of \Prop~\ref{prop_auxphi} is to estimate the difference of the nullities of $\vA'''$ and $\vA'$ and of $\vA''$ and $\vA'$.
The following two lemmas express these differences in terms of two random variables $\vec\alpha,\vec\beta$.
Specifically, let $\ALPHA$ be the fraction of frozen cavities of $\vA'$ and let $\BETA$ be the fraction of frozen variables of $\vA'$.

\begin{lemma}\label{Lemma_A'''}
	For large enough $\Theta(\eps)>0$ and small enough $0<\delta<\delta_0$ we have
\begin{align*}
	\Erw[\nul(\vA''')-\nul(\vA')] &= \Erw\brk{\exp\bc{-3\delta \vec\beta^2} D(1-K'(\vec \alpha)/k)} + \frac{d}{k}\Erw\brk{K'(\vec\alpha)+K(\vec \alpha)} - \frac{d(k+1)}{k} - 3\delta \Erw\brk{1-\vec\beta^2}+o_{\eps}(1).
\end{align*}
\end{lemma}

\begin{lemma}\label{Lemma_A''}
	For large enough $\Theta(\eps)>0$ and small enough $0<\delta<\delta_0$ we have
$$\Erw[\nul(\vA'')-\nul(\vA')] = - d + \frac{d}{k} \Erw\brk{\vec\alpha K'(\vec \alpha)} - 2\delta \Erw\brk{1-\vec\beta^3} + o_{\eps}(1). $$
\end{lemma}

\noindent
After some preparations in \Sec~\ref{sec_prep} we will prove \Lem s~\ref{Lemma_A'''} and~\ref{Lemma_A''} in \Sec s~\ref{Sec_A'''} and \ref{Sec_A''}.

\begin{proof}[Proof of \Prop~\ref{prop_auxphi}]
	The proposition is an immediate consequence of Fact~\ref{Fact_ComplicatedModel}, \Lem~\ref{Lemma_valid}, \Lem~\ref{Lemma_A'''} and \Lem~\ref{Lemma_A''}.
\end{proof}

\subsection{Preparations}\label{sec_prep}
To facilitate the proofs of \Lem s~\ref{Lemma_A'''} and \ref{Lemma_A''} we establish a few basic statements about the coupling.
Some of these are immediate consequence of statements from~\cite{Maurice}, where a similar coupling was used.
Let us begin with the following lower bound on the likely number of cavities.

\begin{lemma}\label{Lemma_cavityCount}
\Whp\ we have $|\cC|\geq\eps dn/2$.
\end{lemma}
\begin{proof}
	Apart from the extra ternary check nodes $t_1, \ldots t_{\vec \lambda'}$, the construction of $\G'$ coincides with that of the Tanner graph from~\cite{Maurice}.
	Because the presence of $t_1, \ldots t_{\vec \lambda'}$ does not affect the number of cavities, the assertion therefore follows from \cite[Lemma 5.5]{Maurice}.
\end{proof}

As a next step we show that \whp\ the random matrix $\vA'$ does not have very many short linear relations.
Specifically, if we choose a bounded number of variables and a bounded number of cavities randomly, then it is quite unlikely that the chosen coordinates form a proper relation.
Formally, let $\cR(\ell_1,\ell_2)$ be the set of all sequences $(i_1,\ldots,i_{\ell_1})\in[n]^{\ell_1}$, $(u_1,j_1),\ldots,(u_{\ell_2},j_{\ell_2})\in\cC$ such that $(i_1,\ldots,i_{\ell_1},u_1,\ldots,u_{\ell_2})$ is a proper relation of $\vA'$.
Furthermore, let $\fR(\zeta,\ell)$ be the event that $|\cR(\ell_1,\ell_2)|\leq\zeta n^{\ell_1}|C|^{\ell_2}$ for all $0\leq \ell_1,\ell_2\leq\ell$.

\begin{lemma}\label{jointfactorlemma}
	For any $\zeta> 0$, $\ell>0$ exist $\Theta_0=\Theta_0(\eps,\zeta,\ell)>0$ and $n_0>0$ such that for all $n \geq n_0$, $\Theta\geq\Theta_0$ we have $\pr\brk{\fR(\zeta,\ell)}>1-\zeta$.
\end{lemma}
\begin{proof}
	Fix any $\ell_1,\ell_2\leq\ell$ such that $\ell_1+\ell_2>0$ and let $\fR(\zeta,\ell_1,\ell_2)$ be the event that $|\cR(\ell_1,\ell_2)|<\zeta n^{\ell_1}|C|^{\ell_2}$.
	Then it suffices to show that $\pr\brk{\fR(\zeta,\ell_1,\ell_2)}>1-\zeta$ as we can just replace $\zeta$ by $\zeta/(\ell+1)^2$ and apply the union bound.
	To this end we may assume that $\zeta<\zeta_0(\eps,\ell)$ for a small enough $\zeta_0(\eps,\ell)>0$.

	We will actually estimate $|\cR(\ell_1,\ell_2)|$ on a certain likely event.
	Specifically, due to Lemma \ref{Lemma_cavityCount} we have $|\cC|\geq \eps n/2$ \whp\
	In addition, let $\cA$ be the event that $\vA'$ is $(\zeta^4/L^\ell,\ell)$-free.
	Then \Lem~\ref{Prop_Alp} shows that $\pr\brk\cA>1-\zeta/3$, provided that $n\geq n_0$ for a large enough $n_0=n_0(\zeta,\ell)$.
	To see this, consider the matrix $\vB$ obtained from $\vA'$ by deleting the rows representing the unary checks $p_i$.
	Then \Lem~\ref{Prop_Alp} shows that the matrix $\vB[\vec\theta]$ obtained from $\vB$ via the pinning operation is $(\zeta^4,L^\ell)$-free with probability $1-\zeta/3$, provided that $\Theta$ is chosen sufficiently large.
	The only difference between $\vB[\vec\theta]$ and $\vA'$ is that in the former random matrix we apply the pinning operation to $\vec\theta$ random coordinates, while in $\vA'$ the unary checks $p_i$ pin the first $\vec\theta$ coordinates.
	However, the distribution of $\vA'$ is actually invariant under permutations of the columns.
	Therefore, the matrices $\vA'$ and $\vB[\vec\theta]$ are $(\zeta^4,L^\ell)$-free with precisely the same probability.
	Hence, \Lem~\ref{Prop_Alp} implies that $\pr\brk\cA>1-\zeta/3$.

	Further, Markov's inequality shows that for any $L>0$,
\begin{align*}
	\pr\brk{\sum_{i=1}^n \vec d_i\vecone\{\vec d_i> L\}\geq\frac{\eps\zeta^2 n}{16\ell}} \leq \frac{16 \ell \Erw\brk{\vec d\vecone\{\vec d> L\}}}{\eps \zeta^2}.
\end{align*}
Therefore, since $\Erw\brk{\vec d} = O_\eps(1)$ we can choose $L=L(\eps, \zeta,\ell)>0$ big enough such that the event $$\cL=\cbc{\sum_{i=1}^n \vec d_i\vecone\{\vec d_i> L\}< \frac{\eps\zeta^2 n}{16\ell}}$$ has probability at least $1-\zeta/3$.
	Thus, the event $\cE=\cA\cap\cL\cap\{|\cC|\geq \eps n/2\}$ satisfies $ \pr\brk{\cE}>1-\zeta.  $
Hence, suffices to show that
\begin{align}\label{Prob_to_expect}
	|\cR(\ell_1,\ell_2)|<\zeta n^{\ell_1}|\cC|^{\ell_2}&&\mbox{ if the event $\cE$ occurs.}
\end{align}

To bound $\cR(\ell_1,\ell_2)$ on $\cE$ we need to take into consideration that the cavities are degree-weighted.
Hence, let $\cR'(\ell_1,\ell_2)$ be the set of all sequences $(i_1,\ldots,i_{\ell_1},(u_1,j_1),\ldots,(u_{\ell_2},j_{\ell_2}))\in\cR(\ell_1,\ell_2)$ such that the degree of some variable node $u_i$ exceeds $L$.
Assuming $\ell_2>0$, on $\cE$ we have
\begin{align}\label{Prob_to_expect_2}
	|\cR'(\ell_1,\ell_2)|\leq n^{\ell_1+1}|\cC|^{\ell_2-1}\cdot\frac2\eps\sum_{i=1}^n \vec d_i\vecone\{\vec d_i> L\}\leq n^{\ell_1}|\cC|^{\ell_2}\cdot\frac2\eps\cdot\frac{\zeta^2n}{16\ell_2}<\frac{\zeta}2,
\end{align}
provided that $\zeta>0$ is small enough.

Finally, we bound the size of $\cR''(\ell_1,\ell_2)=\cR(\ell_1,\ell_2)\setminus\cR'(\ell_1,\ell_2)$.
Since for any $(i_1,\ldots,i_{\ell_1},(u_1,j_1),\ldots,(u_{\ell_2},j_{\ell_2}))\in\cR''(\ell_1,\ell_2)$ the sequence $(i_1,\ldots,i_{\ell_1},u_1,\ldots,u_{\ell_2})$ is a proper relation and since there are no more than $L^{\ell_2}$ ways of choosing the indices $j_1,\ldots,j_{\ell_2}$, on the event $\cE$ we have
\begin{align}
	|\cR''(\ell_1,\ell_2)|&\leq \frac{\zeta^4}{L^\ell}\cdot L^{\ell_2}n^\ell&&\mbox{[because $\vA'$ is $\zeta^4/L^\ell,\ell)$-free]}\nonumber\\
						&\leq \zeta^4\bcfr2\eps^{\ell_2}\cdot n^{\ell_1}|\cC|^{\ell_2}&&\mbox{[because $|\cC|>\eps n/2$]}\nonumber\\
						&<\frac\zeta2n^{\ell_1}|\cC|^{\ell_2},\label{Prob_to_expect_3}
\end{align}
provided that $\zeta<\zeta_0(\eps,\ell)$ is sufficiently small.
Thus, \eqref{Prob_to_expect} follows from \eqref{Prob_to_expect_2} and \eqref{Prob_to_expect_3}.
\end{proof}

Let $(\hat{\vk}_i)_{i \geq 1}$ be a sequence of copies of $\hat{\vk}$, mutually independent and independent of everything else. 
Also let 
\begin{align*}
	\hat{\vec \gamma}_j = \sum_{i=1}^{\vd_{n+2}} \vecone\cbc{\hat{\vk}_i = j},&&
\hat{\vec \gamma} = (\hat{\vec \gamma}_j)_{j \geq 1}.
\end{align*}
Additionally, let $(\hat{\vec\Delta}_j)_{j \geq 3}$ be a family of independent random variables with distribution
\begin{align}\label{eqhatDelta}
\hat{\vec \Delta}_j = \Po\bc{(1-\eps) \Pr\brk{\vk=j}d/k}.
\end{align}
Further, let $\Sigma'$ be the $\sigma$-algebra generated by $\vG',\vA',\vec\theta,\vec\lambda^-,\vM^-,\vec \Gamma_{n,M^-}, (\vec \chi_{i,j,h}')_{i,j,h \geq 1}$ and $(\vd_i)_{i \in [n]}$.
In particular, $\ALPHA$ and $\vec\beta$ are $\Sigma'$-measurable.

\begin{lemma}\label{Lem_cond_distr}
	With probability $1-\mathrm{exp}\bc{-\Omega_\eps(1/\eps)}$, we have   
	$$\dTV{\bc{\Pr\brk{\cbc{\vec\gamma \in \cdot \ } \vert \Sigma'}, \hat{\vec \gamma}}} + \dTV{\bc{\Pr\brk{\cbc{\vec\Delta \in \cdot \ } \vert \Sigma'}, \hat{\vec \Delta}}} = O_\eps(\sqrt{\eps}).$$
\end{lemma}
\begin{proof}
	Because $\G'$ is distributed the same as the Tanner graph from~\cite{Maurice}, apart from the extra ternary checks $t_i$, which do not affect the random vector $\vec\gamma$, the assertion follows from \cite[Lemma 5.8]{Maurice}.
\end{proof}

Let $\ell_*=\lceil\exp(1/\eps^4)\rceil$ and $\delta_*=\exp(-1/\eps^4)$ and consider the event
\begin{align}\label{eqE}
	\cE&=\fR(\delta_*,\ell_*).
\end{align}
Further, consider the event
\begin{align}\label{eqE'}
\cE'=\cbc{|\cC|\geq\eps dn/2\wedge \max_{i\leq n}\vd_i\leq n^{1/2}}.
\end{align}

\begin{corollary}\label{Cor_EE'}
For sufficiently large $\Theta=\Theta(\eps)>0$ we have $\pr\brk{\vA'\in\cE}>\exp(-1/\eps^4)$.
Moreover, $\pr\brk{\cE'}=1-o(1)$.
\end{corollary}
\begin{proof}
	The first statement follows from \Lem~\ref{jointfactorlemma}.
	The second statement follows from the choice of the parameters in~\eqref{eqPoissons},  \Lem~\ref{Lemma_sums} and \Lem~\ref{Lemma_cavityCount}.
\end{proof}

\noindent
With these preparations in place we are ready to proceed to the proofs of \Lem s~\ref{Lemma_A'''} and~\ref{Lemma_A''}.

\subsection{Proof of \Lem~\ref{Lemma_A'''}}\label{Sec_A'''} 
Let
	\begin{align*}
		\vX&=\sum_{i\geq1}\DELTA_i,& \vY&=\sum_{i\geq1}i\DELTA_i,& \vY'&=\sum_{i\geq1}i\GAMMA_i.
	\end{align*}
Then the total number of new non-zero entries upon going from $\vA'$ to $\vA'''$ is bounded by $\vY+\vY'+3\vec\lambda'''$.
Let $$\cE''=\cbc{\vX\vee \vY\vee \vY'\vee\vec\lambda'''\leq1/\eps}.$$

\begin{claim}\label{Claim_A'''2a}
We have $\pr\brk{\cE''}=1-O_{\eps}(\eps)$.
\end{claim}
\begin{proof}
	Apart from the additional ternary checks the argument is similar to~\cite[Proof of Claim~5.9]{Maurice}.
	The construction~\eqref{eqPoissons} ensures that $\Erw[\vX],\Erw[\vY]=O_{\eps}(1)$.
	Therefore,  $\pr\brk{\vX>1/\eps}=O_{\eps}(\eps)$, $\pr\brk{\vY>1/\eps}=O_{\eps}(\eps)$ by Markov's inequality.
	Further, a given check node of degree $i$ is adjacent to $x_{n+1}$ with probability at most $i\vd_{n+1}/\sum_{i=1}^n\vd_i\geq n\leq i\vd_{n+1}/n$.
	Consequently,
	\begin{align*}
		\Erw\brk{\vY'}=\Erw\sum_{i\geq1}i\GAMMA_i\leq\Erw\sum_{i\in[\vm_{\eps,n}^+]}\vk_i^2\vd_{n+1}/n=O_{\eps}(1).
	\end{align*}
	Moreover, \eqref{eqlambda'''} shows that $\ex[\vec\lambda''']=O_\eps(1)$.
	Thus, the assertion follows from Markov's inequality.
\end{proof}

We obtain $\G'''$ from $\G'$ by adding checks $a_{i,j}'''$, $i\geq1$, $j\in[\GAMMA_i]$, $b_{i,j}'''$, $i\geq1$, $j\in[\vM_i^+-\vM_i^--\GAMMA_i]$ and $t_i'''$, $i\in[\vec\lambda''']$.
Let 
$$\cX'''=\bc{\bigcup_{i\geq1}\bigcup_{j=1}^{\GAMMA_i}\partial a_{i,j}'''\setminus\{x_{n+1}\}}\cup\bc{\bigcup_{i\geq1}\bigcup_{j\in[\vM_i^+-\vM_i^--\GAMMA_i]}\partial b_{i,j}'''}\cup\bigcup_{i=1}^{\vec\lambda'''}\partial t_i'''\setminus\{x_{n+1}\}$$
be the set of variable neighbours of these new checks among $x_1,\ldots,x_n$.
Further, let 
	$$\cE'''=\cbc{|\cX|=Y+\sum_{i\geq1}(i-1)\GAMMA_i+\vec\lambda'''}$$
be the event that the variables of $\G'$ where the new checks connect are pairwise distinct.

\begin{claim}\label{Claim_A'''2b}
	We have $\pr\brk{\cE'''\mid\cE'\cap\cE''}=1-o(1).$
\end{claim}
\begin{proof}
	By the same token as in~\cite[proof of Claim~5.10]{Maurice}, given that $\cE'$ occurs the total number of cavities comes to $\Omega(n)$.
	At the same time, the maximum variable node degree is of order $O(\sqrt n)$.
	Moreover, given the event $\cE''$ no more than $Y+Y'=O_{\eps}(1/\eps)$ random cavities are chosen as neighbours of the new checks $a_{i,j}''',b_{i,j}'''$.
	Thus, by the birthday paradox the probability that the checks $a_{i,j}''',b_{i,j}'''$ occupy more than one cavity of any variable node is $o(1)$.
	Furthermore, the additional ternary nodes $t_i'''$ choose their two neighbours among $x_1,\ldots,x_n$ mutually independently and independently of the $a_{i,j}''',b_{i,j}'''$.
	Since $\vec\lambda'''$ is bounded given $1/\eps$, the overall probability of choosing the same variable twice is $o(1)$.
\end{proof}

The following claim shows that the unlikely event that $\cE\cap\cE'\cap\cE''\cap\cE'''$ does not occur does not contributed significantly to the expected change in nullity.

\begin{claim}\label{Claim_A'''3}
We have $\Erw\brk{\abs{\nul(\vA''')-\nul(\vA')}(1-\vecone\cE\cap\cE'\cap\cE''\cap\cE''')}=o_{\eps}(1)$.
\end{claim}
\begin{proof}
	We modify the proof of~\cite[Claim~5.11]{Maurice} to accommodate the extra ternary nodes.
	Since $\vA'''$ results from $\vA'$ by adding one column and no more than $\vX+\vd_{n+1}+\vec\lambda'''$ rows, we have $\abs{\nul(\vA''')-\nul(\vA')}\leq X+\vd_{n+1}+\vec\lambda'''+1$.
	Because $\vX,\vd_{n+1}^2,{\vec\lambda'''}$ have bounded second moments, the Cauchy-Schwarz inequality therefore yields the estimate
\begin{align}
\Erw\brk{\abs{\nul(\vA''')-\nul(\vA')}(1-\vecone\cE'')}&\leq
		\Erw\brk{(X+\vd_{n+1}+\vec\lambda'''+1)^2}^{1/2}\bc{1-\pr\brk{\cE''}}^{1/2}=o_{\eps}(1).
	\label{eqClaim_A'''3_1}
\end{align}
Moreover, combining \Cor~\ref{Cor_EE'} and Claims~\ref{Claim_A'''2a}--\ref{Claim_A'''2b}, we obtain
\begin{align}	\label{eqClaim_A'''3_2}
\Erw&\brk{\abs{\nul(\vA''')-\nul(\vA')}\vecone\cE''\setminus\cE}	
	\leq O_{\eps}(\eps^{-1})\exp(-1/\eps^4)=o_{\eps}(1),\\
\Erw&\brk{\abs{\nul(\vA''')-\nul(\vA')}\vecone\cE''\setminus\cE'},\Erw\brk{\abs{\nul(\vA''')-\nul(\vA')}\vecone\cE''\cap\cE'\setminus\cE'''}=o(1).\label{eqClaim_A'''3_3}
\end{align}
The assertion follows from \eqref{eqClaim_A'''3_1}--\eqref{eqClaim_A'''3_3}.
\end{proof}

Recall that $\vec\alpha$ denotes the fraction of frozen cavities and $\BETA$ the fraction of frozen variables of $\vA'$.
Further, let $\Sigma''\supset\Sigma'$ be the $\sigma$-algebra generated by $\vec\theta$, $\G'$, $\A'$, $\vM_-$, $(\vd_i)_{i\in[n+1]}$, $\GAMMA$, $\vM$, $\DELTA$, $\vec\lambda^-,\vec\lambda'''$.
Then $\ALPHA,\BETA$ as well as $\cE,\cE',\cE''$ are $\Sigma''$-measurable but $\cE'''$ is not.

\begin{claim}\label{Claim_A'''6}
	On the event $\cE\cap\cE'\cap\cE''$ we have
	\begin{align*}
		\Erw\brk{\bc{\nul(\vA''')-\nul(\vA')}\vecone\cE'''\mid\Sigma''}=o_{\eps}(1)&+(1-\BETA^2)^{\vec\lambda'''}\prod_{i\geq1}(1-\ALPHA^{i-1})^{\GAMMA_i}-\sum_{i\geq1}(1-\vec\alpha^{i-1})\GAMMA_i-\vec\lambda'''(1-\BETA^2)\\
		&-\sum_{i\geq1}(1-\vec\alpha^{i})(\vM_i^+-\vM_i^--\GAMMA_i).
	\end{align*}
\end{claim}
\begin{proof}
	We modify the proof of~\cite[Claim~5.12]{Maurice} by taking the additional ternary checks into consideration.
Let
\begin{align*}
	\cA&=\cbc{a_{i,j}''':i\geq1,\ j\in[\GAMMA_i]},&
	\cB&=\cbc{b_{i,j}''':i\geq1,\ j\in[\vM_i^+-\vM_i^--\GAMMA_i]},&
	\cT&=\cbc{t_i:i\in[\vec\lambda''']}.
\end{align*}
We set up a random matrix $\vB$ with rows indexed by $\cA\cup\cB\cup\cT$ and columns indexed by $V_n=\{x_1,\ldots,x_n\}$.
For a check $a\in\cA\cup\cB\cup\cT$ and a variable $x\in V_n$ the $(a,x)$-entry of $\vB$ equals zero unless $x\in\partial_{\G'''} a$.
Further, the non-zero entries of $\vB$ are independent copies of $\vec\chi$.
Additionally, obtain $\vB_*$ from $\vB$ by zeroing out the $x$-column for every variable $x\in\fF(\vA')$.
Finally, let $\vC\in\FF^{\cA\cup\cB\cup\cT}$ be a random vector whose entries $\vC_a$, $a\in\cA\cup\cT$, are independent copies of $\vec\chi$, while $\vC_b=0$ for all $b\in\cB$.

If $\cE'''$ occurs, $\vB$ has row full rank because there is at most one non-zero entry in every column and at least one non-zero entry in every row.
Hence,
\begin{align*}
 \rk(\vB)=|\cA\cup\cB\cup\cT|=\sum_{i\geq1}\vM_i^+-\vM_i^-+\vec\lambda'''.
\end{align*}
Furthermore, since the rank is invariant under row and column permutations, given $\cE\cap\cE'\cap\cE''\cap\cE'''$ we have
\begin{align*}
\nul\vA'''&=\nul\begin{pmatrix}\vA'&0\\\vB&\vC\end{pmatrix}.
\end{align*}
Moreover, given $\cE'$ the set $\cX'''$ of all non-zero columns of $\vB$ satisfies $|\cX'''|\leq Y+Y'+\vec\lambda'''\leq3/\eps$ while $|\cC|\geq\eps dn/2$.
Therefore, the set of cavities that $\vec\Gamma'''$ occupies is within total variation distance $o(1)$ of a commensurate number of cavities drawn independently, i.e., with replacement.
Furthermore, the variables where the checks from $\cT$ attach are chosen uniformly at random from $x_1,\ldots,x_n$.
Therefore, on $\cE\cap\cE'\cap\cE''$ the conditional probability given $\cE'''$ that $\cX'''$ forms a proper relation is bounded by $O_{\eps}(\exp(-1/\eps^4))$.
Consequently, \Lem~\ref{Cor_free} implies that on the event $\cE\cap\cE'\cap\cE''$,
\begin{align}\label{eqClaim_A'''6_2}
	\Erw\brk{\bc{\nul(\vA''')-\nul(\vA')}\vecone\cE'''\mid\Sigma''}&=1-\Erw\brk{\rk\bc{\vB_*\ \vC}\mid\Sigma''}+o_{\eps}(1).
\end{align}

We are thus left to calculate the rank of $\vQ=\bc{\vB_*\ \vC}$.
Given $\cE'''$ this block matrix decomposes into the $\cA\cup\cT$-rows $\vQ_{\cA\cup\cT}$ and the $\cB$-rows $\vQ_{\cB}$ such that $\rk(\vQ)=\rk(\vQ_{\cA\cup\cT})+\rk(\vQ_{\cB})$.
Therefore, it suffices to prove that
\begin{align}\label{eqQ'}
\Erw\brk{\rk\bc{\vQ_{\cB}}\mid\Sigma''}&=\sum_{i\geq1}\bc{1-\ALPHA^i}(\vM_i^+-\vM_i^--\GAMMA_i)+o(1),\\
\Erw\brk{\rk(\vQ_{\cA\cup\cT})\mid\Sigma''}&=\lambda'''(1-\BETA^2)+\sum_{i\geq1}\bc{1-\ALPHA^{i-1}}\GAMMA_i+1-(1-\BETA^2)^{\vec\lambda'''}\prod_{i\geq1}\bc{1-\ALPHA^{i-1}}^{\GAMMA_i}+o(1).\label{eqQ''}
\end{align}

Towards \eqref{eqQ'} consider a check $b\in\cB$ whose corresponding row sports $i$ non-zero entries.
Since we may pretend (up to $o(1)$ in total variation) that these entries are drawn uniformly and independently from the set of cavities, the probability that they are all frozen comes to $\vec\alpha^i+o(1)$.
Since there are $\vM_i^+-\vM_i^--\GAMMA_i$ such checks $b\in\cB$, we obtain \eqref{eqQ'}.

Moving on to \eqref{eqQ''}, consider $a\in\cA$ whose corresponding row has $i-1$ non-zero entries.
By the same token as in the previous paragraph, the probability that all entries in the $a$-row correspond to frozen cavities comes to $\ALPHA^{i-1}+o(1)$.
Hence, the expected rank of the $\cA\times V_n$-minor works out to be $\sum_{i\geq1}\bc{1-\ALPHA^{i-1}}\GAMMA_i+o(1)$, which is the second summand in \eqref{eqQ''}.
Similarly, a $t\in\cT$-row adds to the rank unless both the variables in the corresponding check are frozen.
The latter event occurs with probability $\vec\beta^2$. 
Hence the first summand.
Finally, the $\vC$-column adds to the rank if none of the $\cA\cup\cT$-rows become all-zero, which occurs with probability $(1-\BETA^2)^{\vec\lambda'''}\prod_{i\geq1}\bc{1-\ALPHA^{i-1}}^{\GAMMA_i}+o(1)$.
\end{proof}

\begin{proof}[Proof of \Lem~\ref{Lemma_A'''}]
Letting $\fE=\cE\cap\cE'\cap\cE''\cap\cE'''$ and combining Claims~\ref{Claim_A'''2a}--\ref{Claim_A'''6}, we obtain
\begin{align}\nonumber
	\Erw&\Big|\Erw\brk{\nul(\vA''')-\nul(\vA')\mid\Sigma''}-\Big((1-\BETA^2)^{\vec\lambda'''}\prod_{i\geq1}(1-\ALPHA^{i-1})^{\GAMMA_i} -\sum_{i\geq1}(1-\vec\alpha^{i-1})\GAMMA_i \\
		&\qquad-\sum_{i\geq1}(1-\vec\alpha^{i})(\vM_i^+-\vM_i^--\GAMMA_i)-\vec\lambda'''(1-\BETA^2)\Big)\vecone\fE\Big| =o_{\eps}(1).\label{eqLemma_A'''_0}
\end{align}
On $\fE$ all $i$ with $\vM_i^+-\vM_i^--\GAMMA_i>0$ are bounded.
Moreover, \whp\ we have $\vM_i\sim\Erw[\vM_i]=\Omega(n)$ for all bounded $i$ by Chebyshev's inequality.
Hence, \eqref{eqminus} implies that $\vM_i^-=\vM_i-\GAMMA_i$ \whp\
Consequently, \eqref{eqLemma_A'''_0} becomes
\begin{align}\nonumber
	\Erw\Big|\Erw\brk{\nul(\vA''')-\nul(\vA')\mid\Sigma''}&-\Big((1-\BETA^2)^{\vec\lambda'''}\prod_{i\geq1}(1-\ALPHA^{i-1})^{\GAMMA_i} -\sum_{i\geq1}(1-\vec\alpha^{i-1})\GAMMA_i \\
		&-\sum_{i\geq1}(1-\vec\alpha^{i})\vec\Delta_i-\vec\lambda'''(1-\BETA^2)\Big)\vecone\fE\Big| =o_{\eps}(1).\label{eqLemma_A'''_1}
\end{align}

We proceed to estimate the various terms on the r.h.s.\ of \eqref{eqLemma_A'''_1} separately.
Since $\pr\brk\fE=1-o_{\eps}(1)$ by \Cor~\ref{Cor_EE'} and Claims \ref{Claim_A'''2a} and \ref{Claim_A'''2b}, \Lem~\ref{Lem_cond_distr} yield
\begin{align}\nonumber
	\ex\brk{ \vecone\fE\cdot(1-\BETA^2)^{\vec\lambda'''}\prod_{i\geq1}(1-\ALPHA^{i-1})^{\GAMMA_i}\mid\Sigma''}&=\ex\brk{(1-\BETA^2)^{\vec\lambda'''}\prod_{i\geq1}(1-\vec\alpha^{i-1})^{\hat{\vec\gamma_i}}\mid\Sigma''}+o_{\eps}(1)\\
																											  &=\exp(-3\delta\vec\beta^2)D(1-K'(\vec\alpha)/k)&\mbox{[by \eqref{eqSizeBiasd} and \eqref{eqlambda'''}]}.
																											  \label{eqseparate1}
\end{align}
Moreover, since $\sum_{i\geq1}\GAMMA_i\leq\vd_{n+1}$ and $\vd_{n+1}$ has a bounded second moment, \Lem~\ref{Lem_cond_distr} implies that
\begin{align}
	\ex\brk{\vecone\fE\cdot\sum_{i\geq1}(1-\vec\alpha^{i-1})\GAMMA_i\mid\Sigma''}&=\ex\brk{\sum_{i\geq1}(1-\vec\alpha^{i-1})\hat\GAMMA_i\mid\Sigma''}+o_{\eps}(1)=d-\frac dkK'(\vec\alpha)+o_{\eps}(1).\label{eqseparate2}
\end{align}
Further, by Claim~\ref{Claim_A'''2a}, \Lem~\ref{Lem_cond_distr} and \eqref{eqhatDelta},
\begin{align}\label{eqseparate3}
	\Erw\brk{\vecone\fE\cdot\sum_{i\geq1}(1-\vec\alpha^{i})\DELTA_i\mid\Sigma''}&=\Erw\brk{\sum_{i\geq1}(1-\vec\alpha^{i})\DELTA_i\mid\Sigma''}+o_{\eps}(1)=o_{\eps}(1)+\frac dk-\frac dk\Erw[K(\ALPHA)].
\end{align}
Finally, \eqref{eqlambda'''} yields
\begin{align}\label{eqseparate4}
	\Erw\brk{\vecone\fE\cdot\vec\lambda'''(1-\vec\beta^2)\mid\Sigma''}&=3\delta(1-\vec\beta^2)+o_{\eps}(1).
\end{align}
Thus, the assertion follows from \eqref{eqLemma_A'''_1}--\eqref{eqseparate4}.
\end{proof}

\subsection{Proof of \Lem~\ref{Lemma_A''}}\label{Sec_A''}
We proceed similarly as in the proof of \Lem~\ref{Lemma_A'''}; actually matters are a bit simpler because we only add checks, while in the proof of \Lem~\ref{Lemma_A'''} we also had to deal with the extra variable node $x_{n+1}$.
Let $\cE,\cE'$ be the events from \eqref{eqE} and \eqref{eqE'} and let $\cE''=\cbc{\vd_{n+1}+\vec\lambda''\leq1/\eps}$.
As a direct consequence of the assumption $\Erw[\vd_{n+1}^2]=O_{\eps,n}(1)$ and of \eqref{eqlambda''}, we obtain the following.

\begin{fact}\label{Claim_A''_0}
We have $\pr\brk{\cE''}=1-O_{\eps}(\eps^2).$
\end{fact}

Let
	$$\cX''=\bigcup_{i\geq1}\bigcup_{j\in[\vM_i-\vM_i^-]}\partial_{\G''}a_{i,j}''\cup\bigcup_{i=1}^{\vec\lambda''}\partial t_i''$$ 
	be the set of variable nodes where the new checks that we add upon going from $\vA'$ to $\vA''$ attach.
Let $\cE'''$ be the event that in $\G''$ no variable from $\cX''$ is connected with the checks $\{a_{i,j}'':i\geq1,j\in[\vM_i-\vM_i^-]\}\cup\{t_i'':i\in[\vec\lambda'']\}$ by more than one edge.

\begin{claim}\label{Claim_A''_E'''}
We have $\pr\brk{\cE'''\mid\cE'\cap\cE''}=1-o(1).$
\end{claim}
\begin{proof}
	This follows from the ``birthday paradox'' (see the proof of Claim~\ref{Claim_A'''2b}).
\end{proof}

\begin{claim}\label{Claim_A''_1}
We have $\Erw\brk{\abs{\nul(\vA'')-\nul(\vA')}(1-\vecone\cE\cap\cE'\cap\cE''\cap\cE''')}=o_{\eps}(1)$.
\end{claim}
\begin{proof}
We have $\abs{\nul(\vA'')-\nul(\vA')}\leq \vd_{n+1}+\vec\lambda''$ as we add at most $\vd_{n+1}+\vec\lambda''$ rows.
Because $\Erw[(\vd_{n+1}+\vec\lambda'')^2]=O_{\eps}(1)$ by~\eqref{eqlambda''}, Claim~\ref{Claim_A''_0} and the Cauchy-Schwarz inequality yield
\begin{align}\label{eqClaim_A''_1_1}
	\Erw\brk{\abs{\nul(\vA'')-\nul(\vA')}(1-\vecone\cE'')}&\leq
	\Erw\brk{(\vd_{n+1}+\vec\lambda'')^2}^{1/2}(1-\pr\brk{\cE})^{1/2}=o_{\eps}(1).
\end{align}
Moreover, \Cor~\ref{Cor_EE'} and Claim~\ref{Claim_A''_E'''} show that
\begin{align}\label{eqClaim_A''_1_2}
	\Erw\brk{\abs{\nul(\vA'')-\nul(\vA')}\vecone\cE''\setminus\cE}, \Erw\brk{\abs{\nul(\vA'')-\nul(\vA')}\vecone\cE''\setminus\cE'}, \Erw\brk{\abs{\nul(\vA'')-\nul(\vA')}\vecone\cE''\setminus\cE'''} =o_{\eps}(1).
\end{align}
The assertion follows from \eqref{eqClaim_A''_1_1} and\eqref{eqClaim_A''_1_2}.
\end{proof}

The matrix $\vA''$ results from $\vA'$ by adding checks $a_{i,j}''$, $i\geq1$, $j\in[\vM_i-\vM_i^-]$ that are connected to random cavities of $\A'$.

Moreover, as before let $\Sigma''\supset\Sigma'$ be the $\sigma$-algebra generated by $\vec\theta$, $\G'$, $\A'$, $\vM_-$, $(\vd_i)_{i\in[n+1]}$, $\GAMMA$, $\vM$, $\DELTA$, $\vec\lambda^-,\vec\lambda'''$.
Then $\cE,\cE',\cE''$ are $\Sigma''$-measurable, but $\cE'''$ is not.

\begin{claim}\label{Claim_A''_6}
	On $\cE\cap\cE'\cap\cE''$ we have
	$$\Erw\brk{(\nul(\vA'')-\nul(\vA'))\vecone\cE'''\mid{\Sigma''}}=o_{\eps}(1)-\sum_{i\geq1}(1-\vec\alpha^{i})(\vM_i-\vM_i^-)-\vec\lambda''(1-\BETA^3).$$
\end{claim}
\begin{proof}
	Let $\cA=\{a_{i,j}'':i\geq1,\ j\in[\vM_i-\vM_i^-]\}$.
	Moreover, let $\cT$ be the set of new ternary checks $t_i''$, $i\in[\vec\lambda'']$.
	Let $\vB$ be the $\FF_q$-matrix whose rows are indexed by $\cA\cup\cT$ and whose columns are indexed by $V_n=\{x_1,\ldots,x_n\}$.
	The $(a,x)$-entry of $\vB$ is zero unless $a,x$ are adjacent in $\G''$, in which case the entry is an independent copy of $\vec\chi$.
	Given $\cE'''$ the matrix $\vB$ has full row rank $\rk(\vB)=|\cA|=\vec\lambda''+\sum_{i\geq1}\vM_i^+-\vM_i$, because no column contains two non-zero entries and each row has at least one non-zero entry.
	Further, obtain $\vB_*$ from $\vB$ by zeroing out the $x$-column of every $x\in\fF(\vA')$.

	On $\cE\cap\cE'\cap\cE''\cap\cE'''$ we see that
	\begin{align}\label{eqClaim_A''6_1}
		\nul\vA''&=\nul\begin{pmatrix}\vA'\\\vB\end{pmatrix}.
	\end{align}
	Moreover, let $\cI$ be the set of non-zero columns of $\vB$.
	Then on $\cE'\cap\cE''$ we have $|\cI|\leq \vd_{n+1}+\vec\lambda''\leq1/\eps$.
	Hence, on $\cE\cap\cE'\cap\cE''\cap\cE'''$ the probability that $\cI$ forms a proper relation is bounded by $\exp(-1/\eps^4)$.
	Hence, \Lem~\ref{Cor_free} shows
	\begin{align}\label{eqClaim_A''6_2}
		\Erw\brk{\bc{\nul(\vA'')-\nul(\vA')}\vecone\cE'''\mid\Sigma''}&=o_{\eps}(1)-\Erw\brk{\rk\bc{\vB_*}\mid{\Sigma''}}.
	\end{align}

	We are thus left to calculate the rank of $\vB_*$.
	Recalling that $\ALPHA$ stands for the fraction of frozen cavities, we see that for $a\in\cA$ of degree $i$ the $a$-row is all-zero in $\vB_*$ with probability $\ALPHA^i+o(1)$.
	Similarly, for $a\in\cT$ the $a$-row of $\vB$ gets zeroed out with probability $\BETA^3$.
	Hence, we conclude that
	\begin{align}\label{eqClaim_A''6_3}
		\Erw\brk{\rk\bc{\vB_*}\mid\Sigma''}&=o_{\eps}(1)+\vec\lambda''(1-\BETA^3)+\sum_{i\geq1}\bc{1-\ALPHA^i}(\vM_i-\vM_i^-).
	\end{align}
	Combining \eqref{eqClaim_A''6_2} and \eqref{eqClaim_A''6_3} completes the proof.
\end{proof}

\begin{proof}[Proof of \Lem~\ref{Lemma_A''}]
	Let $\fE=\cE\cap\cE'\cap\cE''\cap\cE'''$.
	Combining Claims~\ref{Claim_A''_1}--\ref{Claim_A''_6}, we see that
	\begin{align}\label{eqAddingVar01110101}
		\Erw\abs{\Erw[\nul(\vA'')-\nul(\vA')\mid{\Sigma''}]+\bc{\vec\lambda''(1-\BETA^3)+\sum_{i\geq1}(1-\vec\alpha^{i})(\vM_i-\vM_i^-)}\vecone\fE}	&=o_{\eps}(1).
	\end{align}
	On $\fE$ all degrees $i$ with $\vM_i^+-\vM_i^->0$ are bounded.
	Moreover, $\vM_i^-=\Omega(n)$ \whp\ for every bounded $i$ by Chebyshev's inequality.
	Therefore, \eqref{eqminus} shows that $\vM_i-\vM_i^-=\GAMMA_i$ for all $i$ with $\vM_i^+-\vM_i^->0$ \whp\
	Hence, \eqref{eqAddingVar01110101} turns into
	\begin{align}\label{eqAddingVar011}
		\Erw\abs{\Erw[\nul(\vA'')-\nul(\vA')\mid{\Sigma''}]+\bc{\vec\lambda''(1-\BETA^3)+\sum_{i\geq1}(1-\vec\alpha^{i})\GAMMA_i}\vecone\fE}	&=o_{\eps}(1).
	\end{align}

	We now estimate the two parts of the last expression separately.
	Since $\pr\brk{\fE}=1-o_{\eps}(1)$ by \Cor~\ref{Cor_EE'}, Fact~\ref{Claim_A''_0} and Claim \ref{Claim_A''_E'''}, the definition \eqref{eqlambda''} of $\vec\lambda''$ yields
	\begin{align}\label{eqAddingVar011}
		\Erw\abs{\vec\lambda''(1-\BETA^3)\vecone\fE}&=2\delta(1-\ex[\BETA^3])+o_{\eps}(1).
	\end{align}
	Moreover, because $\sum_{i\geq1}\GAMMA_i\leq\vd_{n+1}$, $\Erw[\vd_{n+1}]=O_{\eps}(1)$,
	\begin{align}
		\Erw\brk{\sum_{i\geq1}(1-\vec\alpha^{i})\GAMMA_i\vecone\fE}&=
		\Erw\brk{\sum_{i\geq1}(1-\vec\alpha^{i})\hat\GAMMA_i\vecone\cbc{\vec\lambda''+\sum_{i\geq1}\hat\GAMMA_i\leq\eps^{-1/4}}}+o_{\eps}(1)
																   &&\mbox{[by \Lem~\ref{Lem_cond_distr} and Claim~\ref{Claim_A''_0}]}\nonumber\\
																   &=d\Erw[1-\ALPHA^{\hat\vk}]+o_{\eps}(1)=d-d\Erw[\vec\alpha K'(\vec\alpha)]/k+o_{\eps}(1)
																   &&\mbox{[by \eqref{eqSizeBiasd}].} \label{eqAddingVar012}
	\end{align}
	Combining \eqref{eqAddingVar011} and~\eqref{eqAddingVar012} completes the proof.
\end{proof}

\subsection{Proof of \Lem~\ref{Lemma_valid}}\label{Sec_Lemma_valid}
The proof is relatively straightforward, not least because once again we can reuse some technical statements from~\cite{Maurice}.
Let us deal with $\vA''$ and $\vA'''$ separately.

\begin{claim}\label{Claim_Lemma_valid_1}
We have  $\Erw[\nul(\vA'')]=\Erw[\nul(\vA_{n,\vM, \lambda})]+o(1)$.
\end{claim}
\begin{proof}
The matrix models $\vA_{n,\vM, \lambda}$ and $\vA''$ coincide with the corresponding models from \cite[Claim 5.17]{Maurice}, except that here we add extra ternary checks.
Because these extra checks are added independently, the coupling from \cite[Claim 5.17]{Maurice} directly induces a coupling of the enhanced models by attaching the same number $\vec\lambda''$ of ternary equations to the same neighbors.
\end{proof}

\begin{claim}\label{Claim_Lemma_valid_2}
We have  $\Erw[\nul(\vA''')]=\Erw[\nul(\vA_{n+1,\vM^+, \lambda^+})]+o(1)$.
\end{claim}
\begin{proof}
The matrix models $\vA_{n+1,\vM^+, \lambda^+}$ and $\vA''''$ coincide with the corresponding models from \cite[Section 5.5]{Maurice} plus the extra independent ternary equations.
Hence, the coupling from \cite[Claim 5.17]{Maurice} yields a coupling of the enhanced models just as in Claim~\ref{Claim_Lemma_valid_1}.
\end{proof}

\begin{proof}[Proof of \Lem~\ref{Lemma_valid}]
	The lemma is an immediate consequence of Claims~\ref{Claim_Lemma_valid_1} and~\ref{Claim_Lemma_valid_2}.
\end{proof}

\section{Appendix}

\noindent
In this appendix we give a self-contained proof of \Lem~\ref{lemma_entropy_llt}, the local limit theorem for sums of independent vectors.
We employ a simplified version of the strategy of the proof of \Lem~\ref{lemma_uniformly}.
Recall that we assume the existence of a constant $\eta>0$ such that $\Erw[\vec d^{2+\eta}] + \Erw[\vec k^{2+\eta}] < \infty$.

Given $\omega>0$, we choose $\eps_0=\eps_0(\omega,q)$ sufficiently small and let $0<\eps<\eps_0$. With these parameters, we set
\begin{align}\label{def_normal}
s_n := \sqrt{\sum_{i=1}^n \vec d_i^2}.
\end{align}
As in the proof of \Lem~\ref{lemma_entropy_llt}, given $\omega>0$, we choose $\eps_0=\eps_0(\omega,q)$ sufficiently small and let $0<\eps<\eps_0$. With these parameters, we set
\begin{align*}
	\fL_0&=\cbc{r\in \ZZ^{\FF_q^\ast}: \Pr_\fA\bc{\mycheck{\rho}_{\vec \sigma} =r}>0 \text{ and }  \left\| r- \frac{\vec \Delta}{q} \vecone \right\|_1<\omega n^{-1/2} \vec\Delta} \quad \text{ and }\\
	\fL_0(r_*,\eps)&=\cbc{r\in\fL_0:\|r-r_*\|_\infty<\eps s_n}.
\end{align*} 
Then \begin{align*}\fL_0 \subseteq \fd \ZZ^{\FF_q^\ast}.\end{align*}

We begin by observing that the vector $\mycheck{\rho}_{\vec \sigma}$ is asymptotically normal given $\fA$.
As before we let $\vec I_{q-1}$ the $(q-1)\times(q-1)$-identity matrix and let $\vN\in\RR^{\FF_q^*}$ be a Gaussian vector with zero mean and covariance matrix 
	\begin{align}\label{eqCmatrix}
		\cC=q^{-1}\vec I_{q-1}-q^{-2}\vecone_{(q-1)\times(q-1)}.
	\end{align}

\begin{claim}\label{fact_box2}
	There exists a function $\iota=\iota_{\eta,q}(n)=o(1)$ such that for all axis-aligned cubes $U\subset\RR^{\FF_q^*}$ we have
	$$\ex\abs{\pr_{\fA}\brk{\frac{\mycheck{\rho}_{\vec \sigma}-q^{-1}\vec\Delta\vecone}{s_n}\in U}-\pr\brk{\vN\in U}}\leq \iota.$$
\end{claim}
\begin{proof}
	The mean of each entry $\hat\rho_{\vec\sigma}(\tau)$  clearly equals $\vec\Delta/q$ for every $\tau\in\FF_q^*$.
	Concerning the covariance matrix, for distinct $s\neq t$ we obtain 
	\begin{align*}
		\ex_{\fA}[\hat{\rho}_{\vec\sigma}^2(s)]&=\sum_{i,j\in[n]:i\neq j}\frac{\vd_i\vd_j}{q^2}+\sum_{i=1}^n\frac{\vd_i^2}q=\sum_{i,j=1}^n\frac{\vd_i\vd_j}{q^2}+\sum_{i=1}^n\frac{\vd_i^2}q\bc{1-\frac1q},\\
		\ex_{\fA}[\hat{\rho}_{\vec\sigma}(s)\hat{\rho}_{\vec\sigma}(t)]&=\sum_{i,j\in[n]:i\neq j}\frac{\vd_i\vd_j}{q^2}=\sum_{i,j=1}^n\frac{\vd_i\vd_j}{q^2}-\sum_{i=1}^n\frac{\vd_i^2}{q^2}.
	\end{align*}
	Hence, the means and covariances of $(\hat\rho_{\vec\sigma}-q^{-1}\vec\Delta\vecone)/s_n$ and $\vN$ match.

	We are thus left to prove that $(\hat\rho_{\vec\sigma}-q^{-1}\vec\Delta\vecone)/s_n$ is asymptotically normal, with the required uniformity.
	Thus, given a small $\xi>0$ we pick $\fD=\fD(q,\eta,\xi)>0$ and $n_0=n_0(\fD)$ sufficiently large.
	Suppose $n>n_0$ and let
	\begin{align*}
		\vd_i'&=\vecone\{\vd_i\leq\fD\}\vd_i,&\vd_i''&=\vd_i-\vd_i',\\
		\hat\rho_{\vec\sigma}'(s)&=\sum_{i=1}^n\vecone\{\vec\sigma_i=s\}\vd_i',&
		\hat\rho_{\vec\sigma}''(s)&=\sum_{i=1}^n\vecone\{\vec\sigma_i=s\}\vd_i'',\\
		{s_n'}^2&=\sum_{i=1}^n{\vd_i'}^2,&{s_n''}^2&=\sum_{i=1}^n{\vd_i''}^2,\\
		\vec\Delta'&=\sum_{i=1}^n\vd_i',&\vec\Delta''&=\sum_{i=1}^n\vd_i''.
	\end{align*}
	Since $\ex[\vd^{2+\eta}]<\infty$, by Markov's inequality and by construction we have \whp
	\begin{align}\label{eqpedestrian3}
		\vec\Delta''&<\xi^8 n,&\vec\Delta&=\vec\Delta'+\vec\Delta'',&
		{s_n''}^2&<\xi^8 n,& {s_n'}^2&<\fD^2n,&s_n^2&={s_n'}^2+{s_n''}^2,
	\end{align}
	providing $\fD$ is large enough.
	Hence, the multivariate Berry--Esseen theorem (e.g., \cite{Raic}) shows that \whp\ for all $U$,
	\begin{align}\label{eqpedestrian4}
		\pr_{\fA}\brk{\frac{\mycheck{\rho}'_{\vec \sigma}-q^{-1}\vec\Delta'\vecone}{s_n'}\in U}-\pr\brk{\vN\in U}&=O(n^{-1/2}).
	\end{align}
	Furthermore, combining \eqref{eqpedestrian3} with Chebyshev's inequality, we see that \whp
	\begin{align}\label{eqpedestrian5}
		\pr_{\fA}\brk{\abs{\frac{\mycheck{\rho}''_{\vec \sigma}-q^{-1}\vec\Delta''\vecone}{s_n}}>\xi^2}&<\xi^2.
	\end{align}
	Thus, combining \eqref{eqpedestrian4} and \eqref{eqpedestrian5}, we conclude that \whp
	\begin{align}\label{eqpedestrian6}\abs{\pr_{\fA}\brk{\frac{\mycheck{\rho}_{\vec \sigma}-q^{-1}\vec\Delta\vecone}{s_n}\in U}-\pr\brk{\vN\in U}}\leq \xi.\end{align}
	Finally, the assertion follows from \eqref{eqpedestrian6} by taking the limit $\xi\to0$ slowly enough as $n\to\infty$.
\end{proof}

There exist $g \in \NN$, $a_1, \ldots, a_g \in \ZZ$ and $\delta_1, \ldots, \delta_g$ in the support of $\vd$ such that the greatest common divisor of the support can be linearly combined as
	\begin{align}
	\fd = \sum_{i=1}^g a_i \delta_i.
	\end{align}
We next count how many variables there are with degree $\delta_i$. For $i \in [g]$, let $\fI_i$ denote the set of all $j \in [n]$ with $\vd_j=\delta_i$ (the set of all variables that appear in $\delta_i$ equations). Set $\fI_0 = [n] \setminus \bc{\fI_1 \cup \ldots \cup \fI_g}$. Then 
\begin{align*}
\fI_0 \cup \ldots \cup \fI_g = [n]
\end{align*}
and
$|\fI_0|,|\fI_1|, \ldots, |\fI_g|=\Theta(n)$ \whp\ (if $\fI_0$ is non-empty).
We further count how many entries of value $s \in \FF_q^{\ast}$ all variables of degree $\delta_i$ generate under the assignment $\vec \sigma$, and the contribution from the rest, yielding
\begin{align*}
	\vr_0(s)&=\sum_{j\in\fI_0} \vd_j \vecone\cbc{\vec \sigma_j = s},&
	\vr_i(s)&= \sum_{j\in\fI_i} \vecone\cbc{\vec \sigma_j = s}. &&(i \in [g], s\in\FF_q^*)
\end{align*}
Then summing the contributions, we get back $\mycheck{\rho}_{\vec\sigma}=\vr_0 + \sum_{i=1}^g \delta_i \vr_i$, where  $\vr_i=(\vr_i(s))_{s\in\FF_q^*}$.


Because $\vec\sigma_1, \ldots, \vec\sigma_n$ are mutually independent given $\fA$, so are $\vr_0, \vr_1, \ldots, \vr_g$.
Moreover, given $\fA$, for $i \in [g]$, $\vec r_i$ has a multinomial distribution with parameter $|\fI_i|$ and uniform probabilities $q^{-1}$.
In effect, the individual entries $\vec r_i(s)$, $s \in \FF_q^{\ast}$, will typically differ by only a few standard deviations, i.e., their typical difference will be of order $O(\sqrt{|\fI_i|})$.
We require a precise quantitative version of this statement.


Furthermore, we say that $\vec r_i$ is {\em $t$-tame} if $|\vec r_i(s)-q^{-1}|\fI_i||\leq t\sqrt{|\fI_i|}$ for all $s \in \FF_q^\ast$.
Let $\fT(t)$ be the event that $\vec r_1, \ldots, \vec r_g$ are $t$-tame.


\begin{lemma}\label{claim_tame3_sigma}
	\Whp\ for every $r_*\in\fL_0$ there exists $r^*\in\fL_0(r_*,\eps)$ such that
	\begin{align}\label{eq_claim_tame3_sigma}
	\pr_{\fA}\brk{\mycheck{\rho}_{\vec \sigma}=r^*}&\geq\frac{1}{2|\fL_0(r_*,\eps)|}&\mbox{and}&& \pr_{\fA}\brk{\fT(-\log\eps)\mid \mycheck{\rho}_{\vec \sigma}=r^*}&\geq1-\eps^4.
\end{align}
\end{lemma}
\begin{proof}
	Since $ \vec r_i$ has a multinomial distribution given $\fA$ the Chernoff bound shows that for a large enough $c=c(q)$ \whp\ 
	\begin{align}\label{claim_tame_sigma}
		\pr_{\fA}\brk{\fT(-\log\eps)}\geq1-\exp(-\Omega_\eps(\log^2(\eps))).
	\end{align}
	Further, Claim~\ref{fact_box2} implies that \whp\ $\pr_{\fA}\brk{\mycheck{\rho}_{\vec \sigma}\in\fL_0(r_*,\eps)}\geq \Omega_\eps(\eps^{q-1})\geq\eps^q$, provided $\eps<\eps_0=\eps_0(\omega)$ is small enough.
	Combining this estimate with \eqref{claim_tame_sigma} and Bayes' formula, we conclude that \whp\ for every $r_*\in\fL_0$,
	\begin{align}\label{claim_tame2_sigma}
	\pr_{\fA}\brk{\fT(-\log\eps),\ \mycheck{\rho}_{\vec \sigma}\in\fL_0(r_*,\eps)}\geq1-\eps^{5}.
\end{align}

To complete the proof, assume that there does not exist $r^*\in\fL_0(r_*,\eps)$ that satisfies \eqref{eq_claim_tame3_sigma}.
	Then for every $r\in\fL_0(r_*,\eps)$ we either have
	\begin{align}\label{eqX0_sigma}
		\pr_{\fA}\brk{\mycheck{\rho}_{\vec\sigma}=r}&<\frac{1}{2|\fL_0(r_*,\eps)|}&&\mbox{or}\\
		\pr_{\fA}\brk{\fT(-\log\eps) \vert \mycheck{\rho}_{\vec\sigma}=r}&<1-\eps^4.\label{eqX1_sigma}
	\end{align}
	Let $\fX_0$ be the set of all $r\in\fL_0(r_*,\eps)$ for which \eqref{eqX0_sigma} holds, and let $\fX_1=\fL_0(r_*,\eps)\setminus\fX_0$.
	Then \eqref{eqX0_sigma}--\eqref{eqX1_sigma} yield
	\begin{align*}\nonumber
	\pr_{\fA}\brk{\fT(-\log\eps)\mid \mycheck{\rho}_{\vec \sigma}\in\fL_0(r_*,\eps)}
		&\leq \frac{\pr_{\fA}\brk{\mycheck{\rho}_{\vec \sigma}\in\fX_0}+\sum_{r\in\fX_1}\pr_{\fA}\brk{\fT(-\log\eps) \vert\mycheck{\rho}_{\vec\sigma}=r}\pr_{\fA}\brk{\mycheck{\rho}_{\vec \sigma}=r}}{\pr_{\fA}\brk{\fL_0(r_*,\eps)}}\\
		&\leq\frac{\pr_{\fA}\brk{\mycheck{\rho}_{\vec\sigma}\in\fX_0}+(1-\eps^4)\pr_{\fA}\brk{\mycheck{\rho}_{\vec\sigma}\in\fX_1}}{\abs{\fL_0(r_*,\eps)}}< 1-\eps^4,
	\end{align*} 
	provided that $1-\eps^4>\frac{1}{2}$, in contradiction to \eqref{claim_tame2_sigma}.
\end{proof}

Also let $\fT(r,t)$ be the event that $\mycheck{\rho}_{\vec \sigma}=r$ and that  $\vec r_1, \ldots, \vec r_g$ are $t$-tame. 
We write $(r_0, \ldots, r_g)\in\fT(r,t)$ if $r_0 + \sum_{i=1}^g \delta_i r_i =r$ and $|r_i(s)-q^{-1}|\fI_i||\leq t\sqrt{|\fI_i|}$ for all $s \in \FF_q^\ast$.
The following lemma summarises the key step of the proof of \Lem~\ref{lemma_uniformly}.

\begin{lemma}\label{lemma_transform_sigma}
	\Whp\ for any $r_*\in\fL_0$, any $1\leq t\leq\log n$ and any $r,r'\in\fL_0(r_*,\eps)$ there exists a one-to-one map $\psi:\fT(r,t)\to\fT(r',t+O_\eps(\eps))$ such that for all $(r_0, \ldots, r_g)\in\fT(r,t)$ we have 
	\begin{align}\label{eqlemma_transform}
		\log\frac{\pr_\fA\brk{(\vr_0, \ldots,\vr_g)=(r_0, \ldots, r_g)}}{\pr_\fA\brk{(\vr_0,, \ldots, \vr_r)=\psi(r_0,\ldots,r_g)}}=O_\eps(\eps(\omega+t)).
	\end{align}
\end{lemma}
\begin{proof}
	Since $r,r'\in \fL_0(r_*,\eps)$, we have $r-r'\in\fd \ZZ^{\FF_q^\ast}$.
	Hence, with $e_1, \ldots, e_{q-1}$ denoting the standard basis of $\RR^{\FF_q^{\ast}}$, there is a unique representation 
	\begin{align}\label{eqlemma_transform00_sigma}
		r'-r&=\sum_{i=1}^{q-1}\lambda_i\fd e_i
	\end{align}
	with $\lambda_1, \ldots, \lambda_{q-1} \in\ZZ$.
	Because $r,r'\in\fL_0(r_*,\eps)$ and 
\begin{align*}
\lambda:=\begin{pmatrix}\lambda_1\\\vdots\\\lambda_{q-1}\end{pmatrix}=\fd^{-1}(r'-r),
	\end{align*}
	the coefficients satisfy 
	\begin{align}\label{eqlemma_transform0_sigma}
		|\lambda_i|&=O_\eps\bc{\eps s_n}\qquad\mbox{for all $i=1,\ldots,q-1$.}
	\end{align}
	Now recall $g \in \NN$, $a_1, \ldots, a_g \in \ZZ$ and $\delta_1, \ldots, \delta_g$ in the support of $\vd$ with
	\begin{align*}
	\fd = \sum_{i=1}^g a_i \delta_i.
	\end{align*}
	Fir $i \in [g]$, we set
	\begin{align*}
		r_i'&=  r_i+\frac{a_i }{\fd} \lambda 
	\end{align*}
	as well as $ r_0' = r_0$. Further, define $\psi(r_0, \ldots, r_g)=(r_0', \ldots, r_g')$.
	Then clearly
	\begin{align}\label{eqlemma_transform000_sigma}
	r_0 + \sum_{i=1}^g \delta_i  r_i' = r + \sum_{i=1}^g \frac{a_i \delta_i}{\fd} \lambda = r  + r' - r =r'.
		\end{align}
	and due to  \eqref{eqlemma_transform0_sigma}, we have $\psi(r_0,\ldots,r_g)\in\fT(r',t+O_\eps(\eps))$. Finally, for $i \in [g]$ set
	\begin{align*}
	r_i(0) = |\fI_i|-\sum_{s\in \FF_q^\ast}r_i(s), \qquad r_i'(0) = |\fI_i|-\sum_{s\in \FF_q^\ast}r_i'(s).
			\end{align*}
	Moreover, Stirling's formula and the mean value theorem show that
	\begin{align}
		\frac{\pr_\fA\brk{(\vr_0,\ldots,\vr_g)=(r_0, \ldots, r_g)}}{\pr_\fA\brk{(\vr_0, \ldots,\vr_g)=\psi(r_0, \ldots, r_g)}}&=\frac{\binom{|\fI_1|}{(r_1(0), r_1)} \cdot \ldots \cdot \binom{|\fI_g|}{ (r_g(0), r_g) }}{\binom{|\fI_1|}{(r_1'(0),r_1')} \cdot \ldots \cdot \binom{|\fI_g|}{(r_g'(0), r_g')} }=\exp\brk{\sum_{i=1}^g\sum_{s \in \FF_q}O_\eps\bc{r_i'(s)\log r'_i(s)-r_i(s)\log r_i(s)}}\nonumber\\
																								   &=\exp\brk{\sum_{i=1}^gO_\eps(|\fI_i|)\sum_{s \in \FF_q}\abs{\int_{r_i(s)/|\fI_i|}^{r_i'(s)/|\fI_i|}\log z\dd z}}\nonumber\\
																								   &=\exp\brk{\sum_{i=1}^gO_\eps(|\fI_i|)\sum_{s \in \FF_q} \bc{\frac{r_i'(s)}{|\fI_i|}-\frac{r_i(s)}{|\fI_i|}}\log\bc{\frac1q+O_\eps\bcfr{(\omega+t)s_n}{|\fI_i|}}}\nonumber\\
																								   &=\exp\brk{ \sum_{i=1}^gO_\eps(|\fI_i|)\sum_{s \in \FF_q} O_\eps\bc{\frac{(\omega+t)s_n}{|\fI_i|}\bc{\frac{r_i'(s)}{|\fI_i|}-\frac{r_i(s)}{|\fI_i|}}}}.\label{eqlemma_transform1}
	\end{align}
	Since $|\fI_1|, \ldots, |\fI_g|=\Theta_\eps(n)$ \whp, \eqref{eqlemma_transform1} implies \eqref{eqlemma_transform}.
	Finally, $\psi$ is one-to-one because each vector has a unique representation with respect to the basis $(e_1,\ldots,e_{q-1})$.
\end{proof}

Roughly speaking, \Lem\ref{lemma_transform} shows that any two tame $r,r'\in\fL_0(r_*,\eps)$ close to a conceivable $r_*\in\fL_0$ are about equally likely.
However, the map $\psi$ produces solutions that are a little less tame than the ones we start from.
The following corollary, which combines \Lem s~\ref{claim_tame3} and~\ref{lemma_transform}, remedies this issue.

\begin{corollary}\label{cor_transform4}
	\Whp\ for all $r_*\in\fL_0$ and all $r,r'\in\fL_0(r_*,\eps)$ we have
	\begin{align*}
		\pr_\fA\brk{\fT(r,-3\log\eps)}=(1+o_\eps(1))\pr_\fA\brk{\fT(r',-3\log\eps)}.
	\end{align*}
\end{corollary}
\begin{proof}
	Let $r^*$ be the vector supplied by \Lem~\ref{claim_tame3_sigma}.
	Applying \Lem~\ref{lemma_transform_sigma} to $r^*$ and $r\in\fL_0(r_*,\eps)$, we see that \whp
\begin{align}\label{cor_transform_sigma}
\pr_\fA\brk{\fT(r,-2\log\eps)}\geq(1+O_\eps(\eps\log\eps))\pr_\fA\brk{\fT(r^*,-\log\eps)}\geq\frac1{3|\fL_0(r_*,\eps)|}\qquad\mbox{for all }r\in\fL_0(r_*,\eps).
\end{align}

In addition, we claim that \whp
	\begin{align}\label{cor_transform2_sigma}
		\pr_{\fA}\brk{\fT(r,-4\log\eps)\setminus\fT(r,-3\log\eps)}\leq\eps\pr_{\fA}\brk{\fT(r^*,-\log\eps)}\qquad\mbox{for all }r\in\fL_0(r_*,\eps).
	\end{align}
Indeed, applying \Lem~\ref{lemma_transform} twice to $r$ and $r^*$ and invoking \eqref{eq_claim_tame3}, we see that \whp
	\begin{align}\nonumber
		\pr_{\fA}\brk{\fT(r,-2\log\eps)}&\geq\exp(O_\eps(\eps\log\eps))\pr_{\fA}\brk{\fT(r^*,-3\log\eps)}\\&\geq\bc{1-O_\eps(\eps\log\eps)}\pr_{\fA}\brk{\hat{\rho}_{\vec\sigma}=r^*},\label{eq_cor_transform2_1_sigma}\\
			\pr_{\fA}\brk{\fT(r,-4\log\eps)\setminus\fT(r,-3\log\eps)}&\leq\exp(O_\eps(\eps\log\eps))\pr_{\fA}\brk{\fT(r^*,-3\log\eps)\setminus\fT(r^*,-2\log\eps)}\nonumber\\
																			 &\leq O_\eps(\eps^4)\pr_{\fA}\brk{\hat{\rho}_{\vec\sigma}=r^*}.\label{eq_cor_transform2_2_sigma}
		\end{align}
		Combining \eqref{eq_cor_transform2_1_sigma} and \eqref{eq_cor_transform2_2_sigma} yields \eqref{cor_transform2_sigma}.

		Finally, \eqref{eq_claim_tame3}, \eqref{cor_transform_sigma} and \eqref{cor_transform2_sigma} show that \whp\ 
\begin{align}\label{cor_transform3_sigma}
		\pr_\fA\brk{\fT(-3\log\eps)\mid\hat{\rho}_{\vec\sigma}=r}\geq1-\sqrt\eps,\quad\pr_\fA\brk{\fT(-3\log\eps)\mid\hat{\rho}_{\vec \sigma}=r'}\geq1-\sqrt\eps
		\qquad\mbox{for all }r,r'\in\fL_0(r_*,\eps),
	\end{align}
	and combining \eqref{cor_transform3_sigma} with \Lem~\ref{lemma_transform_sigma} completes the proof.
\end{proof}

\begin{proof}[Proof of \Lem~\ref{lemma_entropy_llt}]


	Claim~\ref{fact_box2} shows that for any $r\in\fL_0$ and $\vec N \sim \mathcal{N}(0, \mathcal C)$
	\begin{align*}
	\pr_\fA\bc{\mycheck{\rho}_{\vec\sigma} \in \fL_0(r,\eps)} = \pr_\fA\bc{\norm{ \vN - \frac{r - \vec\Delta \vecone/q}{s_n }}_{\infty} < \eps} + o(1).
	\end{align*}
	Moreover, \Cor~\ref{cor_transform4} implies that given $\hat{\rho}_{\vec\sigma}\in\fL_0(r,\eps)$, $\hat{\rho}_{\vec\sigma}$ is within $o_\eps(1)$ of the uniform distribution on $\fL_0(r,\eps)$.
	Furthermore, \Lem~\ref{lemma_gridcount} 
	shows that the number of points in $\fL_0(r,\eps)$ satisfies
	\begin{align*}
\frac{|\fL_0(r,\eps)|}{\abs{\cbc{z\in\ZZ^{q-1}:\|z-r\|_\infty\leq\eps s_n}}}\sim \fd^{1-q}.
		\end{align*}
	Finally, the eigenvalues of the matrix $\cC$ are $q^{-2}$ (once) and $q^{-1}$ ($(q-2)$ times).
	Hence, $\det\cC=q^{-q}$. Therefore, \whp\ for all $r\in\fL_0$ we have
	\begin{align}\label{eq_lemma_uniformly_2_sigma}
		\pr_\fA\brk{\mycheck{\rho}_{\vec\sigma}=r}&=(1+o_\eps(1))\frac{q^{q/2}\fd^{q-1}}{(2\pi \sum_{i=1}^n\vec d_i^2 )^{(q-1)/2}} \exp\brk{-\frac{(r-q^{-1}\vec\Delta\vecone)^T\mathcal C^{-1}(r-q^{-1}\vec\Delta\vecone)}{2 \sum_{i=1}^n\vec d_i^2}}.
	\end{align}
	Finally, since $\Erw\brk{\vec d^2} < \infty$, $\sum_{i=1} \vec d_i^2 \sim n \Erw\brk{\vec d^2}$ and the claim follows. 
\end{proof}


\begin{thebibliography}{99}
\bibitem{AMc} D.\ Achlioptas, F.\ McSherry: Fast computation of low-rank matrix approximations. Journal of the ACM {\bf 54} (2007) \#9.
\bibitem{AchlioptasMolloy} D.\ Achlioptas, M.\ Molloy: The solution space geometry of random linear equations.  Random Structures and Algorithms {\bf46} (2015) 197--231.
\bibitem{AM} D.\ Achlioptas, C.\ Moore: Random $k$-SAT: two moments suffice to cross a sharp threshold. SIAM Journal on Computing {\bf 36} (2006) 740--762.
\bibitem{ANP} D.\ Achlioptas, A.\ Naor, Y.\ Peres: Rigorous location of phase transitions in hard optimization problems. Nature {\bf 435} 759--764.
\bibitem{Aizenman} M.\ Aizenman, R.\ Sims, S.\ Starr: An extended variational principle for the SK spin-glass model.
Phys.\ Rev.\ B {\bf68}  (2003) 214403.
\bibitem{Ayre} P.\ Ayre, A.\ Coja-Oghlan, P.\ Gao, N.\ M\"uller: The satisfiability threshold for random linear equations.  Combinatorica {\bf40} (2020) 179--235.
\bibitem{Balakin2} G.\ Balakin: The distribution of random matrices over a finite field.  Theory Probab.\ Appl.\ {\bf13} (1968) 631--641.
\bibitem{BKW} J.\ Bl\"omer, R.\ Karp, E.\ Welzl: The rank of sparse random matrices over finite fields.
Random Structures and Algorithms {\bf 10} (1997) 407--419.
\bibitem{BLS} C.\ Bordenave, M.\ Lelarge, J.\ Salez: The rank of diluted random graphs. Ann.\ Probab.\ {\bf 39} (2011) 1097--1121.
\bibitem{Maurice} A.~Coja-Oghlan, A.~Erg\"ur, P.~Gao, S.~Hetterich, M.~Rolvien: The rank of sparse random matrices. Proc.\ 31st SODA (2020) 579--591.
\bibitem{CKPZ} A.\ Coja-Oghlan, F.\ Krzakala, W.\ Perkins,  L.\ Zdeborova: Information-theoretic thresholds from the cavity method.  Advances in Mathematics {\bf 333} (2018) 694--795.
\bibitem{Kosta} A.\ Coja-Oghlan, K.\ Panagiotou: The asymptotic $k$-SAT threshold. Advances in Mathematics {\bf 288} (2016) 985--1068.
\bibitem{KostaNAE} A.\ Coja-Oghlan, K.\ Panagiotou: Catching the $k$-NAESAT threshold. Proc.~44th STOC (2012) 899--908.
\bibitem{Jean} A. Coja-Oghlan, N.\ M\"uller, J.\ Ravelomanana: Belief Propagation on the random $k$-SAT model. Annals of Applied Probability, in press.
\bibitem{CFP} C.\ Cooper, A.\ Frieze, W.\ Pegden: On the rank of a random binary matrix.  Electron.\ J.\ Comb.\ {\bf26} (2019) P4.12.
\bibitem{costello2008rank} K.~Costello, V.~Vu: The rank of random graphs.  Random Structures and Algorithms {\bf 33} (2008) 269--285.
\bibitem{costello2010rank} K.~Costello, V.~Vu:  On the rank of random sparse matrices.  Combinatorics, Probability and Computing {\bf 19} (2010)  321--342.
\bibitem{DavMc} B.\ Davis, D.\ McDonald: An elementary proof of the local limit theorem.  Journal of Theoretical Probability {\bf8} (1995) 693--701.
\bibitem{Dietzfelbinger} M.\ Dietzfelbinger, A.\ Goerdt, M.\ Mitzenmacher, A.\ Montanari, R.\ Pagh, M.\ Rink: Tight thresholds for cuckoo hashing via XORSAT.  Proc.\ 37th ICALP (2010) 213--225.
\bibitem{DSS3} J.~Ding, A.~Sly, N.~Sun: Proof of the satisfiability conjecture for large $k$.  Proc.\ 47th STOC (2015) 59--68.
\bibitem{DuboisMandler} O.\ Dubois, J.\ Mandler: The 3-XORSAT threshold. Proc.\  43rd FOCS (2002) 769--778.
\bibitem{Ferber} A.\ Ferber, M.\ Kwan, A.\ Sah, M.\ Sawhney: Singularity of the k-core of a random graph. arXiv:2106.05719 
\bibitem{GoerdtFalke} A.\ Goerdt, L.\ Falke: Satisfiability thresholds beyond $k$-XORSAT.  Proc.\ 7th International Computer Science Symposium in Russia (2012) 148--159.
\bibitem{Huang} J.\ Huang: Invertibility of adjacency matrices for random d-regular graphs.		arXiv:1807.06465.
\bibitem{Ibrahimi} M.\ Ibrahimi, Y.\ Kanoria, M.\ Kraning, A.\ Montanari: The set of solutions of random XORSAT formulae.
Annals of Applied Probability {\bf 25} (2015) 2743--2808.
\bibitem{KKS} J.\ Kahn; J.\ Komlo\'s, E.\ Szemer\'edi: On the probability that a random $\pm1$-matrix is singular.  Journal of the AMS {\bf 8} (1995) 223--240.
\bibitem{Kolchin} V.~Kolchin: Random graphs and systems of linear equations in finite fields.  Random Structures and Algorithms {\bf 5} (1995) 425--436.
\bibitem{Kolchin1} V.\ Kolchin, V.\ Khokhlov: On the number of cycles in a random non-equiprobable graph. Discrete Math.\ Appl.\ {\bf2} (1992) 109--118.
\bibitem{Kolchin2} V.\ Kolchin: Consistency of a system of random congruences.  Discrete Math.\ Appl.\ {\bf3} (1993) 103--113.
\bibitem{Komlos} J.\ Koml\'os: On the determinant of (0,1) matrices. Studia Sci.\ Math.\ Hungar.\ {\bf2} (1967) 7--21.
\bibitem{Kovalenko} I.~Kovalenko: On the limit distribution of the number of solutions of a random system of linear equations in the class of Boolean functions.  Theory Probab.\ Appl.\ {\bf 12} (1967) 51--61.
\bibitem{Kovalenko2} I.~Kovalenko, A.\ Levitskaya, M.\ Savchuk: Selected problems of probabilistic combinatorics.  Naukova Dumka,  Kiev (1986).
\bibitem{Lenstra} H.~Lenstra: Lattices.  In J.~Buhler, P.~Stevenhagen (eds.): Algorithmic number theory: lattices, number fields, curves and cryptography.  Cambridge University Press (2008) 127--181.
\bibitem{Lev1} A.~Levitskaya: Theorems on invariance for the systems of random linear equations over an arbitrary finite ring. Soviet Math.\ Dokl.\ {\bf263} (1982) 289--291.
\bibitem{Lev2} A.~Levitskaya: The probability of consistency of a system of random linear equations over a finite ring. Theory Probab.\ Appl.\ {\bf30} (1985) 339--350.
\bibitem{MM} M.~M\'ezard, A.~Montanari: Information, physics and computation.  Oxford University Press~2009.
\bibitem{MRTZ} M.\ M\'ezard, F.\ Ricci-Tersenghi, R.\ Zecchina: Two solutions to diluted $p$-spin models and XORSAT problems.  Journal of Statistical Physics {\bf 111} (2003) 505--533.
\bibitem{MillerCohen} G.\ Miller,  G.\ Cohen: The rate of regular LDPC codes.  IEEE Transactions on Information Theory {\bf49} (2003) 2989--2992.
\bibitem{Montanari} A.\ Montanari: Estimating random variables from random sparse observations. European Transactions on Telecommunications {\bf19}(4) (2008) 385--403.
\bibitem{PittelSorkin} B.\ Pittel, G.\ Sorkin:  The satisfiability threshold for $k$-XORSAT.  \CPC\ {\bf25} (2016) 236--268.
\bibitem{Raic} M.~Rai\v c: A multivariate Berry–Esseen theorem with explicit constants. Bernoulli {\bf25} (2019) 2824--2853.
\bibitem{Raghavendra} P.\ Raghavendra, N.\ Tan: Approximating CSPs with global cardinality constraints using SDP hierarchies. Proc.\ 23rd SODA (2012) 373--387.
\bibitem{RichardsonUrbanke} T.\ Richardson, R.\ Urbanke: Modern coding theory. Cambridge University Press (2008).
\bibitem{TaoVu} T.\ Tao, V.\ Vu: On the singularity probability of random Bernoulli matrices.  Journal of the AMS {\bf20} (2007) 603--628.
\bibitem{Tikhomirov} K.~Tikhomirov: Singularity of random Bernoulli matrices. Annals of Mathematics {\bf 191} (2020) 593--634.
\bibitem{Vu} V.\ Vu: Combinatorial problems in random matrix theory.  Proc.\ International Congress of Mathematicians IV (2014) 489--508.
\bibitem{Vu2} V.\ Vu: Recent progress in combinatorial random matrix theory. arXiv:2005.02797. 
\bibitem{Maneva} M.\ Wainwright, E.\ Maneva, E.\ Martinian: Lossy source compression using low-density generator matrix codes: analysis and algorithms. IEEE Trans.\ Inf.\ Theory {\bf56} (2010) 1351--1368.
\end{thebibliography}
\end{document}